\pdfoutput=1
\RequirePackage{ifpdf}
\ifpdf % We are running pdfTeX in pdf mode
\documentclass[pdftex]{sigma}
\else
\documentclass{sigma}
\fi

\numberwithin{equation}{section}

\newtheorem{Theorem}{Theorem}[section]
\newtheorem{Corollary}[Theorem]{Corollary}
\newtheorem{Lemma}[Theorem]{Lemma}
\newtheorem{Proposition}[Theorem]{Proposition}
 { \theoremstyle{definition}
\newtheorem{Definition}[Theorem]{Definition}
\newtheorem{Example}[Theorem]{Example}
 }

\begin{document}
\allowdisplaybreaks

\newcommand{\arXivNumber}{1801.07041}

\renewcommand{\thefootnote}{}

\renewcommand{\PaperNumber}{077}

\FirstPageHeading

\ShortArticleName{Connection Formula for the Jackson Integral of Type $A_n$}

\ArticleName{Connection Formula for the Jackson Integral\\ of Type $\boldsymbol{A_n}$ and Elliptic Lagrange Interpolation\footnote{This paper is a~contribution to the Special Issue on Elliptic Hypergeometric Functions and Their Applications. The full collection is available at \href{https://www.emis.de/journals/SIGMA/EHF2017.html}{https://www.emis.de/journals/SIGMA/EHF2017.html}}}

\Author{Masahiko ITO~$^\dag$ and Masatoshi NOUMI~$^\ddag$}

\AuthorNameForHeading{M.~Ito and M.~Noumi}

\Address{$^\dag$~Department of Mathematical Sciences, University of the Ryukyus, Okinawa 903-0213, Japan}
\EmailD{\href{mailto:mito@sci.u-ryukyu.ac.jp}{mito@sci.u-ryukyu.ac.jp}}

\Address{$^\ddag$~Department of Mathematics, Kobe University, Rokko, Kobe 657-8501, Japan}
\EmailD{\href{noumi@math.kobe-u.ac.jp}{noumi@math.kobe-u.ac.jp}}

\ArticleDates{Received January 23, 2018, in final form July 07, 2018; Published online July 24, 2018}

\Abstract{We investigate the connection problem for the Jackson integral of type $A_n$. Our connection formula implies a Slater type expansion of a bilateral multiple basic hypergeometric series as a linear combination of several specific multiple series. Introducing certain elliptic Lagrange interpolation functions, we determine the explicit form of the connection coefficients. We also use basic properties of the interpolation functions to establish an explicit determinant formula for a fundamental solution matrix of the associated system of $q$-difference equations.}


\Classification{33D52; 39A13}

\renewcommand{\thefootnote}{\arabic{footnote}}
\setcounter{footnote}{0}

\section{Introduction}

Throughout this paper we fix the base $q\in\mathbb{C}$ with $0<|q|<1$, and use the notation of $q$-shifted factorials
\begin{gather*}
(u)_\infty=\prod^\infty_{l=0}\big(1-uq^l\big),\qquad (a_1,\dots,a_r)_\infty=(a_1)_\infty\cdots(a_r)_\infty,\\
(u)_\nu = (u)_\infty/\big(uq^\nu\big)_\infty,\qquad (a_1,\dots,a_r)_\nu=(a_1)_\nu\cdots(a_r)_\nu.
\end{gather*}
For generic complex parameters $a_1,\ldots,a_r$ and $b_1,\ldots,b_r$, the bilateral basic hypergeometric series $_r\psi_r$ \cite[equation~(5.1.1)]{GR2004} is defined by
\begin{gather*} %\label{eq:6.1}
 {}_r\psi_r \left[
\begin{matrix}
 a_1,\dots,a_r \\
 b_1,\dots,b_r
\end{matrix}; q,x
 \right]
= \sum_{\nu=-\infty}^\infty\frac{(a_1,\dots,a_r)_\nu}{(b_1,\dots,b_r)_\nu}x^\nu,\qquad \text{where}\quad |b_1\cdots b_r/a_1\cdots a_r|<|x|<1.
\end{gather*}
An important summation for $_1\psi_1$ series is Ramanujan's theorem \cite[equation~(5.2.1)]{GR2004},
\begin{gather}\label{eq:1psi1}
 {}_1\psi_1 \left[ \begin{matrix}
 a\\
 b
\end{matrix}; q,x
 \right]=\frac{(ax,q,b/a,q/ax)_\infty}{(x,b,q/a,b/ax)_\infty}.
\end{gather}
There is also a summation formula for a very-well-poised $_6\psi_6$ series due to Bailey \cite[equation~(5.3.1)]{GR2004},
\begin{gather}\label{eq:66}{}_6\psi_6
\left[ \begin{matrix}
 q\sqrt{a},-q\sqrt{a},\, b,\, c,\, d,\, e\, \\
 \sqrt{a}, \, -\sqrt{a},\frac{aq}{b}, \frac{aq}{c}, \frac{aq}{d}, \frac{aq}{e}\end{matrix} ; q, \frac{a^2q}{bcde} \right]=
\frac{\big(aq,\frac{aq}{bc},\frac{aq}{bd},\frac{aq}{be},\frac{aq}{cd},\frac{aq}{ce},\frac{aq}{de},q,\frac{q}{a}\big)_\infty}
{\big(\frac{aq}{b},\frac{aq}{c},\frac{aq}{d},\frac{aq}{e},\frac{q}{b},\frac{q}{c},\frac{q}{d},\frac{q}{e},\frac{a^2q}{bcde}\big)_\infty},
\end{gather}
where $|a^2q/bcde|<1$. Furthermore there are a lot of transformation formulas for the ${}_r\psi_r$ series. In the paper~\cite{Sl52} Slater proved the transformation formula
\begin{gather}
 {}_r\psi_r \left[\begin{matrix} a_1,\dots,a_r \\
 b_1,\dots,b_r \end{matrix}; q,z \right] =\frac
{\big(\frac{c_{1}}{a_{1}},\ldots,\frac{c_{1}}{a_{r}},
c_{2},\ldots,c_{r},
\frac{q}{c_{2}},\ldots,\frac{q}{c_{r}},
\frac{b_{1}q}{c_{1}},\ldots,\frac{b_{r}q}{c_{1}},
Ac_{1}z,\frac{q}{Ac_{1}z} \big)_{\infty }}
{\big(\frac{q}{a_{1}},\ldots,\frac{q}{a_{r}},
\frac{c_{1}}{c_{2}},\ldots,\frac{c_{1}}{c_{r}},
\frac{c_{2}q}{c_{1}},\ldots,\frac{c_{r}q}{c_{1}},
b_{1},\ldots,b_{r}, Azq,\frac{1}{Az}\big)_{\infty }}\nonumber\\
\hphantom{{}_r\psi_r \left[\begin{matrix} a_1,\dots,a_r \\
 b_1,\dots,b_r \end{matrix}; q,z \right] =}{}\times\, {}_{r}\psi _{r}\left[
\begin{matrix}	\frac{a_{1}q}{c_{1}},\ldots,\frac{a_{r}q}{c_{1}}\vspace{1mm}\\
		\frac{b_{1}q}{c_{1}},\ldots,\frac{b_{r}q}{c_{1}}\end{matrix};q,z\right] + {\rm idem} (c_{1};c_{2},\ldots,c_{r}),\label{eq:Slater}
\end{gather}
where $A=a_{1}\cdots a_{r}/c_{1}\cdots c_{r}$ and $|b_1\cdots b_r/a_1\cdots a_r|<|z|<1$, which is called the {\it Slater's general transformation for a~${}_{r}\psi _{r}$ series} \cite[equation~(5.4.3)]{GR2004}. Here the symbol ``idem$(c_{1};c_{2},\ldots,c_{r})$'' stands for the sum of the $r-1$ expressions obtained from the preceding one by interchan\-ging~$c_{1}$ with each $c_{k}$, $k=2,\ldots,r$. Slater also proved a~transformation formula for a very-well-poised~${}_{2r}\psi_{2r}$ series, which is expressed as
\begin{gather}
{}_{2r} \psi _{2r}\left[
\begin{matrix}
 q\sqrt{a}, -q\sqrt{a}, b_{3}, \ldots, b_{2r} \\
 \sqrt{a}, -\sqrt{a}, \frac{aq}{b_{3}}, \ldots, \frac{aq}{b_{2r}}\end{matrix}; q,\frac{a^{r-1}q^{r-2}}{b_{3}\ldots b_{2r}}\right]\nonumber\\
\quad{} =\frac{	\big(a_{4},\ldots,a_{r},\frac{q}{a_{4}},\ldots,\frac{q}{a_{r}},
	\frac{a_{4}}{a},\ldots,\frac{a_{r}}{a},\frac{aq}{a_{4}},\ldots,\frac{aq}{a_{r}},
	\frac{a_{3}q}{b_{3}},\ldots,\frac{a_{3}q}{b_{2r}},\frac{aq}{a_{3}b_{3}},\ldots,\frac{aq}{a_{3}b_{2r}}\big)_{\infty}	}
	{\big(\frac{q}{b_{3}},\ldots,\frac{q}{b_{2r}},\frac{aq}{b_{3}},\ldots,\frac{aq}{b_{2r}},
	\frac{a_{4}}{a_{3}},\ldots,\frac{a_{r}}{a_{3}},\frac{a_{3}q}{a_{4}},\ldots,\frac{a_{3}q}{a_{r}},
	\frac{a_{3}a_{4}}{a},\ldots,\frac{a_{3}a_{r}}{a},\frac{aq}{a_{3}a_{4}},\ldots,\frac{aq}{a_{3}a_{r}}\big)_{\infty}}\nonumber\\
\quad{} \times
\frac{\big(aq,\frac{q}{a}\big)_{\infty}}{\big(\frac{a_{3}^2q}{a},\frac{aq}{a_{3}^2}\big)_{\infty}} \,{}_{2r}\psi _{2r}\left[
\begin{matrix}
\frac{qa_{3}}{\sqrt{a}},-\frac{qa_{3}}{\sqrt{a}},\frac{a_{3}b_{3}}{a},\ldots,\frac{a_{3}b_{2r}}{a}\\
\frac{a_{3}}{\sqrt{a}},-\frac{a_{3}}{\sqrt{a}},\frac{a_{3}q}{b_{3}},\ldots,\frac{a_{3}q}{b_{2r}}
		\end{matrix};q,	\frac{a^{r-1}q^{r-2}}{b_{3}\ldots b_{2r}}\right] + {\rm idem} (a_{3};a_{4},\ldots,a_{r}),\!\!\!\!\label{eq:2r2r}
\end{gather}
where $|a^{r-1}q^{r-2}/b_3\cdots b_{2r}|<1$, which is called {\it Slater's transformation for a very-well-poised balanced ${}_{2r}\psi_{2r}$ series} \cite[equation~(5.5.2)]{GR2004}.

The aim of this paper is to give an explanation and a generalization of Ramanujan's $_1\psi_1$ summation \eqref{eq:1psi1} and Slater's $_r\psi_r$ transformation (\ref{eq:Slater}) in relation to the Selberg integral~\cite{Se44} (see also the recent sources \cite{Fo10, FW08}), i.e.,
\begin{gather}
\frac{1}{n!}\int_{0}^{1}\cdots\int_{0}^{1} \prod_{i=1}^n z_i^{\alpha-1}(1-z_i)^{\beta-1}\prod_{1\le j<k\le n}|z_j-z_k|^{2\tau}\,dz_1dz_2\cdots dz_n \nonumber\\
\qquad {} =\prod_{j=1}^n\frac{\Gamma(\alpha+(j-1)\tau)\Gamma(\beta+(j-1)\tau)\Gamma(j\tau)}{\Gamma(\alpha+\beta+(n+j-2)\tau)\Gamma(\tau)}, \label{eq:Selberg}
\end{gather}
where ${\rm Re}(\alpha)>0$, ${\rm Re}(\beta)>0$ and ${\rm Re}(\gamma)>-\min\{1/n, {\rm Re}(\alpha)/(n-1), {\rm Re}(\beta)/(n-1)\}$, which is a multi-dimensional generalization of the evaluation of the Euler beta integral in terms of products of gamma functions. (This is not the first such paper using Selberg integrals to obtain multiple analogues of some of Slater's transformation formulas. For example, the paper~\cite{IN16} about a~generalization for Bailey's very-well-poised~$_6\psi_6$ summation \eqref{eq:66} and Slater's very-well-poised $_{2r}\psi_{2r}$ transformation \eqref{eq:2r2r} with $BC_n$ symmetry, is very similar in spirit and methodology to the current paper.) For this purpose we define some terminology about a $q$-analog of the Selberg type integral. For $z=(z_1,\ldots,z_n)\in (\mathbb{C}^*)^n$, let $\Phi(z)$ be specified by
\begin{gather}
\Phi(z)= \Phi_s(z;\alpha,a,b;q,t)\nonumber\\
\hphantom{\Phi(z)}{} = (z_1z_2\cdots z_n)^\alpha \prod_{i=1}^n\prod_{m=1}^s\frac{\big(qa_m^{-1}z_i\big)_\infty}{(b_m z_i)_\infty}
\prod_{1\le j<k\le n} z_j^{2\tau-1}\frac{\big(qt^{-1}z_k/z_j\big)_\infty}{(t z_k/z_j)_\infty},\label{eq:Phi}
\end{gather}
with $\alpha\in \mathbb{C}$, $a=(a_1,\ldots,a_s)$, $b=(b_1,\ldots,b_s)\in(\mathbb{C}^*)^s$ and $t\in \mathbb{C}^*$, where $\tau\in \mathbb{C}$ is given by $t=q^{\tau}$. For $z=(z_1,\ldots,z_n)\in (\mathbb{C}^*)^n$, let $\Delta(z)$ be the Vandermonde product
\begin{gather}
\Delta(z)=\prod_{1\le j<k\le n}(z_j-z_k). \label{eq:Delta}
\end{gather}
For a function $\varphi=\varphi(z)$ of $z=(z_1,\ldots,z_n)\in(\mathbb{C}^*)^n$, we denote by
\begin{gather}
\langle\varphi, z\rangle= \int_0^{z\infty} \varphi(w)\Phi(w)\Delta(w) \frac{d_qw_1}{w_1}\wedge\cdots\wedge \frac{d_qw_n}{w_n} \nonumber\\
\hphantom{\langle\varphi, z\rangle}{} =(1-q)^n\sum_{\nu\in \mathbb{Z}^n}\varphi(z q^\nu)\Phi(z q^\nu)\Delta(z q^\nu),\label{eq:la-ra}
\end{gather}
the Jackson integral associated with the multiplicative lattice $z q^\nu=(z_1q^{\nu_1},\ldots,z_nq^{\nu_n})\in(\mathbb{C}^*)^n$, $\nu=(\nu_1,\ldots,\nu_n)\in \mathbb{Z}^n$. Note that the sum $\langle\varphi, z\rangle$ is invariant under the shifts $z\to zq^{\nu}$ for all $\nu\in \mathbb{Z}^n$ when it converges. This sum is called the {\em Jackson integral of $A$-type} in the context of~\cite{Ao98}.
For the general setting of Jackson integrals see~\cite{AK91}.

We remark that for $s=1$, $\varphi\equiv 1$ and $z=\big(a_1,a_1 t,\ldots,a_1 t^{n-1}\big)\in (\mathbb{C}^*)^n$, the Jackson integ\-ral~\eqref{eq:la-ra} is expressed as
\begin{gather}
\big\langle 1,\big(a_1,a_1 t,\ldots,a_1 t^{n-1}\big) \big\rangle\nonumber\\
\qquad {}= (1-q)^n\prod_{j=1}^n \big(a_1t^{j-1}\big)^{\alpha+2(n-j)\tau} \frac{(q)_\infty (t)_\infty \big(q^{\alpha}a_1b_1t^{n+j-2}\big)_\infty} {(t^j)_\infty\big(q^{\alpha}t^{j-1}\big)_\infty\big(a_1b_1t^{j-1}\big)_\infty},\label{eq:q-Selberg}
\end{gather}
whose limiting case where $q\to 1$ is equivalent to the Selberg integral \eqref{eq:Selberg}. In this sense the Jackson integral \eqref{eq:la-ra} includes the Selberg integral as a special case. The formula \eqref{eq:q-Selberg} was discovered by Askey \cite{As80} and proved by Habsieger~\cite{Ha88}, Kadell~\cite{Kad88-2}, Evans~\cite{Ev92} and Kaneko~\cite{Kan96} in the case where $\tau$ is a positive integer, while~(\ref{eq:q-Selberg}) for general complex $\tau$ was given by Aomoto~\cite{Ao98}. See~\cite{IF15} for further details. Other closely related and relevant works are \cite{Kad88-1,Kad01,W05}.

We denote by $\mathfrak{S}_n$ the symmetric group of degree $n$; this group acts on the field of meromorphic functions on~$({\mathbb C}^*)^n$ through the permutations of variables $z_1,\ldots,z_n$. Setting
\begin{gather}\label{eq:lla-rra}
\langle\!\langle\varphi, z\rangle\!\rangle=\frac{\langle\varphi, z\rangle}{\Theta(z)},\qquad \Theta(z)=\prod_{i=1}^n \frac{z_i^\alpha}{\prod\limits_{m=1}^{s}\theta(b_mz_i)} \prod_{1\le j<k\le n} \frac{z_j^{2\tau}\theta(z_j/z_k)}{\theta(tz_j/z_k)},
\end{gather}
where $\theta(u)=(u)_\infty \big(qu^{-1}\big)_\infty$, we have the following.

\begin{Lemma}\label{lem:sym-hol} Let $\varphi(z)$ be an $\frak{S}_n$-invariant holomorphic function on $(\mathbb{C}^*)^n$, and suppose that the Jackson integral $\langle\varphi, z\rangle$ of~\eqref{eq:la-ra} converges as a meromorphic function on $(\mathbb{C}^*)^n$. Then the function $f(z)=\langle\!\langle\varphi, z\rangle\!\rangle$ defined by \eqref{eq:lla-rra} is an $\frak{S}_n$-invariant holomorphic function on $(\mathbb{C}^*)^n$. Furthermore it satisfies the quasi-periodicity $T_{q,z_i}f(z)=f(z)/(-z_i)^s q^\alpha t^{n-1}b_1\cdots b_s$ for $i=1,\ldots,n$, where $T_{q,z_i}$ stands for the $q$-shift operator in $z_i$.
\end{Lemma}

For the proof of this lemma, see Lemma \ref{lem:lla-rra}.

We call $\langle\!\langle\varphi, z\rangle\!\rangle$ the {\it regularized Jackson integral of $A$-type}, which is the main object of this paper. We remark that, when $\varphi(z)$ is a symmetric polynomial such that $\deg_{z_i}\varphi(z) \le s-1$, $i=1,\ldots,n$, it satisfies the condition of Lemma~\ref{lem:sym-hol} under the assumption~\eqref{eq:condition:converge01} below.

In order to state our first main theorem we define some terminology. We set
\begin{gather}\label{eq:Zsn}
Z=Z_{s,n}=\big\{\mu=(\mu_1,\mu_2,\ldots,\mu_s)\in \mathbb{N}^s\,|\,\mu_1+\mu_2+\cdots+\mu_s=n\big\},
\end{gather}
where $\mathbb{N}=\{0,1,2,\ldots\}$, so that $|Z_{s,n}|={s+n-1\choose n}$. We use the symbol $\preceq$ for the lexicographic order on $Z_{s,n}$.
Namely, for $\mu,\nu\in Z_{s,n}$, we denote $\mu\prec\nu$ if there exists $k\in\{1,2,\ldots,n\}$ such that $\mu_1=\nu_1,\mu_2=\nu_2,\ldots,\mu_{k-1}=\nu_{k-1}$ and $\mu_k<\nu_k$. For an arbitrary $x=(x_1,x_2,\ldots,x_s)\in (\mathbb{C}^*)^s$, we consider the points $x_\mu$ $(\mu \in Z_{s,n})$ in $(\mathbb{C}^*)^n$ specified by
\begin{gather}\label{eq:x_mu}
x_\mu=\big(\underbrace{x_1,x_1t,\ldots,x_1t^{\mu_1-1}\vphantom{\Big|}}_{\mu_1},\underbrace{x_2,x_2t,\ldots,x_2t^{\mu_2-1}\vphantom{\Big|}}_{\mu_2},\ldots, \underbrace{x_s,x_st,\ldots,x_st^{\mu_s-1}\vphantom{\Big|}}_{\mu_s}\big)
\in (\mathbb{C}^*)^n.
\end{gather}

\begin{Theorem}[connection formula] \label{thm:connection} Suppose that $\varphi(z)$ is an $\frak{S}_n$-invariant holomorphic function satisfying the condition of Lemma~{\rm \ref{lem:sym-hol}}. Then, for generic $x\in (\mathbb{C}^*)^s$ we have
\begin{gather}\label{eq:connection}
\langle\!\langle \varphi, z\rangle\!\rangle=\sum_{\lambda\in Z_{s,n}}c_\lambda\,\langle\!\langle\varphi, x_\lambda\rangle\!\rangle,
\end{gather}
where the connection coefficients $c_\lambda$ are explicitly written as
\begin{gather}
c_\lambda=\sum_{\substack{K_1\sqcup\cdots\sqcup K_s \\ =\{1,2,\ldots,n\}}} \prod_{i=1}^s\prod_{k\in K_i}\Bigg[\frac{\theta\Big(q^\alpha b_1\cdots b_s t^{n-1}z_k\prod\limits_{1\le l\le s\atop l\ne i}x_lt^{\lambda_l^{(k-1)}}\Big)}{\theta\Big(q^\alpha b_1\cdots b_s t^{n-1}\prod\limits_{l=1}^s x_lt^{\lambda_l^{(k-1)}}\Big)}\nonumber\\
\hphantom{c_\lambda=\sum_{\substack{K_1\sqcup\cdots\sqcup K_s \\ =\{1,2,\ldots,n\}}} \prod_{i=1}^s\prod_{k\in K_i}}{}\times
\prod_{\substack{1\le j\le s\\ j\ne i}}\frac{\theta\big(z_{k}\,x_j^{-1}t^{-\lambda_j^{(k-1)}}\big)}{\theta\big(x_it^{\lambda_i^{(k-1)}}x_j^{-1}t^{-\lambda_j^{(k-1)}}\big)}\Bigg],\label{eq:E-explicit-0}
\end{gather}
where $\lambda_i^{(k)}=|K_i\cap\{1,2,\ldots,k\}|$, and the summation is taken over all partitions $K_1\sqcup\cdots\sqcup K_s=\{1,2,\ldots,n\}$ such that $|K_i|=\lambda_i$, $i=1,2,\ldots,s$.
\end{Theorem}

We call the formula (\ref{eq:connection}) the {\it generalized Slater transformation}. In fact, the connection formula \eqref{eq:connection} for $n=1$ coincides with the transformation formula (\ref{eq:Slater}) if $\varphi(z)\equiv1$, as explained below.

\begin{Example}\label{ex:connection1}
The case $n=1$. Setting $\epsilon_i=(0,\ldots,0,\overset{\text{\tiny$i$}}{\stackrel{\text{\tiny$\smile$}}{1}},0,\ldots,0)$ for $i=1,2,\ldots,s$, we have $Z_{s,1}=\{\epsilon_1,\epsilon_2,\ldots,\epsilon_s\}$ and $x_{\epsilon_i}=x_i\in \mathbb{C}^*$, i.e., we obtain
\begin{gather}\label{eq:connection n=1}
\langle\!\langle \varphi,z\rangle\!\rangle=\sum_{i=1}^s c_{\epsilon_i}\langle\!\langle \varphi,x_i\rangle\!\rangle,\qquad \text{where}\quad
c_{\epsilon_i}=\frac{\theta(q^{\alpha}b_1\cdots b_s x_1\cdots x_s z/x_i)} {\theta(q^{\alpha}b_1\cdots b_s x_1\cdots x_s)}\prod_{\substack{1\le j\le s\\ j\ne i}}\frac{\theta(z/x_j)}{\theta(x_i/x_j)}.\!\!\!\!
\end{gather}
The formula \eqref{eq:connection n=1} for $\varphi\equiv 1$ (which was first stated by Mimachi \cite[theorem in Section~4]{Mim89}) coincides with Slater's transformation in~(\ref{eq:Slater}). For $z=a_i$, $i=1,\ldots, s$, and $x_j=b_j^{-1}$, $j=1,\ldots, s$, it was given by Aomoto and Kato \cite[equation~(2.12) in Corollary~2.1]{AK94-1}. For further details about the formula~(\ref{eq:connection n=1}), see also \cite[p.~255, Theorem~5.1 in the Appendix]{IS08}.
\end{Example}

\begin{Example} The case $s=1$. We have $Z_{1,n}=\{(n)\}$, which has only one element, and have
$x_{(n)}=(x,x t,\ldots, x t^{n-1})$ for $x\in \mathbb{C}^*$. Then
\begin{gather}\label{eq:connection s=1}
\langle\!\langle \varphi,z\rangle\!\rangle=c_{(n)}\langle\!\langle \varphi,x_{(n)}\rangle\!\rangle,
\qquad\mbox{where}\quad c_{(n)}=\prod_{i=1}^n \frac{\theta\big(q^{\alpha}b_1t^{n-1}z_i\big)}{\theta\big(q^{\alpha}b_1t^{n-1}xt^{i-1}\big)}.
\end{gather}
The formula (\ref{eq:connection s=1}) for $\varphi\equiv 1$ was given by Aomoto in~\cite{Ao98}. See also \cite[Lemma~3.6]{IF15}.
\end{Example}

As we will see below the connection coefficients $c_\mu$ in (\ref{eq:E-explicit-0}) are characterized, as functions of $z\in (\mathbb{C}^*)^n$,
by an interpolation property.

We next apply the connection formula (\ref{eq:connection}) in Theorem \ref{thm:connection} to establish a determinant formula associated with the Jackson integral~\eqref{eq:la-ra}. Let $B_{s,n}$ be the set of partitions defined by
\begin{gather}\label{eq:Bsn}
B=B_{s,n}=\big\{\lambda=(\lambda_1,\lambda_2,\ldots, \lambda_n)\in \mathbb{Z}^n\,|\,s-1\ge\lambda_1\ge\lambda_2\ge \cdots\ge \lambda_n\ge 0\big\},
\end{gather}
so that $|B_{s,n}|={s+n-1\choose n}$. We also use the symbol $\preceq$ for the lexicographic order of $B_{s,n}$. For each $\lambda\in B_{s,n}$, we denote by $s_\lambda(z)$ the {\em Schur function}
\begin{gather*}
s_\lambda(z)=\frac{\det\big(z_i^{\lambda_j+n-j}\big)_{1\le i,j\le n}} {\det\big(z_i^{n-j}\big)_{1\le i,j\le n}} =\frac{\det\big(z_i^{\lambda_j+n-j}\big)_{1\le i,j\le n}} {\Delta(z)},
\end{gather*}
which is a $\frak{S}_n$-invariant polynomial. Our second main theorem is the following.
\begin{Theorem}[determinant formula]\label{thm:Wronski} Suppose that $x\in (\mathbb{C}^*)^s$ is generic. Then we have
\begin{gather}
\det\big(\langle\!\langle s_\lambda, x_\mu \rangle\!\rangle\big)_{\!\lambda\in B\atop\!\mu\in Z}
 =C \prod_{k=1}^n\big[\theta\big(q^\alpha x_1\cdots x_s b_1\cdots b_s t^{n+k-2}\big)\big]^{s+k-2\choose k-1}\nonumber\\
\hphantom{\det\big(\langle\!\langle s_\lambda, x_\mu \rangle\!\rangle\big)_{\!\lambda\in B\atop\!\mu\in Z}=}{}
\times\prod_{k=1}^n\left[\prod_{r=0}^{n-k}\prod_{1\le i<j\le s}x_jt^r\theta\big(x_ix_j^{-1}t^{n-k-2r}\big)\right]^{s+k-3\choose k-1},\label{eq:Wronski}
\end{gather}
where the rows and the columns of the matrix are arranged by the lexicographic orders of $\lambda\in B_{s,n}$ and $\mu\in Z_{s,n}$, respectively. Here $C$ is a constant independent of $x$, which is explicitly written as
\begin{gather*}
C= \prod_{k=1}^n \left[ \frac{(1-q)^s(q)_\infty^s\big(qt^{-(n-k+1)}\big)_\infty^s \prod\limits_{i,j=1}^{s} \big(qa_i^{-1}b_j^{-1} t^{-(n-k)}\big)_\infty} {\big(qt^{-1}\big)_\infty^s \big(q^\alpha t^{n-k}\big)_\infty\Big(q^{1-\alpha}t^{-(n+k-2)}\prod\limits_{i=1}^{s} a_i^{-1}b_i^{-1}\Big)_\infty}\right]^{s+k-2\choose k-1}.
\end{gather*}
\end{Theorem}

\begin{remark*}
When $s=1$, the determinant formula (\ref{eq:Wronski}), combined with the connection formula $\langle\!\langle 1,z\rangle\!\rangle=c_{(n)} \langle\!\langle 1,x_{(n)}\rangle\!\rangle$ of (\ref{eq:connection s=1}), implies the summation formula
\begin{gather*}%\label{eq:01cal I(x)2}
\langle\!\langle 1,z\rangle\!\rangle= (1-q)^n \prod_{j=1}^{n} \frac{(q)_\infty\big(qt^{-j}\big)_\infty\big(qa_1^{-1}b_1^{-1}t^{-(j-1)}\big)_\infty
}{\big(qt^{-1}\big)_\infty\big(q^{\alpha}t^{j-1}\big)_\infty\big(q^{1-\alpha}a_1^{-1}b_1^{-1}t^{-(n+j-2)}\big)_\infty} \theta\big(q^{\alpha}b_1t^{n-1}z_j\big),
\end{gather*}
for the Jackson integral $\langle\!\langle 1,z\rangle\!\rangle$ for an arbitrary $z\in (\mathbb{C}^*)^n$. This formula for $n=1$ coincides with Ramanujan's formula (\ref{eq:1psi1}). This is another multi-dimensional bilateral extension of Ramanujan's $_1\psi_1$ summation theorem, which is different from the Milne--Gustafson summation theo\-rem~\cite{Gu87,Mil86}. Another class of extension relates to the theory of Macdonald polynomials; see for example \cite{Kan96,Kan98,MS02,W05} as cited in~\cite{W13}.
\end{remark*}

In this paper, we prove Theorems \ref{thm:connection} and \ref{thm:Wronski} from the viewpoint of the elliptic Lagrange interpolation functions of type $A_n$. Let $\mathcal{O}(({\mathbb C}^*)^n)$ be the $\mathbb{C}$-vector space of holomorphic functions on $({\mathbb C}^*)^n$. In view of Lemma~\ref{lem:sym-hol}, fixing a constant $\zeta\in \mathbb{C}^*$ we consider the $\mathbb{C}$-linear subspace $\mathcal{H}_{s,n,\zeta}\subset \mathcal{O}(({\mathbb C}^*)^n$ consisting of all $\frak{S}_n$-invariant holomorphic functions $f(z)$ such that $T_{q,z_i}f(z)=f(z)/(-z_i)^s\zeta$ for $i=1,\ldots,n$, where $T_{q,z_i}$ stands for the $q$-shift operator in~$z_i$:
\begin{gather}\label{eq:def-Hsn}
\mathcal{H}_{s,n,\zeta}=\big\{f(z)\in \mathcal{O}((\mathbb{C}^*)^n)^{\frak{S}_n}\,|\, T_{q,z_i}f(z)=f(z)/(-z_i)^s \zeta \ \text{for} \ i=1,2,\ldots,n\big\}.
\end{gather}
The dimension of $\mathcal{H}_{s,n,\zeta}$ as a $\mathbb{C}$-vector space will be shown to be $n+s-1\choose n$ in Section \ref{Section2}. Moreover,

\begin{Theorem}\label{thm:MainThm-interpolation} For generic $x\in (\mathbb{C}^*)^s$ there exists a unique $\mathbb{C}$-basis $\{E_\mu(x;z)\,|\,\mu\in Z_{s,n}\}$ of the space $\mathcal{H}_{s,n,\zeta}$ such that
\begin{gather}\label{eq:E-delta}
E_\mu(x;x_\nu)=\delta_{\mu\nu} \qquad \text{for} \quad \mu,\nu\in Z_{s,n},
\end{gather}
where $x_\nu\in (\mathbb{C}^*)^n$, $\nu\in Z_{s,n}$, are the reference points specified by~\eqref{eq:x_mu} and $\delta_{\lambda\mu}$ is the Kronecker delta.
\end{Theorem}

This theorem will be proved in the end of Section~\ref{Section2.1}. We call $E_\mu(x;z)$ the {\it elliptic Lagrange interpolation functions of type $A_n$} associated with the set of reference points $x_\nu$, $\nu\in Z_{s,n}$. Note that an arbitrary function $f(z)\in \mathcal{H}_{s,n,\zeta}$ can be written as a linear combination of $E_\mu(x;z)$, $\mu\in Z_{s,n}$, with coefficients~$f(x_\mu)$:
\begin{gather}\label{eq:E-expantion}
f(z)=\sum_{\mu\in Z_{s,n}}f(x_\mu)E_\mu(x;z).
\end{gather}

\begin{Theorem} \label{thm:expressionE} The functions $E_\lambda(x;z)$ are expressed as
\begin{gather}\label{eq:E-explicit-4}
E_\lambda(x;z)=\sum_{\substack{K_1\sqcup\cdots\sqcup K_s \\ =\{1,2,\ldots,n\}}}\prod_{i=1}^s\prod_{k\in K_i}\!\left[\frac{\theta\Big(z_k\zeta \prod\limits_{1\le l\le s\atop l\ne i}x_lt^{\lambda_l^{(k-1)}}\Big)}
{\theta\Big(\zeta\prod\limits_{l=1}^s x_lt^{\lambda_l^{(k-1)}}\Big)}\prod_{\substack{1\le j\le s\\ j\ne i}}\frac{\theta\big(z_{k} x_j^{-1}t^{-\lambda_j^{(k-1)}}\big)}
{\theta\big(x_it^{\lambda_i^{(k-1)}}x_j^{-1}t^{-\lambda_j^{(k-1)}}\big)}\right],\!\!\!
\end{gather}
where $\lambda_i^{(k)}=|K_i\cap\{1,2,\ldots,k\}|$, and the summation is taken over all partitions $K_1\sqcup\cdots\sqcup K_s=\{1,2,\ldots,n\}$ such that $|K_i|=\lambda_i$, $i=1,2,\ldots,s$.
\end{Theorem}

This theorem will be proved in Section \ref{Section2} as Theorem \ref{thm:expressionE2}.

Once Theorems \ref{thm:MainThm-interpolation} and \ref{thm:expressionE} have been established, the connection formula (\ref{eq:connection}) in Theorem~\ref{thm:connection} is immediately obtained. In the setting of Theorem~\ref{thm:connection}, the regularization $\langle\!\langle\varphi,z\rangle\!\rangle=\langle \varphi,z\rangle/\Theta(z)$ belongs to $\mathcal{H}_{s,n,\zeta}$ with $\zeta=q^\alpha t^{n-1}b_1\cdots b_s$ by Lemma~\ref{lem:sym-hol}. Hence, by \eqref{eq:E-expantion} we obtain
\begin{gather}\label{eq:connection2}
\langle\!\langle\varphi, z\rangle\!\rangle=\sum_{\mu\in Z_{s,n}}\langle\!\langle\varphi, x_\mu\rangle\!\rangle E_\mu(x;z).
\end{gather}
This means that the coefficients in (\ref{eq:connection}) are given by $c_\mu=E_\mu(x;z)$. The explicit formula~(\ref{eq:E-explicit-0}) follows from Theorem~\ref{thm:expressionE}.

For the evaluation of the determinant of Theorem \ref{thm:Wronski}, we make use of the following determinant formula for the elliptic Lagrange interpolation functions $E_\mu(x;z)$ with $\zeta=q^\alpha t^{n-1}b_1\cdots b_s$.
\begin{Theorem}\label{thm:detE}Suppose that $x,y\in (\mathbb{C}^*)^s$ are generic. Then,
\begin{gather*}
\det \big(E_\mu (x;y_{\nu})\big)_{\!\mu,\nu\in Z_{s,n}}= \prod_{k=1}^n\left[\frac{\theta\big(q^\alpha y_1\cdots y_s b_1\cdots b_st^{n+k-2}\big)}
{\theta\big(q^\alpha x_1\cdots x_s b_1\cdots b_st^{n+k-2}\big)} \right]^{s+k-2\choose k-1}\nonumber\\
\hphantom{\det \big(E_\mu (x;y_{\nu})\big)_{\!\mu,\nu\in Z_{s,n}}=}{}
\times \prod_{k=1}^n\left[\prod_{r=0}^{n-k}\prod_{1\le i<j\le s}\frac{y_it^r\theta\big(t^{(n-k)-2r}y_jy_i^{-1}\big)} {x_it^r\theta\big(t^{(n-k)-2r}x_jx_i^{-1}\big)}\right]^{{s+k-3\choose k-1}},%\label{eq:detE-2}
\end{gather*}
where $y_{\nu}$ are specified as in \eqref{eq:x_mu}.
\end{Theorem}

This theorem will be proved as Theorem \ref{thm:detE2} in Section~\ref{Section2}.

Applying the connection formula (\ref{eq:connection}) in Theorem \ref{thm:connection}, we see that Theorem~\ref{thm:Wronski} is reduced to the special case where $x=a=(a_1,\ldots,a_s)$. Since
\begin{gather*}
\langle\!\langle \varphi, z\rangle\!\rangle=\sum_{\mu\in Z_{s,n}}\langle\!\langle\varphi, a_\mu\rangle\!\rangle E_\mu(a;z)
\end{gather*}
by (\ref{eq:connection2}), setting $\varphi (z)= s_\lambda(z)$ ($\lambda\in B_{s,n}$) and $z=x_\nu$ ($\nu\in Z_{s,n}$), we have
\begin{gather*}
\langle\!\langle s_\lambda,x_\nu\rangle\!\rangle =\sum_{\mu\in Z_{s,n}} \langle\!\langle s_\lambda,a_\mu\rangle\!\rangle E_\mu(a;x_\nu) \qquad \text{for} \quad \lambda\in B_{s,n},\quad \nu\in Z_{s,n},
\end{gather*}
so that
\begin{gather*}%\label{eq:det=det det}
\det\big(\langle\!\langle s_\lambda,x_\nu\rangle\!\rangle\big)_{\lambda\in B\atop \nu\in Z}=
\det\big(\langle\!\langle s_\lambda,a_\mu\rangle\!\rangle\big)_{\lambda\in B\atop \mu\in Z} \det \big(E_\mu(a;x_\nu)\big)_{\mu\in Z\atop \nu\in Z}.
\end{gather*}
Since we already know the explicit form of $\det \big(E_\mu(a;x_\nu)\big)_{\mu\in Z,\nu\in Z}$ by Theorem~\ref{thm:detE}, the evaluation of
$\det \big(\langle\!\langle s_\lambda,x_\nu\rangle\!\rangle\big)_{\lambda\in B, \nu\in Z}$ reduces to that of $\det\big(\langle\!\langle s_\lambda,a_\mu\rangle\!\rangle\big)_{\lambda\in B,\mu\in Z}$.
\begin{Lemma}\label{lem:Wronski-a} We have
\begin{gather}
\det\big(\langle\!\langle s_\lambda, a_\mu \rangle\!\rangle\big)_{\!\lambda\in B\atop\!\mu\in Z}=
\prod_{k=1}^n\Bigg[(1-q)^s\frac{(q)_\infty^s\big(qt^{-(n-k+1)}\big)_\infty^s}{\big(qt^{-1}\big)_\infty^s}\nonumber\\
\hphantom{\det\big(\langle\!\langle s_\lambda, a_\mu \rangle\!\rangle\big)_{\!\lambda\in B\atop\!\mu\in Z}=}{} \times
\frac{\Big(q^{\alpha}t^{n+k-2}\prod\limits_{i=1}^{s} a_ib_i\Big)_\infty\prod\limits_{i=1}^{s}\prod\limits_{j=1}^{s}\big(qa_i^{-1}b_j^{-1} t^{-(n-k)}\big)_\infty}{\big(q^\alpha t^{n-k}\big)_\infty}\Bigg]^{s+k-2\choose k-1}\nonumber\\
\hphantom{\det\big(\langle\!\langle s_\lambda, a_\mu \rangle\!\rangle\big)_{\!\lambda\in B\atop\!\mu\in Z}=}{} \times
\prod_{k=1}^n\left[\prod_{r=0}^{n-k}\prod_{1\le i<j\le s}a_jt^r\theta\big(a_ia_j^{-1}t^{n-k-2r}\big)\right]^{s+k-3\choose k-1}.\label{eq:Wronski-a1}
\end{gather}
\end{Lemma}

Lemma \ref{lem:Wronski-a} will be proved in Section \ref{Section6}.

\begin{remark*}Tarasov and Varchenko \cite{TV97} have already obtained a determinant formula for a multiple contour integrals of hypergeometric type, which is similar to~(\ref{eq:Wronski-a1}). See \cite[Theorem~5.9] {TV97} for the details.
\end{remark*}

In the succeeding sections, we give proofs for Theorems \ref{thm:MainThm-interpolation}, \ref{thm:expressionE}, \ref{thm:detE} and
Lemmas~\ref{lem:sym-hol},~\ref{lem:Wronski-a}, which we use for proving our main theorems.

This paper is organized as follows. In the first part of Section \ref{Section2} we prove Theorems~\ref{thm:MainThm-interpolation} and~\ref{thm:expressionE} on the basis of an explicit construction of the elliptic Lagrange interpolation functions by means of a kernel function as in~\cite{KNS09}. In the second part of Section~\ref{Section2} we investigate the transition coefficients between two sets of elliptic interpolation functions. In particular we provide a proof of Theorem~\ref{thm:detE} for the determinant of the transition matrix. A proof of Lemma~\ref{lem:sym-hol} for the regularized Jackson integrals will be given in Section~\ref{Section3}. In Section~\ref{Section4} we introduce certain interpolation polynomials which are a limiting case of the elliptic Lagrange interpolation functions. These polynomials are used in Section~\ref{Section5} for the construction of the $q$-difference system satisfied by the Jackson integrals. In particular we establish two-term difference equations with respect to $\alpha\to \alpha+1$ for the determinant of the Jackson integrals. Section~\ref{Section6} is devoted to analyzing the boundary condition for difference equations through asymptotic analysis of the Jackson integrals as $\alpha\to +\infty$, which completes the proof of Lemma~\ref{lem:Wronski-a}. In Appendix~\ref{SectionA}, we provide a detailed proof of Lemma~\ref{lem:Anabla} which is omitted in Section~\ref{Section5}. In Appendix~\ref{SectionB} we give proofs for some propositions in Section~\ref{Section4} by using the kernel function in the similar way as in Section \ref{Section2}.

\section[The elliptic Lagrange interpolation functions of type $A$]{The elliptic Lagrange interpolation functions of type $\boldsymbol{A}$}\label{Section2}

In this section we give proofs of Theorems \ref{thm:MainThm-interpolation}, \ref{thm:expressionE} and \ref{thm:detE}.

Let $P_+$ be the set of partitions of length at most $n$ specified by
\begin{gather*}
P_+=\big\{\lambda=(\lambda_1,\lambda_2,\ldots, \lambda_n)\in \mathbb{Z}^n\,|\,\lambda_1\ge\lambda_2\ge \cdots\ge \lambda_n\ge 0\big\}.
\end{gather*}
For $\mu=(\mu_1,\ldots,\mu_n)\in \mathbb{Z}^n$, we denote by $z^\mu$ the monomial $z_1^{\mu_1}\cdots z_n^{\mu_n}$. For the partitions $\lambda\in P_+$ let $m_{\lambda}(z)$ be the monomial symmetric functions~\cite{Mac95} defined by
\begin{gather*}
m_{\lambda}(z)=\sum_{\mu\in \frak{S}_n\lambda}z^\mu,
\end{gather*}
where $\frak{S}_n\lambda=\{w\lambda\,|\,w\in \frak{S}_n\}$ is the $\frak{S}_n$-orbit of $\lambda$. For a function $f(z)=f(z_1,z_2,\ldots,z_n)$ on~$({\mathbb{C}^*})^n$, we define the action of the symmetric group $\frak{S}_n$ on $f(z)$ by
\begin{gather*}
(\sigma f)(z)=f\big(\sigma^{-1}(z)\big)=f\big(z_{\sigma(1)},z_{\sigma(2)},\ldots,z_{\sigma(n)}\big)
\qquad\mbox{for}\quad \sigma\in \frak{S}_n.
\end{gather*}
We say that a function $f(z)$ on $({\mathbb{C}^*})^n$ is {\it symmetric} or {\it skew-symmetric} if $\sigma f(z)=f(z)$ or $\sigma f(z)=(\operatorname{sgn}\sigma) f(z)$ for all $\sigma \in \frak{S}_n$, respectively. We denote by ${\cal A} f(z)$ the alternating sum over~$\frak{S}_n$ defined by
\begin{gather}\label{eq:00Af}
({\cal A} f)(z)=\sum_{\sigma\in \frak{S}_n}(\operatorname{sgn} \sigma) \sigma f(z),
\end{gather}
which is skew-symmetric.

For an arbitrarily fixed $\zeta\in \mathbb{C}^*$, we consider the $\mathbb{C}$-linear subspace ${\cal H}_{s,n,\zeta}\subset {\cal O}(({\mathbb C}^*)^n)$ consisting of all $\frak{S}_n$-invariant holomorphic functions $f(z)$ such that $T_{q,z_i}f(z)=f(z)/(-z_i)^s \zeta$, $i=1,\ldots,n$, as in~\eqref{eq:def-Hsn}. In this section we use the symbol
\begin{gather*}
e(u,v)=u\theta(v/u),\qquad u,v\in \mathbb{C}^*.
\end{gather*}
Since $\theta(u)=\theta(q/u)$ and $\theta(qu)=-u^{-1}\theta(u)$, this symbol satisfies
\begin{gather*}
e(u,v)=-e(v,u),\qquad e(qu,v)=(-v/u)e(u,v).
\end{gather*}
In particular we have $e(u,v)\to u-v$ in the limit $q\to 0$.

\subsection{Construction of the interpolation functions}

This subsection is devoted to providing a proof of Theorem~\ref{thm:MainThm-interpolation}. In the first half, we define a~family of functions $E^{(n)}_\lambda(x;z)$ recursively with respect to the number $n$ of variables $z_1,\ldots,z_n$, and show that those $E^{(n)}_\lambda(x;z)$ are expressed explicitly as (\ref{eq:E-explicit-4}) in Theorem \ref{thm:expressionE}. In the second half, we prove that they are in fact the elliptic interpolation functions in the sense of Theorem~\ref{thm:MainThm-interpolation}; we show that $E^{(n)}_\lambda(x;z)\in \mathcal{H}_{s,n,\zeta}$ and $E^{(n)}_\lambda(x;x_\mu)=\delta_{\lambda\mu}$ using the dual Cauchy kernel as in~\cite{KNS09}.

First of all we show the following lemma.

\begin{Lemma}\label{lem:dim=<}
$\dim_{\mathbb{C}}\mathcal{H}_{s,n,\zeta}\le {n+s-1\choose n}$.
\end{Lemma}

\begin{proof}For an arbitrary $f(z)\in \mathcal{H}_{s,n,\zeta}$, since $f(z)$ is a holomorphic on $(\mathbb{C}^*)^n$, $f(z)$ may be expanded as Laurent series as $f(z)=\sum_{\lambda\in \mathbb{Z}^n}c_\lambda z^{\lambda}$. Since~$f(z)$ is symmetric, all coefficients of~$f(z)$ are determined from $c_\lambda$ corresponding to $\lambda\in P_+$. On the other hand, since $f(z)$ satisfies $T_{q,z_i}f(z)=f(z)/(-z_i)^s\zeta$ for $i=1,\ldots,n$, we have
\begin{gather*}
\sum_{\lambda\in \mathbb{Z}^n}q^{\lambda_i}c_\lambda z^{\lambda} =\sum_{\lambda\in \mathbb{Z}^n}c_\lambda z^{\lambda}/(-z_i)^s\zeta.
\end{gather*}
Equating coefficients of $z^\lambda$ on both sides, all coefficients of $f(z)$ are determined from $c_\lambda$ corresponding to~$\lambda$ satisfying $s-1 \ge \lambda_i \ge 0$, $i=1,\ldots,n$. Therefore, $f(z)$ is determined by the coefficients $c_\lambda$ corresponding to $\lambda\in B_{s,n}$ defined by~(\ref{eq:Bsn}). Since $|B_{s,n}|={n+s-1\choose n}$, we obtain $\dim_{\mathbb{C}}\mathcal{H}_{s,n,\zeta}\le {n+s-1\choose n}$.
\end{proof}

Before introducing $E^{(n)}_\lambda(x;z)$ we prove Theorem \ref{thm:MainThm-interpolation} for $n=1$ by independent means. From the definition~(\ref{eq:Zsn}), we have $Z_{s,1}=\{\epsilon_1,\epsilon_2,\ldots,\epsilon_s\}$, where $\epsilon_i$ is specified in Example~\ref{ex:connection1}, and have $x_{\epsilon_i}=x_i\in \mathbb{C}^*$.

\begin{Lemma}\label{lem:MainThm-interpolation(n=1)}
For $x=(x_1,\ldots,x_s)\in (\mathbb{C}^*)^s$ and $z\in \mathbb{C}^*$ the functions
\begin{gather}\label{eq:E-explicit(n=1)}
E_{\epsilon_i}(x;z) =\frac{e\Big(z\zeta \prod\limits_{k=1}^sx_k,x_i\Big)} {e\Big(x_i\zeta \prod\limits_{k=1}^sx_k,x_i\Big)}
\prod_{\substack{1\le j\le s\atop j\ne i}}\frac{e(z,x_j)}{e(x_i,x_j)}, \qquad i=1,\ldots,s,
\end{gather}
satisfy $E_{\epsilon_i}(x;x_{\epsilon_j})=\delta_{ij}$. The set $\{E_{\epsilon_i}(x;z)\,|\, i=1,\ldots,s\}$ is a basis of the $\mathbb{C}$-linear space $\mathcal{H}_{s,1,\zeta}$. In particular $\dim_\mathbb{C}\mathcal{H}_{s,1,\zeta}=s$.
\end{Lemma}

\begin{proof} It is directly confirmed that $E_{\epsilon_i}(x;x_{\epsilon_j})=\delta_{ij}$ from (\ref{eq:E-explicit(n=1)}). It is obvious that $E_{\epsilon_i}(x;z)\in \mathcal{H}_{s,1,\zeta}$ and the linearly independence of $\{E_{\epsilon_i}(x;z)\,|\, i=1,\ldots,s\}$ follows since $E_{\epsilon_i}(x;x_{\epsilon_j})=\delta_{ij}$. Since we have $\dim_\mathbb{C}\mathcal{H}_{s,1,\zeta}\le s$ by Lemma~\ref{lem:dim=<}, we see that
$\{E_{\epsilon_i}(x;z)\,|\, i=1,\ldots,s\}$ is a basis of~$\mathcal{H}_{s,1,\zeta}$, and we obtain
$\dim_\mathbb{C}\mathcal{H}_{s,1,\zeta}= s$.
\end{proof}

\begin{Definition} For $x=(x_1,\ldots,x_s)\in (\mathbb{C}^*)^s$ and $z=(z_1,\ldots,z_n)\in (\mathbb{C}^*)^n$, let $E_\lambda^{(n)}(x;z)$, $\lambda\in Z_{s,n}$, be functions defined inductively by $E^{(1)}_{\epsilon_k}(x;z_1)=E_{\epsilon_k}(x;z_1)$, $k=1,\ldots,s$, and for $n\ge 2$ by
\begin{gather}\label{eq:E-recurrence}
E_\lambda^{(n)}(x;z_1,\ldots,z_n)=\sum_{\substack{1\le k\le s\\[1pt] \lambda_k>0} }E_{\lambda-\epsilon_k}^{(n-1)}(x;z_1,\ldots,z_{n-1})
E_{\epsilon_k}^{(1)}\big(xt^{\lambda-\epsilon_k};z_{n}\big),
\end{gather}
where
\begin{gather}\label{eq:xtmu}
xt^{\mu}=\big(x_1t^{\mu_1},x_2t^{\mu_2},\ldots,x_st^{\mu_s}\big)\in (\mathbb{C}^*)^s.
\end{gather}
\end{Definition}

By the repeated use of (\ref{eq:E-recurrence}) we have
\begin{gather}
E_\lambda^{(n)}(x;z) = \!\!\sum_{\substack{{\substack{(i_1,\ldots,i_n)\\\hspace{3pt}\in \{1,\ldots,s\}^n,}}\\[2pt] \epsilon_{i_1}+\cdots+\epsilon_{i_n}=\lambda}}\!\! E_{\epsilon_{i_1}}(x;z_1)E_{\epsilon_{i_2}}\!\big(xt^{\epsilon_{i_1}};z_2\big)
E_{\epsilon_{i_3}}\!\big(xt^{\epsilon_{i_1}+\epsilon_{i_2}};z_3\big)
\cdots E_{\epsilon_{i_n}}\!\big(xt^{\epsilon_{i_1}+\cdots+\epsilon_{i_{n-1}}};z_n\big)\nonumber\\
\hphantom{E_\lambda^{(n)}(x;z)}{} =\sum_{\substack{{(i_1,\ldots,i_n)\in \{1,\ldots,s\}^n}\\[1pt] \epsilon_{i_1}+\cdots+\epsilon_{i_n}=\lambda}}
\prod_{k=1}^nE_{\epsilon_{i_k}}\big(xt^{\lambda^{(k-1)}};z_k\big),\label{eq:E-explicit-1}
\end{gather}
where $\lambda^{(k)}=\epsilon_{i_1}+\cdots+\epsilon_{i_k}$. Rewriting this formula as in \cite[p.~373, Theorem~3.4]{IN16} we obtain the following.

\begin{Theorem}\label{thm:expressionE2} The functions $E^{(n)}_\lambda(x;z)$ defined by \eqref{eq:E-recurrence} are expressed explicitly as
\begin{gather*}%\label{eq:E-explicit-3}
E_\lambda^{(n)}(x;z)=\sum_{\substack{K_1\sqcup\cdots\sqcup K_s \\ =\{1,2,\ldots,n\}}} \prod_{i=1}^s\prod_{k\in K_i}E_{\epsilon_{i}}\big(xt^{\lambda^{(k-1)}};z_k\big),
\end{gather*}
where $\lambda^{(k)}=\big(\lambda_1^{(k)},\ldots,\lambda_s^{(k)}\big)\in Z_{s,k}$, $\lambda_i^{(k)}=|K_i\cap\{1,2,\ldots,k\}|$ and the summation is taken over all index sets $K_i$, $i=1,2,\ldots,s$ satisfying $|K_i|=\lambda_i$ and $K_1\sqcup\cdots\sqcup K_s=\{1,2,\ldots,n\}$.
\end{Theorem}

We remark that these functions for $\lambda=n\epsilon_i\in Z_{s,n}$ have simple factorized forms; this fact will be used in the next subsection.

\begin{Corollary}\label{cor:ne_i} For $n\epsilon_i\in Z_{s,n}$, $i=1,\ldots, s$, one has
\begin{gather}\label{eq:ne_i}
E_{n\epsilon_i}^{(n)}(x;z) =
\frac{\prod\limits_{k=1}^n e\Big(z_k\zeta \prod\limits_{m=1}^sx_m,x_i\Big)}
{e\Big(x_i\zeta \prod\limits_{m=1}^sx_m,x_i\Big)_n} \prod_{\substack{1\le j\le s\\ j\ne i}}
\frac{\prod\limits_{k=1}^ne(z_k,x_j)}{e(x_i,x_j)_n},
\end{gather}
where
\begin{gather}\label{eq:e(u,v)_r}
e(u,v)_{r}=e(u,v)e(ut,v)\cdots e\big(ut^{r-1},v\big) \qquad \text{for} \quad r=0,1,2,\ldots.
\end{gather}
\end{Corollary}

\begin{proof}If we put $\lambda=n\epsilon_i$ in the formula~(\ref{eq:E-explicit-1}) then the right-hand side reduces to a single term with $(i_1,i_2,\ldots,i_n)=(i,i,\ldots,i)$. Therefore, using~(\ref{eq:E-explicit(n=1)}) we obtain
\begin{gather*}
E_{n\epsilon_i}^{(n)}(x;z) =E_{\epsilon_{i}}(x;z_1)E_{\epsilon_{i}}\big(xt^{\epsilon_{i}};z_2\big)E_{\epsilon_{i}}\big(xt^{2\epsilon_{i}};z_3\big)\cdots
E_{\epsilon_{i}}\big(xt^{(n-1)\epsilon_{i}};z_n\big),
\end{gather*}
which coincides with (\ref{eq:ne_i}).
\end{proof}

We simply write $E_\lambda(x;z)=E_\lambda^{(n)}(x;z)$ when there is no fear of misunderstanding. In the remaining part of this subsection we show that $E_\lambda(x;z)\in \mathcal{H}_{s,n,\zeta}$ and $E_\lambda(x;x_\mu)=\delta_{\lambda\mu}$.

For $x=(x_1,\ldots,x_s)\in (\mathbb{C}^*)^s$ and $w=(w_1,\ldots,w_s)\in (\mathbb{C}^*)^s$, let $F_\mu(x;w)$, $\mu\in \mathbb{N}^s$, be functions specified by
\begin{gather}\label{eq:F(x;w)}
F_\mu(x;w)=\prod_{i=1}^s\prod_{j=1}^se(x_i,w_j)_{\mu_i},
\end{gather}
where $e(u,v)_{r}$ is given by~\eqref{eq:e(u,v)_r}. By definition the functions $F_\mu(x;w)$ satisfy
\begin{gather}\label{eq:FF=F}
F_\mu(x;w) F_\nu(xt^{\mu};w) =F_{\mu+\nu}(x;w),
\end{gather}
where $xt^{\mu}$ is given by \eqref{eq:xtmu}. For $z=(z_1,\ldots,z_n)\in (\mathbb{C}^*)^n$ and $w=(w_1,\ldots,w_s)\in (\mathbb{C}^*)^s$, let $\Psi(z;w)$ be function specified by
\begin{gather*}
\Psi(z;w)= \prod_{i=1}^n\prod_{j=1}^s e(z_i,w_j),
\end{gather*}
which we call the {\em dual Cauchy kernel}. If $\zeta\prod_{i=1}^s w_i=1$, then by definition $\Psi(z;w)$ as a holomorphic function of $z\in (\mathbb{C}^*)^n$ satisfies $\Psi(z;w)\in \mathcal{H}_{s,n,\zeta}$. Note that $F_\mu(x;w)$ of~(\ref{eq:F(x;w)}) is expressed as
\begin{gather*}
F_\mu(x;w)=\Psi(x_\mu;w).
\end{gather*}

\begin{Lemma}[duality] Under the condition $\zeta\prod_{i=1}^s w_i=1$, $\Psi(z;w)$ expands as
\begin{gather}\label{eq:duality}
\Psi(z;w) =\sum_{\mu\in Z_{s,n}} E_\mu(x;z)F_\mu(x;w).
\end{gather}
\end{Lemma}

\begin{proof} We proceed by induction on $n$, the cardinality of $z$. We consider the case $n=1$ as the base case. Under the condition $\zeta\prod_{i=1}^s w_i=1$, we have $\Psi(z_1;w)=\prod_{j=1}^s e(z_1,w_j)\in \mathcal{H}_{s,1}$. Then, from Lemma~\ref{lem:MainThm-interpolation(n=1)}, $\Psi(z_1;w)$ is expanded as $\Psi(z_1;w)=\sum_{i=1}^n\Psi(x_{\epsilon_i};w)E_{\epsilon_i}(x;z)$,
whose coefficient $\Psi(x_{\epsilon_i};w)$ is written as $\Psi(x_{\epsilon_i};w)=\Psi(x_i;w)=\prod_{j=1}^s e(x_i,w_j)=F_{\epsilon_i}(x;w)$. This indicates~(\ref{eq:duality}) of the case $n=1$.

Next we suppose $n\ge 2$. We assume (\ref{eq:duality}) holds for the number of variables for $z$ less than~$n$. Denoting $z_{\widehat{n}}=(z_1,\ldots,z_{n-1})\in (\mathbb{C}^*)^{n-1}$ for $z=(z_1,\ldots,z_n)\in (\mathbb{C}^*)^n$, (\ref{eq:E-recurrence}) is rewritten as
\begin{gather}\label{eq:E-recurrence2}
E_\lambda^{(n)}(x;z)=\sum_{i=1}^s E_{\lambda-\epsilon_{i}}^{(n-1)}(x;z_{\widehat{n}})E_{\epsilon_{i}}^{(1)}\big(xt^{\lambda-\epsilon_{i}};z_n\big),
\end{gather}
where we regard $E_{\lambda-\epsilon_{i}}^{(n-1)}(x;z_{\widehat{n}})=0$ if $\lambda-\epsilon_{i}\not\in Z_{s,n-1}$. Then, we obtain
\begin{align*}
\Psi(z;w) & =\prod_{i=1}^n\prod_{j=1}^s e(z_i,w_j)=
\prod_{i=1}^{n-1}\prod_{j=1}^s e(z_i,w_j) \times \prod_{j=1}^s e(z_n,w_j)\\
&=\sum_{\mu\in Z_{s,n-1}}E_\mu^{(n-1)}(x;z_{\widehat{n}})F_\mu(x;w)\bigg(
\sum_{\nu\in Z_{s,1}}E_\nu^{(1)}(xt^\mu;z_n)F_\nu(xt^\mu;w)\bigg)\\
& \hspace{170pt}\mbox{(by the assumption of induction)}\\
&=\sum_{\substack{\mu\in Z_{s,n-1}\\[1pt] \nu\in Z_{s,1}}}E_\mu^{(n-1)}(x;z_{\widehat{n}})E_\nu^{(1)}(xt^\mu;z_n)F_\mu(x;w)F_\nu(xt^\mu;w)\\
&=\sum_{\substack{\mu\in Z_{s,n-1}\\[1pt] \nu\in Z_{s,1}}}E_\mu^{(n-1)}(x;z_{\widehat{n}})E_\nu^{(1)}(xt^\mu;z_n)F_{\mu+\nu}(x;w)
\hspace{52.5pt}\mbox{(by (\ref{eq:FF=F}))}\\
&=\sum_{\lambda\in Z_{s,n}}\bigg(\sum_{\substack{\mu\in Z_{s,n-1},\,\nu\in Z_{s,1}\\[1pt] \mu+\nu=\lambda}}E_\mu^{(n-1)}(x;z_{\widehat{n}})E_\nu^{(1)}(xt^\mu;z_n)\bigg)F_\lambda(x;w)\\
&=\sum_{\lambda\in Z_{s,n}}E_\mu^{(n)}(x;z)F_\lambda(x;w),\hspace{145pt}\mbox{(by (\ref{eq:E-recurrence2}))}
\end{align*}
as desired.
\end{proof}

When we consider the family of functions $F_\mu(x;w)$, $\mu\in Z_{s,n}$, on the hypersurface $\big\{w\in (\mathbb{C}^*)^s\,|\,\zeta\prod_{i=1}^{s}w_i=1\big\}$, we regard $F_\mu(x;w)$ as a function of $s-1$ variables $(w_1,\ldots,w_{s-1})\in (\mathbb{C}^*)^{s-1}$ with the relation $w_s=\big(\zeta\prod_{i=1}^{s-1}w_i\big)^{-1}$. Here, for $x=(x_1,\ldots,x_s)\in (\mathbb{C}^*)^{s}$ and
$\nu=(\nu_1,\ldots,\nu_s)\in Z_{s,n}$, we define a special point $\eta_\nu(x)$ on the hypersurface as
\begin{gather*}
\eta_\nu(x)=\left(x_1t^{\nu_1},\ldots,x_{s-1}t^{\nu_{s-1}}, \left(\zeta{ \prod_{i=1}^{s-1}x_it^{\nu_i}}\right)^{-1}\right).
\end{gather*}

\begin{Lemma}[triangularity]\label{lem:triangular-F} For each $\mu,\nu\in Z_{s,n}$, $F_\mu(x;\eta_{\nu}(x))=0$ unless $\mu_i\le \nu_i$ for $i=1,\ldots,$ $s-1$. In particular, $F_\mu(x;\eta_\nu(x))=0$ for $\mu\succ \nu$ with respect to the lexicographic order of $Z_{s,n}$. Moreover, if $x\in (\mathbb{C}^*)^s$ is generic, then $F_\mu(x;\eta_\mu(x))\ne 0$ for all $\mu\in Z_{s,n}$.
\end{Lemma}

This lemma implies that the matrix $F=\big(F_\mu(x;\eta_\nu(x))\big)_{\mu,\nu\in Z}$ is upper triangular, and also invertible if $x\in\mathbb{C}^{s}$ is generic.

\begin{proof} If there exists $j\in \{1,2,\ldots,s-1\}$ such that $\nu_j<\mu_j$, then $F_\mu(x;\eta_{\nu}(x))=0$. In fact, in the expression
\begin{gather*}
F_\mu(x;\eta_\nu(x)) =\prod_{i=1}^s \left[ e\left(x_i,\left(\zeta{ \prod_{i=1}^{s-1}x_it^{\nu_i}}\right)^{-1}\right)_{\mu_i}
\prod_{j=1}^{s-1} e\big(x_i,x_jt^{\nu_j}\big)_{\mu_i}
\right],
\end{gather*}
the function $e(x_j,x_jt^{\nu_j})_{\mu_j}$ has the factor $\theta(t^{\nu_j})\theta\big(t^{\nu_j-1}\big)\cdots \theta\big(t^{\nu_j-(\mu_j-1)}\big)=0$ if $\nu_j<\mu_j$. If $\nu\prec\mu$, then $\nu_i<\mu_i$ for some $i\in \{1,2,\ldots,s-1\}$ by definition, and hence we obtain $F_\mu(x;\eta_{\nu}(x))=0$ if $\nu\prec\mu$. For $\nu=\mu$,
\begin{gather*}
F_\mu(x;\eta_{\mu}(x)) = \prod_{i=1}^{s}\Bigg[ x_i^{\mu_i}t^{\mu_i\choose 2}
\prod_{j=0}^{\mu_i-1} \theta\left(x_i^{-1}t^{-j}\left(\zeta{ \prod_{k=1}^{s-1}x_kt^{\mu_k}}\right)^{-1}\right)\\
\hphantom{F_\mu(x;\eta_{\mu}(x)) =}{} \times \prod_{j=1}^{s-1} \theta\big(x_i^{-1}x_jt^{\mu_j}\big) \theta\big(x_i^{-1}x_jt^{\mu_j-1}\big)\cdots
\theta\big(x_i^{-1}x_jt^{\mu_j-\mu_i+1}\big) \Bigg]
\end{gather*}
does not vanish if we impose an appropriate genericity condition on $x\in (\mathbb{C}^*)^s$.
\end{proof}

\begin{proof}[Proof of Theorem \ref{thm:MainThm-interpolation}]
We fix a generic point $x$ in $(\mathbb{C}^\ast)^s$. Putting $w=\eta_\nu(x)$ in the formula~(\ref{eq:duality}), we have
\begin{gather*}
\Psi(z;\eta_\nu(x)) =\sum_{\mu\in Z_{s,n}}E_\mu(x;z)F_\mu(x;\eta_{\nu}(x)) \qquad \text{for} \quad \nu \in Z_{s,n}.
\end{gather*}
In what follows, we denote by $G=(G_{\mu\nu}(x))_{\mu,\nu\in Z}$ the inverse matrix of $F=(F_{\mu}(x;\eta_\nu(x)))_{\mu,\nu\in Z}$ from Lemma~\ref{lem:triangular-F}. Then we obtain
\begin{gather}\label{eq:E in H}
E_\lambda(x;z)= \sum_{\nu\in Z_{s,n}} \Psi(z;\eta_\nu(x))G_{\nu\lambda}(x) \qquad \text{for} \quad \lambda \in Z_{s,n},
\end{gather}
which implies
\begin{gather*}
E_\lambda(x;z)\in \mathcal{H}_{s,n,\zeta}
\end{gather*}
because of $\Psi(z;\eta_\nu(x))\in \mathcal{H}_{s,n,\zeta}$. Setting $z=x_\mu$ for each $\mu\in Z_{s,n}$ as in \eqref{eq:x_mu}, from (\ref{eq:E in H}), we have
\begin{gather*}%\label{eq:f=delta}
E_\lambda(x;x_\mu) =\sum_{\nu\in Z_{s,n}} \Psi(x_\mu,\eta_\nu(x))G_{\nu\lambda}(x) =\sum_{\nu\in Z_{s,n}}F_\mu(x;\eta_\nu(x))G_{\nu\lambda}(x)
=\delta_{\lambda\mu}
\end{gather*}
for $\lambda,\mu\in Z_{s,n}$. From this we have the set $\{E_\lambda(x;z)\,|\,\lambda\in Z_{s,n}\}$ is linearly independent. Thus we obtain
$\dim_{\mathbb{C}}\mathcal{H}_{s,n,\zeta} \ge |\{E_\lambda(x;z)\,|\,\lambda\in Z_{s,n}\}| ={n+s-1\choose n}$.
On the other hand, we have already shown that $\dim_{\mathbb{C}}\mathcal{H}_{s,n,\zeta}\le {n+s-1\choose n}$ in Lemma~\ref{lem:dim=<}. We therefore obtain $\dim_{\mathbb{C}}\mathcal{H}_{s,n,\zeta}={n+s-1\choose n}$ and see that $\{E_\lambda(x;z)\,|\,\lambda\in Z_{s,n}\}$ is a basis of $\mathcal{H}_{s,n,\zeta}$, satisfying $E_\lambda(x;x_\mu)=\delta_{\lambda\mu}$ $(\lambda,\mu\in Z_{s,n})$.
\end{proof}

\subsection{Transition coefficients for the interpolation functions}\label{Section2.1}

In this subsection we investigate the transition coefficients between two sets of interpolation functions with different parameters. Theorem~\ref{thm:detE} in the Introduction follows from Theorem \ref{thm:detE2} below when $\zeta=q^\alpha t^{n-1}b_1\cdots b_s$.%\\

For generic $x,y\in (\mathbb{C}^*)^s$, the interpolation functions $E_\mu(x;z)\in \mathcal{H}_{s,n,\zeta}$ as functions of
\smash{$z\in (\mathbb{C}^*)^n\!$} are expanded in terms of $E_\nu(y;z)$ ($\nu\in Z_{s,n}$), i.e.,
\begin{gather*}
E_\mu(x;z)=\sum_{\nu\in Z_{s,n}} C_{\mu\nu}(x;y)E_\nu(y;z),
\end{gather*}
where the coefficients $C_{\mu\nu}(x;y)$ are independent of~$z$. From the property~(\ref{eq:E-delta}) of the interpolation functions, we immediately see that $C_{\mu\nu}(x;y)$ may be expressed in terms of special value of $E_\mu(x;z)$ as
\begin{gather*}
C_{\mu\nu}(x;y)=E_\mu (x;y_{\nu}) \qquad \text{for} \quad \mu,\nu\in Z_{s,n}.
\end{gather*}
For $x,y\in (\mathbb{C}^*)^s$, we denote the {\em transition matrix} from $(E_\lambda(x;z))_{\lambda\in Z_{s,n}}$ to $(E_\lambda(y;z))_{\lambda\in Z_{s,n}}$ by
\begin{gather*}
{\bf E}(x;y)=\big(E_\mu (x;y_{\nu})\big)_{\!\mu,\nu\in Z_{s,n}},
\end{gather*}
where the rows and the columns are arranged in the total order $\prec$ of $Z_{s,n}$. By definition, for generic $x,y,w\in (\mathbb{C}^*)^s$ we have
\begin{gather*}\label{eq:E=EE}
{\bf E}(x;y)={\bf E}(x;w){\bf E}(w;y),
\end{gather*}
in particular
\begin{gather*}
{\bf E}(x;x)=I \qquad \mbox{and} \qquad {\bf E}(y;x)= {\bf E}(x;y)^{-1}.
\end{gather*}

\begin{Theorem}\label{thm:detE2}
For generic $x,y\in (\mathbb{C}^*)^s$ the determinant of the transition matrix ${\bf E}(x;y)$ is given explicitly as
\begin{gather*}
\det {\bf E}(x;y)=\prod_{k=1}^n\left[\frac{\theta\Big(\zeta t^{k-1}\prod\limits_{i=1}^s y_i\Big)}{\theta\Big(\zeta t^{k-1}\prod\limits_{i=1}^s x_i\Big)}
\right]^{s+k-2\choose k-1} \!\!\!\times \prod_{k=1}^n\left[\prod_{r=0}^{n-k}\prod_{1\le i<j\le s}\frac{e\big(y_it^r,y_jt^{(n-k)-r}\big)}{e\big(x_it^r,x_jt^{(n-k)-r}\big)}\right]^{s+k-3\choose k-1}.\label{eq:detE-1}
\end{gather*}
\end{Theorem}

\begin{proof}
The proof of this theorem is reduced to a special case indicated as Lemma \ref{lem:E-triangular} below. The details are the same as in the case of the $BC_n$ interpolation functions of \cite[p.~377, proof of Theorem~4.1]{IN16}.
\end{proof}

\begin{Lemma}\label{lem:E-triangular}For $x,y\in (\mathbb{C}^*)^s$ suppose that $y_i=x_i$, $i=1,2,\ldots,s-1$, i.e., $y=(x_1,\ldots,x_{s-1},y_s)$. For $\alpha,\beta\in Z_{s,n}$ if there exists $i\in\{1,2,\ldots,s-1\}$ such that $\alpha_i<\beta_i$, then $E_\alpha(x;y_\beta)=0$. In particular, ${\bf E}(x,y)$ is a lower triangular matrix with the diagonal entries
\begin{gather*}
E_\alpha(x;y_\alpha) = \frac{e\Big(y_s\zeta t^{\alpha_1+\cdots+\alpha_{s-1}}\prod\limits_{m=1}^sx_m,x_s\Big)_{\alpha_s}}
{e\Big(x_s\zeta t^{\alpha_1+\cdots+\alpha_{s-1}}\prod\limits_{m=1}^sx_m,x_s\Big)_{\alpha_s}}
\prod_{i=1}^{s-1}\frac{e\big(y_s,x_it^{\alpha_i}\big)_{\alpha_s}}{e\big(x_s,x_it^{\alpha_i}\big)_{\alpha_s}}.
\end{gather*}
Moreover the determinant of ${\bf E}(x,y)$ for $y=(x_1,\ldots,x_{s-1},y_s)$ is expressed as
\begin{gather*}
\det {\bf E}(x,y) =\prod_{\alpha\in Z_{s,n}}\left[
\frac{e\Big(y_s\zeta t^{\alpha_1+\cdots+\alpha_{s-1}}\prod\limits_{m=1}^sx_m,x_s\Big)_{\alpha_s}}
{e\Big(x_s\zeta t^{\alpha_1+\cdots+\alpha_{s-1}}\prod\limits_{m=1}^sx_m,x_s\Big)_{\alpha_s}}
\prod_{i=1}^{s-1}
\frac{e\big(y_s,x_it^{\alpha_i}\big)_{\alpha_s}}{e\big(x_s,x_it^{\alpha_i}\big)_{\alpha_s}}\right]
\nonumber\\
\hphantom{\det {\bf E}(x,y)}{} =\prod_{k=1}^n
\left[\frac{\theta\Big(\zeta t^{k-1}y_s\prod\limits_{i=1}^{s-1} x_i\Big)}{\theta\Big(\zeta t^{k-1}\prod\limits_{i=1}^s x_i\Big)}
\right]^{s+k-2\choose k-1} \times
\prod_{k=1}^n\left[\prod_{r=0}^{n-k}\prod_{i=1}^{s-1}
\frac{e\big(x_it^r,y_st^{(n-k)-r}\big)}{e\big(x_it^r,x_st^{(n-k)-r}\big)}\right]^{s+k-3\choose k-1}.%\label{eq:detE-3}
\end{gather*}
\end{Lemma}

\begin{proof}See the proof of \cite[p.~376, Lemma~4.2]{IN16}, using (\ref{eq:ne_i}) in Corollary \ref{cor:ne_i} instead of \cite[p.~374, equation~(3.10) of Corollary~3.4]{IN16}.
\end{proof}

\section[Jackson integral of $A$-type]{Jackson integral of $\boldsymbol{A}$-type}\label{Section3}

\subsection {Regularization of the Jackson integral}

For $z=(z_1,z_2,\ldots,z_n)\in (\mathbb{C}^*)^n$, let $\Phi(z)$ and $\Delta(z)$ be functions specified by (\ref{eq:Phi}) and (\ref{eq:Delta}), respectively. In considering the sum
\begin{gather*}
\langle\varphi,\xi\rangle =(1-q)^n\sum_{\nu\in \mathbb{Z}^n}\varphi(\xi q^\nu)\Phi(\xi q^\nu)\Delta(\xi q^\nu),
\end{gather*}
we always assume the following conditions for the point $\xi=(\xi_1,\ldots,\xi_n)\in (\mathbb{C}^*)^n$:
\begin{gather}\label{eq:condition:generic}
b_m \xi_i,\,t \xi_k/\xi_j\not\in q^{\mathbb{Z}}=\big\{q^l\,|\,l\in \mathbb{Z}\big\} \qquad\mbox{for}\quad 1\le m\le s,\quad 1\le i\le n,\quad 1\le j<k\le n.\!\!
\end{gather}

\begin{Lemma}\label{lem:converge01}
Let $\xi\in (\mathbb{C}^*)^n$ be a point satisfying \eqref{eq:condition:generic} and $a_m/\xi_i\not\in q^{\mathbb{Z}}$, $i=1,\ldots,n$, $m=1,\ldots,s$. For $m_\lambda(z)$, $\lambda\in B_{s,n}$, the sum $\langle m_\lambda,\xi\rangle$ converges absolutely if $\alpha\in \mathbb{C}$, $a_m,b_m\in \mathbb{C}^*$, $m=1,\ldots,s$ and $t\in \mathbb{C}^*$ satisfy
\begin{gather}
\big|qa_1^{-1}\cdots a_s^{-1}b_1^{-1}\cdots b_s^{-1}\big|<\big|q^{\alpha}\big|<1,
\qquad %\mbox{and}\nonumber\\
\big|qa_1^{-1}\cdots a_s^{-1}b_1^{-1}\cdots b_s^{-1}\big|<\big|q^{\alpha}t^{2n-2}\big|<1.\label{eq:condition:converge01}
\end{gather}
\end{Lemma}

\begin{proof} Since $\langle m_\lambda,\xi\rangle =\sum_{\mu\in \frak{S}_n\lambda}\langle z^\mu,\xi\rangle$, it suffices to show that $\langle z^\mu,\xi\rangle $ converges absolutely under the conditions~\eqref{eq:condition:converge01} and
\begin{gather}\label{eq:condition:mu}
0\le \mu_i\le s-1 \qquad \text{for} \quad i=1,2,\ldots,n .
\end{gather}
In the following we set $D=\{\nu=(\nu_1,\nu_2,\ldots,\nu_n)\in \mathbb{Z}^n\,|\, \nu_1\le \nu_2\le\cdots\le \nu_n\}$. Note that under the action of $\frak{S}_n$ defined by $\sigma\nu=(\nu_{\sigma^{-1}(1)},\nu_{\sigma^{-1}(2)},\ldots,\nu_{\sigma^{-1}(n)})$, we have
\begin{gather*}
\sigma D=\big\{\nu=(\nu_1,\nu_2,\ldots,\nu_n)\in \mathbb{Z}^n\,|\, \nu_{\sigma(1)}\le \nu_{\sigma(2)}\le\cdots\le \nu_{\sigma(n)}\big\} \qquad \text{for} \quad \sigma\in \frak{S}_n.
\end{gather*}
For each $ I\subset \{1,2,\ldots,n\}$ we set
\begin{gather*}
Q_I=\{\nu=(\nu_1,\nu_2,\ldots,\nu_n)\in \mathbb{Z}^n\,|\, \nu_i\ge 0\ \text{if} \ i\in I\ \text{and}\ \nu_i\le 0\ \text{if} \ i\not\in I\}.
\end{gather*}
We divide $\Phi(z)$ into two parts as
\begin{gather*}
\Phi(z)=\Phi_1(z)\Phi_2(z),
\end{gather*}
where
\begin{gather}\label{eq:Phi1,Phi2}
\Phi_1(z)=\prod_{i=1}^n z_i^\alpha \prod_{m=1}^s\frac{\big(qa_m^{-1}z_i\big)_\infty}{(b_mz_i)_\infty}
\qquad\mbox{and}\qquad
\Phi_2(z)=\prod_{1\le j<k\le n}z_j^{2\tau-1}\frac{\big(qt^{-1}z_k/z_j\big)_\infty}{(tz_k/z_j)_\infty}.
\end{gather}
Since the factor $\big(qa_m^{-1}z_i\big)_\infty/(b_mz_i)_\infty$ in $\Phi_1(z)$ is equal to
\begin{gather*}
\frac{\big(qa_m^{-1}z_i\big)_\infty}{(b_mz_i)_\infty}= \left(z_i^{1-\alpha_m-\beta_m}\frac{\theta\big(qa_m^{-1}z_i\big)}{\theta(b_mz_i)}\right)
\left(\big(z_i^{-1}\big)^{1-\alpha_m-\beta_m}\frac{\big(qb_m^{-1}z_i^{-1}\big)_\infty}{\big(a_mz_i^{-1}\big)_\infty}\right)
\end{gather*}
with $q$-periodic function $z_i^{1-\alpha_m-\beta_m}\theta\big(qa_m^{-1}z_i\big)/\theta(b_mz_i)$, the values $(\xi q^{\nu})^\mu\,\Phi_1(\xi q^\nu)$ for $\nu\in Q_I$ are explicitly written as
\begin{gather}
(\xi q^{\nu})^\mu\,\Phi_1(\xi q^\nu) = \prod_{i\in I}\left[(\xi_iq^{\nu_i})^{\alpha+\mu_i}\prod_{m=1}^s\frac{\big(q^{1+\nu_i}a_m^{-1}\xi_i\big)_\infty}{(q^{\nu_i}b_m\xi_i)_\infty}\right]
\times\prod_{i\not\in I}\left[\left(\prod_{m=1}^s\xi_i^{1-\alpha_m-\beta_m}\frac{\theta\big(qa_m^{-1}\xi_i\big)}{\theta(b_m\xi_i)}\right)\right.
\nonumber\\
\left.\hphantom{(\xi q^{\nu})^\mu\,\Phi_1(\xi q^\nu) = \prod_{i\in I}}{} \times\left(\big(\xi_i^{-1} q^{-\nu_i}\big)^{-\alpha-\mu_i}
\prod_{m=1}^s\big(\xi_i^{-1} q^{-\nu_i}\big)^{1-\alpha_m-\beta_m}\frac{\big(q^{1-\nu_i}b_m^{-1}\xi_i^{-1}\big)_\infty}{\big(q^{-\nu_i}a_m\xi_i^{-1}\big)_\infty}\right)\right].\!\!\!\label{eq:Phi1(xq)}
\end{gather}
For each $I$ there exists a constant $C_1>0$ such that
\begin{gather}
\Bigg|\prod_{i\in I} \xi_i^{\alpha+\mu_i}
\prod_{m=1}^s\frac{\big(q^{1+\nu_i}a_m^{-1}\xi_i\big)_\infty}{(q^{\nu_i}b_m\xi_i)_\infty}\nonumber\\
\qquad{}\times
\prod_{i\not\in I}
\left(\prod_{m=1}^s
\xi_i^{1-\alpha_m-\beta_m}\frac{\theta\big(qa_m^{-1}\xi_i\big)}{\theta(b_m\xi_i)}\right)
\left(\xi_i^{\alpha+\mu_i}\prod_{m=1}^s\frac{(q^{1-\nu_i}b_m^{-1}\xi_i^{-1})_\infty}{\big(q^{-\nu_i}a_m\xi_i^{-1}\big)_\infty}\right)\Bigg|\le C_1,\label{eq:bounded01}
\end{gather}
for all $\nu\in Q_I$. From \eqref{eq:Phi1(xq)} and \eqref{eq:bounded01}, for $\nu\in Q_I$ and $\mu$ satisfying \eqref{eq:condition:mu} we have
\begin{gather}
\big|(\xi q^{\nu})^\mu\,\Phi_1(\xi q^\nu)\big| \le C_1\left| \prod_{i\in I} (q^{\nu_i})^{\alpha+\mu_i}\prod_{i\not\in I} (q^{-\nu_i})^{-\alpha-\mu_i+s-\alpha_1-\cdots-\alpha_s-\beta_1-\cdots-\beta_s}
\right|\nonumber\\
\hphantom{\big|(\xi q^{\nu})^\mu\,\Phi_1(\xi q^\nu)\big|}{}
=C_1\prod_{i\in I} \big|q^{\alpha+\mu_i}\big|^{\nu_i}\prod_{i\not\in I} \big|q^{s-\mu_i-\alpha}a_1^{-1}\cdots a_s^{-1}b_1^{-1}\cdots b_s^{-1}\big|^{-\nu_i} \nonumber\\
\hphantom{\big|(\xi q^{\nu})^\mu\,\Phi_1(\xi q^\nu)\big|}{} \le C_1\prod_{i\in I} |q^{\alpha}|^{\nu_i}\prod_{i\not\in I} \big|q^{1-\alpha}a_1^{-1}\cdots a_s^{-1}b_1^{-1}\cdots b_s^{-1}\big|^{-\nu_i}.\label{eq:|Phi1(xq)|}
\end{gather}
On the other hand, from \eqref{eq:Phi1,Phi2} the function $\Phi_2(z)$ satisfies
\begin{gather*}%\label{eq:sigmaPhi2}
\sigma \Phi_2 (z)=U_{\sigma}(z)\,\Phi_2(z) \qquad\mbox{for}\quad \sigma \in \frak{S}_n
\end{gather*}
with a family of $q$-periodic functions
\begin{gather}\label{eq:Usigma}
U_\sigma(z)=\prod_{\substack {1\le i<j\le n\\ \sigma^{-1}(i)>\sigma^{-1}(j)}}
\Big(\frac{z_i}{z_j}\Big)^{\!1-2\tau}\frac{\theta\big(q^{1-\tau}z_i/z_j\big)}{\theta(q^\tau z_i/z_j)}.
\end{gather}
Since $\sigma \Delta (z)=(\operatorname{sgn} \sigma) \Delta(z)$ for $ \sigma \in \frak{S}_n$, we have
\begin{gather*}%\label{eq:sigmaPhi2Delta}
\sigma \Phi_2 (z)\sigma \Delta (z)=(\operatorname{sgn} \sigma)U_{\sigma}(z) \Phi_2(z) \Delta (z)\qquad\mbox{for}\quad \sigma \in \frak{S}_n,
\end{gather*}
where
\begin{gather*}
\Phi_2(z)\Delta(z) =\prod_{i=1}^n z_i^{2(n-i)\tau}\prod_{1\le j<k\le n}\frac{\big(qt^{-1}z_k/z_j\big)_\infty}{(tz_k/z_j)_\infty}(1-z_k/z_j).
\end{gather*}
For $\nu\in \sigma D$ the values $\Phi_2(\xi q^\nu)\Delta(\xi q^\nu)$ are expressed as
\begin{gather}
 \Phi_2(\xi q^\nu)\Delta(\xi q^\nu) =(\operatorname{sgn} \sigma)U_\sigma(\xi q^\nu)^{-1}\sigma\Phi_2(\xi q^\nu)\sigma\Delta(\xi q^\nu)\nonumber\\
\hphantom{\Phi_2(\xi q^\nu)\Delta(\xi q^\nu)}{} =(\operatorname{sgn} \sigma)U_\sigma(\xi)^{-1}\prod_{i=1}^n (\xi_i q^{\nu_i})^{2(n-\sigma^{-1}(i))\tau}\label{eq:Phi2(xq)}\\
\hphantom{\Phi_2(\xi q^\nu)\Delta(\xi q^\nu)=}{}\times \prod_{1\le j<k\le n}\frac{\big(q^{1+\nu_{\sigma(k)}-\nu_{\sigma(j)}}t^{-1}\xi_{\sigma(k)}/\xi_{\sigma(j)}\big)_\infty}
{\big(q^{\nu_{\sigma(k)}-\nu_{\sigma(j)}}t\xi_{\sigma(k)}/\xi_{\sigma(j)}\big)_\infty}\big(1-q^{\nu_{\sigma(k)}-\nu_{\sigma(j)}}\xi_{\sigma(k)}/\xi_{\sigma(j)}\big).\nonumber
\end{gather}
For each $\sigma \in \frak{S}_n$ there exists a constant $C_2>0$ such that
\begin{gather}
\bigg|\frac{\operatorname{sgn} \sigma}{U_\sigma(\xi)}\prod_{i=1}^n \xi_i^{2(n-\sigma^{-1}(i))\tau}\nonumber\\
\qquad{}\times\prod_{1\le j<k\le n} \frac{\big(q^{1+\nu_{\sigma(k)}-\nu_{\sigma(j)}}t^{-1}\xi_{\sigma(k)}/\xi_{\sigma(j)}\big)_\infty}
{\big(q^{\nu_{\sigma(k)}-\nu_{\sigma(j)}}t\xi_{\sigma(k)}/\xi_{\sigma(j)}\big)_\infty}\big(1-q^{\nu_{\sigma(k)}-\nu_{\sigma(j)}}\xi_{\sigma(k)}/\xi_{\sigma(j)}\big)\bigg|\le C_2,\label{eq:bounded02}
\end{gather}
for all $\nu\in \sigma D$. From \eqref{eq:Phi2(xq)} and \eqref{eq:bounded02}, for $\nu\in \sigma D$ we have
\begin{gather}
\big|\Phi_2(\xi q^\nu)\Delta(\xi q^\nu)\big|\le C_2\left|\prod_{i=1}^n (q^{\nu_i})^{2(n-\sigma^{-1}(i))\tau}\right|
=C_2\prod_{i=1}^n \big|t^{2(n-\sigma^{-1}(i))}\big|^{\nu_i}.\label{eq:|Phi2(xq)|}
\end{gather}
We now estimate the sum $\langle z^\mu,\xi\rangle$. Since $\mathbb{Z}^n=\cup_{I\subset \{1,\ldots,n\}}\cup_{\sigma \in \frak{S}_n}(Q_I\cap \sigma D)$, by definition we have
\begin{gather}
\sum_{\nu\in \mathbb{Z}^n}\big|(\xi q^{\nu})^\mu\,\Phi(\xi q^\nu)\Delta(\xi q^\nu)\big| \le \sum_{\substack{I\subset\{1,\ldots,n\}\\
\sigma\in \frak{S}_n}} \sum_{\nu\in Q_I\cap\sigma D}\big|(\xi q^{\nu})^\mu\,\Phi(\xi q^\nu)\Delta(\xi q^\nu)\big|.\label{eq:<zmu,xi>}
\end{gather}
Moreover, from \eqref{eq:|Phi1(xq)|} and \eqref{eq:|Phi2(xq)|}, we have
\begin{gather}
\sum_{\nu\in Q_I\cap\sigma D}
\big|(\xi q^{\nu})^\mu\,\Phi(\xi q^\nu)\Delta(\xi q^\nu)\big|
=\sum_{\nu\in Q_I\cap\sigma D}
\big|(\xi q^{\nu})^\mu\,\Phi_1(\xi q^\nu)\big|\big|\Phi_2(\xi q^\nu)\Delta(\xi q^\nu)\big|\nonumber\\
\le \sum_{\nu\in Q_I\cap\sigma D} C_1C_2 \prod_{i\in I} \big|q^{\alpha}t^{2(n-\sigma^{-1}(i))}\big|^{\nu_i}
\prod_{i\not\in I} \big|q^{1-\alpha}a_1^{-1}\cdots a_s^{-1}b_1^{-1}\cdots b_s^{-1} t^{-2(n-\sigma^{-1}(i))}\big|^{-\nu_i}. \!\!\!\!\! \label{eq:|Phi(xq)|}
\end{gather}
Under the conditions \eqref{eq:condition:converge01} we have
\begin{gather}
\big|q^{\alpha} t^{2i-2}\big|<1, \qquad %\mbox{and}\qquad
\big|q^{1-\alpha}a_1^{-1}\cdots a_s^{-1}b_1^{-1}\cdots b_s^{-1}t^{-2i+2}\big|<1
\qquad\mbox{for}\quad i=1,2,\ldots,n. \label{eq:condition:converge01a}
\end{gather}
From \eqref{eq:|Phi(xq)|} and \eqref{eq:condition:converge01a} the sum $\sum_{\nu\in Q_I\cap\sigma D}|(\xi q^{\nu})^\mu\,\Phi(\xi q^\nu)\Delta(\xi q^\nu)|$ converges. From~\eqref{eq:<zmu,xi>} the infinite sum $\langle z^\mu,\xi\rangle$ is absolutely convergent for each $\mu\in\frak{S}_n\lambda$.
\end{proof}

By definition the function $\Phi(z)$ satisfies the {\em quasi-symmetric property} that
\begin{gather}\label{eq:Usigma-01}
\sigma \Phi (z)=U_{\sigma}(z)\,\Phi(z) \qquad\mbox{for}\quad \sigma \in \frak{S}_n,
\end{gather}
where $U_\sigma(z)$, as given by (\ref{eq:Usigma}), is invariant under the $q$-shift $z_i\to q z_i$. If $\varphi(z)$ is a symmetric holomorphic function on $(\mathbb{C}^*)^n$, then we have
\begin{gather}\label{eq:Usigma-02}
\sigma\big(\varphi(z)\Phi (z)\Delta(z)\big)=(\operatorname{sgn} \sigma)U_{\sigma}(z)\,\varphi(z)\Phi(z) \Delta(z)\qquad\mbox{for}\quad \sigma \in \frak{S}_n.
\end{gather}
This implies that
\begin{gather}\label{eq:Usigma-I}
\sigma \langle\varphi, z\rangle=(\operatorname{sgn} \sigma)U_\sigma(z) \langle\varphi, z\rangle \qquad\mbox{for}\quad \sigma \in \frak{S}_n.
\end{gather}
In fact
\begin{gather*} \sigma \langle\varphi, z\rangle =\big\langle\varphi, \sigma^{-1}z\big\rangle=
(1-q)^n\sum_{\nu\in \mathbb{Z}^n}\varphi\big(\sigma^{-1}z q^{\nu}\big)\Phi\big(\sigma^{-1}z q^{\nu}\big)\Delta\big(\sigma^{-1}z q^{\nu}\big)\\
\hphantom{\sigma \langle\varphi, z\rangle}{} =
(1-q)^n\sum_{\nu\in \mathbb{Z}^n}\varphi\big(\sigma^{-1}z q^{\sigma^{-1}\nu}\big)\Phi\big(\sigma^{-1}z q^{\sigma^{-1}\nu}\big)\Delta\big(\sigma^{-1}z q^{\sigma^{-1}\nu}\big)\\
\hphantom{\sigma \langle\varphi, z\rangle}{}=(1-q)^n\sum_{\nu\in \mathbb{Z}^n}\sigma\varphi(z q^{\nu})\sigma\Phi(z q^{\nu})\sigma\Delta(z q^{\nu})\\
\hphantom{\sigma \langle\varphi, z\rangle}{} = (1-q)^n\sum_{\nu\in \mathbb{Z}^n}(\operatorname{sgn} \sigma)U_{\sigma}(z q^{\nu})\varphi(z q^{\nu})\Phi(z q^{\nu})\Delta(z q^{\nu})\\
\hphantom{\sigma \langle\varphi, z\rangle}{} =(1-q)^n(\operatorname{sgn} \sigma)U_{\sigma}(z)\sum_{\nu\in \mathbb{Z}^n}\varphi(z q^{\nu})\Phi(z q^{\nu})\Delta(z q^{\nu})\\
\hphantom{\sigma \langle\varphi, z\rangle}{} = (\operatorname{sgn} \sigma)U_{\sigma}(z)\langle\varphi, z\rangle.
\end{gather*}
\begin{Lemma}\label{lem:I-vanishing}
Suppose $\varphi(z)$ is symmetric and holomorphic. Under the condition $\tau\not\in\mathbb{Z}$, if $z_i=z_j$ for some $i$ and $j$, $1\le i<j\le n$, then $\langle\varphi, z\rangle=0$.
\end{Lemma}

\begin{proof}Let $\sigma=(i\,j)$ be the transposition of $i$ and $j$. If we impose $z_i=z_j$, then $\sigma \langle\varphi, z\rangle=\langle\varphi, z\rangle$, so that we have $(1+U_\sigma(z))\langle\varphi, z\rangle=0$
from (\ref{eq:Usigma-I}). Since $(1+U_\sigma(z))\ne 0$, we obtain $\langle\varphi, z\rangle=0$.
\end{proof}

For the regularized Jackson integral $\langle\!\langle\varphi,z\rangle\!\rangle$ defined in (\ref{eq:lla-rra}), we have the following.

\begin{Lemma}\label{lem:lla-rra} Suppose that $\tau\not\in\mathbb{Z}$. If $\varphi(z)$ is a symmetric holomorphic function on $(\mathbb{C}^*)^n$,
then $\langle\!\langle\varphi,z\rangle\!\rangle$ is also a symmetric holomorphic function on $(\mathbb{C}^*)^n$.
\end{Lemma}
\begin{proof}
By the definition (\ref{eq:lla-rra}), $\Theta(z)$ satisfies that $\sigma \Theta(z)=(\operatorname{sgn} \sigma) U_\sigma(z) \Theta(z)$. Combining this and~(\ref{eq:Usigma-I}), we obtain that $\langle\!\langle\varphi,z\rangle\!\rangle=\langle\varphi,z\rangle/\Theta(z)$ is symmetric if $\varphi(z)$ is symmetric. We now check $\langle\!\langle\varphi,z\rangle\!\rangle$ is holomorphic. Taking account of the poles of~$\Phi(z)$, we have the expression
\begin{gather}\label{eq:Phi(z)poles}
\langle\varphi,z\rangle=f(z)\prod_{i=1}^n\frac{z_i^\alpha}{\prod\limits_{m=1}^{s}\theta(b_mz_i)}\prod_{1\le j<k\le n}\frac{z_j^{2\tau-1}}{\theta(t z_k/z_j)},
\end{gather}
where $f(z)$ is some holomorphic function on $(\mathbb{C}^*)^n$. Since $\langle\varphi,z\rangle$ is an invariant under the $q$-shift $z_i\to qz_i$, $1\le i\le n$, Lemma~\ref{lem:I-vanishing} implies that $\langle\varphi,z\rangle$ is divisible by $\prod_{1\le i<j\le n}z_i\theta(z_j/z_i)$ if $\tau\not\in \mathbb{Z}$. This indicates that $f(z)=g(z)\prod_{1\le i<j\le n} z_i\theta(z_j/z_i)$, where $g(z)$ is a holomorphic function on~$(\mathbb{C}^*)^n$. From (\ref{eq:lla-rra}) and (\ref{eq:Phi(z)poles}), we therefore obtain that $\langle\!\langle\varphi,z\rangle\!\rangle=g(z)$ is a holomorphic function on~$(\mathbb{C}^*)^n$.
\end{proof}

\subsection{Truncation of the Jackson integral}

For $a=(a_1,\ldots,a_s)\in (\mathbb{C}^*)^s$ we set
\begin{gather}\label{eq:a_mu}
a_\mu=\big(\underbrace{a_1,a_1t,\ldots,a_1t^{\mu_1-1}\vphantom{\Big|}}_{\mu_1},\underbrace{a_2,a_2t,\ldots,a_2t^{\mu_2-1}\vphantom{\Big|}}_{\mu_2},\ldots,\underbrace{a_s,a_st,\ldots,a_st^{\mu_s-1}
\vphantom{\Big|}}_{\mu_s}\big) \in (\mathbb{C}^*)^n
\end{gather}
for each $\mu=(\mu_1,\ldots,\mu_s)\in Z_{s,n}$ as in (\ref{eq:x_mu}). We denote by $\Lambda_\mu$ the subset of the lattice $\mathbb{Z}^n$ consisting of all $\lambda=(\lambda_1,\ldots,\lambda_n)\in \mathbb{Z}^n$ such that
$0\le \lambda_1\le \lambda_2\le \cdots\le \lambda_{\mu_1}$ and
\begin{gather*}
0\le \lambda_{\mu_1+\mu_2+\cdots+\mu_{k}+1} \le \lambda_{\mu_1+\mu_2+\cdots+\mu_{k}+2}\le \cdots\le\lambda_{\mu_1+\mu_2+\cdots+\mu_{k}+\mu_{k+1}}\qquad\mbox{for}\quad 1\le k\le s-1.
\end{gather*}
Then we have $\Phi\big(a_\mu q^\lambda\big)=0$ unless $\lambda\in \Lambda_\mu$. Hence the Jackson integral $\langle\varphi,a_\mu\rangle$ is a sum over the cone $\Lambda_\mu\subset \mathbb{Z}^n$. We call $\langle\varphi,a_\mu\rangle$ $(\mu\in Z_{s,n})$ the {\it truncated Jackson integral of A-type}.

For $N\in \mathbb{Z}$, let $T_\alpha^N$ be the shift operator $\alpha\to\alpha+N$, i.e., $T_\alpha^N f(\alpha)=f(\alpha+N)$.
\begin{Lemma}\label{lem:<m, a>}
For $m_\lambda(z)$ $(\lambda\in P_+)$, the truncated Jackson integral $\langle m_\lambda,a_\mu\rangle$ converges absolutely if
\begin{gather}\label{eq:condition:converge02}
|q^{\alpha}|<1 \qquad\mbox{and}\qquad \big|q^{\alpha}t^{n-1}\big|<1.
\end{gather}
Moreover, under the condition \eqref{eq:condition:converge02}, the asymptotic behavior of $T_\alpha^N\langle m_\lambda,a_\mu\rangle$ as $N\to \infty$ is given by
\begin{gather}\label{eq:asympt <m, a>}
T_\alpha^N\langle m_\lambda,a_\mu\rangle\sim (1-q)^nm_\lambda(a_\mu)\Delta(a_\mu)T_\alpha^N\Phi(a_\mu) \qquad \text{as} \quad N\to \infty.
\end{gather}
\end{Lemma}

\begin{proof}Setting $D_0=\{\nu\in \mathbb{Z}^n\,|\, 0\le \nu_1\le\nu_2\le \cdots \le \nu_n\}$, we consider
\begin{gather*}
\sigma D_0=\big\{\nu\in \mathbb{Z}^n\,|\, 0\le \nu_{\sigma(1)}\le\nu_{\sigma(2)}\le \cdots \le \nu_{\sigma(n)}\big\} \qquad \text{for} \quad \sigma\in \frak{S}_n.
\end{gather*}
We analyze the convergence of the infinite sum
\begin{gather}\label{eq:<m, a>}
\langle m_\lambda,a_\mu\rangle=(1-q)^n\sum_{\nu\in \Lambda_\mu}m_\lambda(a_\mu q^{\nu})\Phi(a_\mu q^\nu)\Delta(a_\mu q^\nu).
\end{gather}
By \eqref{eq:Usigma-02} we have
\begin{gather}
\sum_{\nu\in \Lambda_\mu} \big|m_\lambda(a_\mu q^{\nu})\Phi(a_\mu q^\nu)\Delta(a_\mu q^\nu)\big| \le \sum_{\nu\in \mathbb{N}^n}\big|m_\lambda(a_\mu q^{\nu})\Phi(a_\mu q^\nu)\Delta(a_\mu q^\nu)\big|\nonumber\\
\qquad{} \le \sum_{\sigma\in \frak{S}_n} \sum_{\nu\in \sigma D_0}\big|m_\lambda(a_\mu q^{\nu})\Phi(a_\mu q^\nu)\Delta(a_\mu q^\nu)\big| \nonumber\\
\qquad{} = \sum_{\sigma\in \frak{S}_n} \sum_{\nu\in \sigma D_0}\big| U_\sigma(a_\mu q^{\nu})^{-1}\big| \big| \sigma m_\lambda(a_\mu q^{\nu})\sigma\Phi(a_\mu q^\nu)\sigma\Delta(a_\mu q^\nu)\big|\nonumber\\
\qquad{} = \sum_{\sigma\in \frak{S}_n}\big| U_\sigma(a_\mu)^{-1}\big| \sum_{\sigma^{-1}\nu\in D_0}\big| m_\lambda\big(\sigma^{-1}a_\mu q^{\sigma^{-1}\nu}\big)\Phi\big(\sigma^{-1}a_\mu q^{\sigma^{-1}\nu}\big)\Delta\big(\sigma^{-1}a_\mu q^{\sigma^{-1}\nu}\big)\big|\nonumber\\
\qquad{} = \sum_{\sigma\in \frak{S}_n}\big| U_\sigma(a_\mu)^{-1}\big| \sum_{\nu\in D_0}\big| m_\lambda(\sigma^{-1}a_\mu q^\nu)\Phi\big(\sigma^{-1}a_\mu q^\nu\big)\Delta\big(\sigma^{-1}a_\mu q^\nu\big)\big|.\label{eq:|<m, a>|}
\end{gather}
Hence it suffices to show that
\begin{gather*}
\sum_{\nu\in D_0}\big|m_\lambda(\xi q^{\nu})\Phi(\xi q^\nu)\Delta(\xi q^\nu)\big|, \qquad \text{where} \quad \xi=\sigma^{-1}a_\mu,
\end{gather*}
converges under the condition \eqref{eq:condition:converge02}. Since $m_\lambda(z)=\sum_{\kappa\in \frak{S}_n\lambda}z^\kappa$, we have
\begin{gather}\label{eq:Sigma|mPD|}
\sum_{\nu\in D_0}\big|m_\lambda(\xi q^{\nu})\Phi(\xi q^\nu)\Delta(\xi q^\nu)\big| \le \sum_{\kappa\in \frak{S}_n\lambda}\sum_{\nu\in D_0}\big|(\xi q^{\nu})^\kappa\Phi(\xi q^\nu)\Delta(\xi q^\nu)\big|.
\end{gather}
The explicit form of $(\xi q^{\nu})^\kappa\Phi(\xi q^\nu)\Delta(\xi q^\nu)$ is given by
\begin{gather*}
(\xi q^{\nu})^\kappa\Phi(\xi q^\nu)\Delta(\xi q^\nu) =\prod_{i=1}^n \left((\xi_i q^{\nu_i})^{\kappa_i+\alpha+2(n-i)\tau}
\prod_{m=1}^s\frac{\big(q^{1+\nu_i}a_m^{-1}\xi_i\big)_\infty}{(q^{\nu_i}b_m\xi_i)_\infty}\right)\\
\hphantom{(\xi q^{\nu})^\kappa\Phi(\xi q^\nu)\Delta(\xi q^\nu) =}{} \times \prod_{1\le j<k\le n}\frac{\big(q^{1+\nu_k-\nu_j}t^{-1}\xi_k/\xi_j\big)_\infty}{\big(q^{\nu_k-\nu_j}t\xi_k/\xi_j\big)_\infty}
\big(1-q^{\nu_k-\nu_j}\xi_k/\xi_j\big).
\end{gather*}
There exists a constant $C>0$ such that
\begin{gather*}
\Bigg|\prod_{i=1}^n
\left(\xi_i^{\kappa_i+\alpha+2(n-i)\tau}
\prod_{m=1}^s\frac{\big(q^{1+\nu_i}a_m^{-1}\xi_i\big)_\infty}{(q^{\nu_i}b_m\xi_i)_\infty}
\right) \\
\qquad{}\times \prod_{1\le j<k\le n}\frac{\big(q^{1+\nu_k-\nu_j}t^{-1}\xi_k/\xi_j\big)_\infty}{\big(q^{\nu_k-\nu_j}t\xi_k/\xi_j\big)_\infty}
\big(1-q^{\nu_k-\nu_j}\xi_k/\xi_j\big)\Bigg|\le C,
\end{gather*}
for all $\nu\in D_0$. Then we have
\begin{gather}
\big|(\xi q^{\nu})^\kappa\Phi(\xi q^\nu)\Delta(\xi q^\nu)\big| \le C \left|\prod_{i=1}^n (q^{\nu_i})^{\kappa_i+\alpha+2(n-i)\tau}\right|
=C\prod_{i=1}^n \big|q^{\kappa_i+\alpha}t^{2(n-i)}\big|^{\nu_i}
\nonumber\\
\hphantom{\big|(\xi q^{\nu})^\kappa\Phi(\xi q^\nu)\Delta(\xi q^\nu)\big|}{} \le C\prod_{i=1}^n \big|q^{\alpha}t^{2(n-i)}\big|^{\nu_i}
= C\prod_{i=1}^n \big|\big(q^{\alpha}t^{n-i}\big)^{n-i+1}\big|^{\nu_i-\nu_{i-1}},\label{eq:|zmuPD|}
\end{gather}
where $\nu_0=0$. Thus, from \eqref{eq:Sigma|mPD|} and \eqref{eq:|zmuPD|}, we obtain
\begin{gather}
\sum_{\nu\in D_0}\big|m_\lambda(\xi q^{\nu})\Phi(\xi q^\nu)\Delta(\xi q^\nu)\big| \le C|\frak{S}_n\lambda| \sum_{\nu\in D_0}\prod_{i=1}^n \big|\big(q^{\alpha}t^{n-i}\big)^{n-i+1}\big|^{\nu_i-\nu_{i-1}}
\nonumber\\
\hphantom{\sum_{\nu\in D_0}\big|m_\lambda(\xi q^{\nu})\Phi(\xi q^\nu)\Delta(\xi q^\nu)\big|}{} \le C|\frak{S}_n\lambda|\prod_{i=1}^n\frac{1}{1-\big|(q^{\alpha}t^{n-i})^{n-i+1}\big|} \label{eq:|mPD|}
\end{gather}
under the condition
\begin{gather*}
\big|q^{\alpha}t^{n-i}\big|<1 \qquad\mbox{for}\quad i=1,2,\ldots,n,
\end{gather*}
which is equivalent to \eqref{eq:condition:converge02}. Since the left-hand side of \eqref{eq:|mPD|} with $\xi=\sigma^{-1}a_\mu$ converges for each $\sigma\in \frak{S}_n$, by \eqref{eq:|<m, a>|}, the infinite sum $\langle m_\lambda,a_\mu\rangle$ of \eqref{eq:<m, a>} converges absolutely. Moreover, under the condition~\eqref{eq:condition:converge02}, from~\eqref{eq:|<m, a>|} and~\eqref{eq:|mPD|}, we have
\begin{gather*}
T_\alpha^N\left|\frac{\langle m_\lambda,a_\mu\rangle} {(1-q)^nm_\lambda(a_\mu)\Phi(a_\mu)\Delta(a_\mu)}-1\right| \le T_\alpha^N
\sum_{\nu\in \Lambda_\mu\setminus\{\bf 0\}}\left| \frac{m_\lambda(a_\mu q^{\nu})\Phi(a_\mu q^\nu)\Delta(a_\mu q^\nu)}{m_\lambda(a_\mu)\Phi(a_\mu)\Delta(a_\mu)}\right| \\
\qquad {} = \sum_{\nu\in \Lambda_\mu\setminus\{\bf 0\}}\prod_{i=1}^n|q|^{\nu_i N}\left| \frac{m_\lambda(a_\mu q^{\nu})\Phi(a_\mu q^\nu)\Delta(a_\mu q^\nu)}{m_\lambda(a_\mu)\Phi(a_\mu)\Delta(a_\mu)}\right|\\
\qquad{} \le \frac{C|\frak{S}_n\lambda|}{|m_\lambda(a_\mu)\Phi(a_\mu)\Delta(a_\mu)|}\sum_{\sigma\in \frak{S}_n}\big| U_\sigma(a_\mu)^{-1}\big|\sum_{\nu\in D_0\setminus\{\bf 0\}}\prod_{i=1}^n \big|\big(q^{\alpha+N}t^{n-i}\big)^{n-i+1}\big|^{\nu_i-\nu_{i-1}}\\
\qquad{} \le \frac{C |\frak{S}_n\lambda|}{|m_\lambda(a_\mu)\Phi(a_\mu)\Delta(a_\mu)|} \sum_{\sigma\in \frak{S}_n}\big| U_\sigma(a_\mu)^{-1}\big|\sum_{\nu\in D_0\setminus\{\bf 0\}}\prod_{i=1}^n |q|^{N(n-i+1)(\nu_i-\nu_{i-1})}
\\
\qquad {} \le\left(\frac{C |\frak{S}_n\lambda|}{|m_\lambda(a_\mu)\Phi(a_\mu)\Delta(a_\mu)|} \sum_{\sigma\in \frak{S}_n}\big| U_\sigma(a_\mu)^{-1}\big|\right) \left(-1+ \prod_{i=1}^n \frac{1}{1-|q|^{N(n-i+1)}}\right),
\end{gather*}
which vanishes if $N\to \infty$. This implies \eqref{eq:asympt <m, a>}.
\end{proof}

\section[The Lagrange interpolation polynomials of type $A$]{The Lagrange interpolation polynomials of type $\boldsymbol{A}$}\label{Section4}

In this section we introduce a class of interpolation polynomials of type $A$ which correspond to the elliptic Lagrange interpolation functions discussed in Section~\ref{Section2}. These polynomials play essential roles in constructing the $q$-difference equations satisfied by the Jackson integrals in the succeeding sections. Here we enumerate some properties of the interpolation polynomials without proof (see Appendix~\ref{SectionB} for their detailed proofs).

Let $\mathsf{H}_{s,n}^z$ be the $\mathbb{C}$-linear subspace of $\mathbb{C}[z_1,\ldots,z_n]$ consisting of all $\frak{S}_n$-invariant polyno\-mials~$f(z)$ such that $\deg_{z_i}f(z)\le s$ for $i=1,\ldots,n$:
\begin{gather}\label{eq:Hsn-z}
\mathsf{H}_{s,n}^z=\big\{f(z)\in \mathbb{C}[z_1,\ldots,z_n]^{\frak{S}_n}\,|\, \deg_{z_i}f(z)\le s \ \text{for} \ i=1,\ldots,n\big\}.
\end{gather}
Note that $\dim_\mathbb{C}\mathsf{H}_{s,n}^z={n+s\choose n}$. For $\mu=(\mu_0,\mu_1,\ldots,\mu_s)\in Z_{s+1,n}$ and $x=(x_1,\ldots,x_s)\in (\mathbb{C}^*)^s$ we denote by $x^{\mu_0}_{\mu_1,\ldots,\mu_s}$ the point in~$(\mathbb{C}^*)^n$ defined as
\begin{gather}\label{eq:zeta-mu}
x^{\mu_0}_{\mu_1,\ldots,\mu_s}=\big( \underbrace{z_1,z_2,\ldots,z_{\mu_0}\vphantom{\Big|}}_{\mu_0}, \underbrace{x_1,x_1 t,\ldots,x_1 t^{\mu_1-1}\vphantom{\Big|}}_{\mu_1},
\ldots, \underbrace{x_s,x_s t,\ldots,x_s t^{\mu_s-1}\vphantom{\Big|}}_{\mu_s} \big) \in (\mathbb{C}^*)^n,
\end{gather}
leaving the first $\mu_0$ variables unspecialized.

\begin{Theorem}\label{thm:InterpolationP}
For generic $x=(x_1,\ldots,x_s)\in (\mathbb{C}^*)^s$ there exists a unique $\mathbb{C}$-basis
\begin{gather}\label{eq:InterpolationP}
\big\{\mathsf{E}_{\lambda_1,\ldots,\lambda_s}^{\lambda_0}(x;z)\,|\,(\lambda_0,\lambda_1,\ldots,\lambda_s)\in Z_{s+1,n}\big\}
\end{gather}
of $\mathsf{H}_{s,n}^z$ satisfying the following conditions:
\begin{itemize}\itemsep=0pt
\item[{\rm (0)}]
$\mathsf{E}^{\lambda_0}_{\lambda_1,\ldots,\lambda_s}\big(x;x^{\lambda_0}_{\lambda_1,\ldots,\lambda_s}\big) =\prod_{i=1}^s \prod_{j=1}^{\lambda_0}\big(z_j-x_i t^{\lambda_i}\big)$,
\item[{\rm (1)}] $\mathsf{E}^{\lambda_0}_{\lambda_1,\ldots,\lambda_s}\big(x;x^{\lambda_0}_{\mu_1,\ldots,\mu_s}\big)=0$ if $(\lambda_1,\ldots,\lambda_s)\ne (\mu_1,\ldots,\mu_s)$,
\item[{\rm (2)}] $\mathsf{E}^{\lambda_0}_{\lambda_1,\ldots,\lambda_s}\big(x;x^{\mu_0}_{\mu_1,\ldots,\mu_s}\big)=0$ if $0\le \mu_0 <\lambda_0$.
\end{itemize}
\end{Theorem}

\begin{proof}See Theorem \ref{thm:tri-E} in Appendix \ref{SectionB}.
\end{proof}

We call $\mathsf{E}_{\lambda_1,\ldots,\lambda_s}^{\lambda_0}(x;z)$, $ \lambda\in Z_{s+1,n}$, the {\it Lagrange interpolation polynomials of type~$A$}.

\begin{remark*}
The polynomial $\mathsf{E}^{\lambda_0}_{\lambda_1,\ldots,\lambda_s}(x;z)$ may be expanded as
\begin{gather*}
\mathsf{E}^{\lambda_0}_{\lambda_1,\ldots,\lambda_s}(x;z) =C^{\lambda_0}_{\lambda_1,\ldots,\lambda_s}m_{((s-1)^{n-\lambda_0}s^{\lambda_0})}(z)+\mbox{lower order terms},
\end{gather*}
where the leading term $m_{((s-1)^{n-\lambda_0}s^{\lambda_0})}(z)$ is the monomial symmetric polynomial of type $((s-1)^{n-\lambda_0}s^{\lambda_0}) =(\underbrace{s,\ldots,s}_{\lambda_0}, \underbrace{s-1,\ldots,s-1}_{n-\lambda_0})\in P_+$ and the remaining terms are of lower order with respect to the lexicographic order of~$P_+$. The coefficient $C^{\lambda_0}_{\lambda_1,\ldots,\lambda_s}$ of the leading term is given by
\begin{gather*}
C^{\lambda_0}_{\lambda_1,\ldots,\lambda_s}= \frac{(t;t)_{n-\lambda_0}}{\prod\limits_{i=1}^s (t;t)_{\lambda_i}}\prod_{i=1}^s
\prod_{\substack{1\le j\le s\\ j\ne i}}\prod_{k=1}^{\lambda_i}\frac{1}{x_it^{k-1}-x_jt^{\lambda_j}},
\end{gather*}
where $(a;t)_\nu=(1-a)(1-ta)\cdots \big(1-t^{\nu-1}a\big)$. The above result will be proved in Lemma~\ref{lem:c of l.t.} in Appendix~\ref{SectionB}.
\end{remark*}

\begin{Example}If $s=1$, then the polynomials $\mathsf{E}^{\lambda_0}_{\lambda_1}(x_1;z)$ are explicitly given by
\begin{gather*}
\mathsf{E}^{r}_{n-r}(x_1;z)=\sum_{1\le i_1<\cdots <i_r\le n} \prod_{k=1}^r \big(z_{i_k}-x_1t^{i_k-k}\big) \qquad \text{for} \quad r=0,1,\ldots,n.
\end{gather*}
In the limit as $x_1\to 0$, $\mathsf{E}^{r}_{n-r}(x_1;z)$ gives the $r$th elementary symmetric function:
\begin{gather*}
\mathsf{E}^{r}_{n-r}(x_1;z)\big|_{x_1\to 0}=\sum_{1\le i_1<\cdots <i_r\le n} z_{i_1}\cdots z_{i_r} \qquad \text{for} \quad r=0,1,\ldots,n.
\end{gather*}
\end{Example}

\begin{Example} For $(\lambda_1,\ldots,\lambda_s)\in Z_{s,n}$, $\mathsf{E}_{\lambda_1,\ldots,\lambda_s}^{0}(x;z)$ is expressed explicitly as
\begin{gather*}%\label{eq:E-explicit}
\mathsf{E}_{\lambda_1,\ldots,\lambda_s}^{0}(x;z)=
\sum_{\substack{K_1\sqcup\cdots\sqcup K_s \\ =\{1,2,\ldots,n\}} }\prod_{i=1}^s\prod_{k\in K_i}\Bigg[\prod_{\substack{1\le j\le s\\ j\ne i}}\frac{z_{k}-x_jt^{\lambda_j^{(k-1)}}}
{x_it^{\lambda_i^{(k-1)}}-x_jt^{\lambda_j^{(k-1)}}}\Bigg],
\end{gather*}
where $\lambda_i^{(k)}=|K_i\cap\{1,2,\ldots,k\}|$, and the summation is taken over all partitions $K_1\sqcup\cdots\sqcup K_s=\{1,2,\ldots,n\}$ such that $|K_i|=\lambda_i$, $i=1,2,\ldots,s$. (See Example~\ref{ex:B-E-explicit-1-2} in Appendix~\ref{SectionB}.)
\end{Example}

For $x=(x_1,\ldots,x_s)\in (\mathbb{C}^*)^s$ we use the symbol $x_{\widehat{1}}=(x_2,\ldots,x_s)\in (\mathbb{C}^*)^{s-1}$.
\begin{Lemma} For $\lambda=(\lambda_1,\ldots,\lambda_s)\in Z_{s,n}$, the polynomial $\mathsf{E}^{0}_{\lambda_1,\ldots,\lambda_s}(x;z)$ expands in terms of $\mathsf{E}^{\mu_1}_{\mu_2,\ldots,\mu_s}(x_{\widehat{1}};z)$, $\mu\succeq \lambda$, i.e.,
\begin{gather}\label{eq:E0-expansion1}
\mathsf{E}^{0}_{\lambda_1,\ldots,\lambda_s}(x;z)=l_{\lambda\lambda} \mathsf{E}^{\lambda_1}_{\lambda_2,\ldots,\lambda_s}(x_{\widehat{1}};z)+\sum_{\substack{\mu\in Z_{s,n}\\ \lambda\prec\mu}}l_{\mu\lambda}
\mathsf{E}^{\mu_1}_{\mu_2,\ldots,\mu_s}(x_{\widehat{1}};z).
\end{gather}
In particular, the coefficient $l_{\lambda\lambda}$ of the lowest order term can be written as
\begin{gather}\label{eq:E0-expansion2}
l_{\lambda\lambda} =\prod_{i=1}^{\lambda_1}\prod_{j=2}^s \frac{1}{x_1t^{i-1}-x_jt^{\lambda_j}}.
\end{gather}
\end{Lemma}

\begin{proof}Since $\mathsf{E}^{0}_{\lambda_1,\ldots,\lambda_s}(x;z)\in \mathsf{H}_{s-1,n}^z$, $\mathsf{E}^{0}_{\lambda_1,\ldots,\lambda_s}(x;z)$ expands in terms of $\mathsf{E}^{\mu_1}_{\mu_2,\ldots,\mu_s}(x_{\widehat{1}};z)$, $\mu\in Z_{s,n}$;
\begin{gather}\label{eq:E0-expansion3}
\mathsf{E}^{0}_{\lambda_1,\ldots,\lambda_s}(x;z)=\sum_{\mu\in Z_{s,n}} l_{\mu\lambda} \mathsf{E}^{\mu_1}_{\mu_2,\ldots,\mu_s}(x_{\widehat{1}};z).
\end{gather}
From the conditions (0), (1) and (2) of Theorem \ref{thm:InterpolationP}, we have
\begin{gather}\label{eq:E0-expansion4}
\mathsf{E}^{0}_{\lambda_1,\ldots,\lambda_s}\big(x;x^{0}_{\nu_1,\ldots,\nu_s}\big)=\delta_{\lambda\nu}.
\end{gather}
On the other hand, from the conditions (1) and (2) of Theorem \ref{thm:InterpolationP}, for $\mu,\nu\in Z_{s,n}$, if (1) $\mu_1=\nu_1$ and $(\mu_2,\ldots,\mu_s)\ne(\nu_2,\ldots,\nu_s)$, or~(2) $\mu_1>\nu_1$, then
\begin{gather}\label{eq:E0-expansion5}
\mathsf{E}^{\mu_1}_{\mu_2,\ldots,\mu_s}\big(x_{\widehat{1}};x^{0}_{\nu_1,\ldots,\nu_s}\big) =\mathsf{E}^{\mu_1}_{\mu_2,\ldots,\mu_s}\big(x_{\widehat{1}};x^{\nu_1}_{\nu_2,\ldots,\nu_s}\big)
\big|_{z_i=x_1t^{i-1}\, (i=1,\ldots,\nu_1)}=0
\end{gather}
Moreover, from the condition (0) of Theorem \ref{thm:InterpolationP}, we have
\begin{gather}
\label{eq:E0-expansion6}
\mathsf{E}^{\nu_1}_{\nu_2,\ldots,\nu_s}\big(x_{\widehat{1}};x^{0}_{\nu_1,\ldots,\nu_s}\big) =\prod_{i=2}^s\prod_{j=1}^{\nu_1}\big(x_1t^{j-1}-x_it^{\nu_i}\big)\ne 0,
\end{gather}
for generic $x$. We now suppose that $\nu\prec\lambda$. Applying (\ref{eq:E0-expansion4}), (\ref{eq:E0-expansion5}) and (\ref{eq:E0-expansion6}) to (\ref{eq:E0-expansion3}), we have
\begin{gather*}
0 =l_{\nu\lambda} \mathsf{E}^{\nu_1}_{\nu_2,\ldots,\nu_s}\big(x_{\widehat{1}};x^{0}_{\nu_1,\ldots,\nu_s}\big)+\sum_{\substack{\mu\in Z_{s,n}\\ \mu_1<\nu_1}}
l_{\mu\lambda}\mathsf{E}^{\mu_1}_{\mu_2,\ldots,\mu_s}\big(x_{\widehat{1}};x^{0}_{\nu_1,\ldots,\nu_s}\big) \qquad \text{for} \quad \nu\prec\lambda.
\end{gather*}
Starting from the minimal element $\nu=(0,\ldots,0,n)$ of the lexicographic order of $\nu$, and solving these simultaneous equations for $\{l_{\nu\lambda}\,|\,\nu\prec\lambda\}$ inductively we have $l_{\nu\lambda}=0$ for all $\nu\prec\lambda$. Therefore we have the expression~(\ref{eq:E0-expansion1}). From~(\ref{eq:E0-expansion4}) and~(\ref{eq:E0-expansion6}) the coefficient
\begin{gather*}
l_{\lambda\lambda}=\mathsf{E}^{0}_{\lambda_1,\ldots,\lambda_s}\big(x;x^{0}_{\lambda_1, \ldots,\lambda_s}\big)/\mathsf{E}^{\lambda_1}_{\lambda_2,\ldots,\lambda_s}\big(x_{\widehat{1}};x^{0}_{\lambda_1,\ldots,\lambda_s}\big)
\end{gather*}
coincides with (\ref{eq:E0-expansion2}).
\end{proof}

\begin{Corollary}\label{cor:E=EL} We have
\begin{gather}\label{eq:E=EL}
\big(\mathsf{E}^{0}_{\lambda_1,\ldots,\lambda_s}(x;z)\big)_{\lambda\in Z_{s,n}}= \big(\mathsf{E}^{\lambda_1}_{\lambda_2,\ldots,\lambda_s}(x_{\widehat{1}};z)\big)_{\lambda\in Z_{s,n}}L,
\end{gather}
where $L=(l_{\mu\nu})$ is a lower triangular matrix of size ${s+n-1\choose n}$.
\end{Corollary}

\begin{Lemma}\label{lem:EE=E**}For $\lambda=(\lambda_1,\lambda_2,\ldots,\lambda_s)\in Z_{s,n}$, we have
\begin{gather*}
\mathsf{E}^{n}_0(x_1;z) \mathsf{E}^{\lambda_1}_{\lambda_2,\ldots,\lambda_s}(x_2,\dots,x_s;z)=\mathsf{E}^{\lambda_1}_{0,\lambda_2,\ldots,\lambda_s}
(x_1,x_2,\dots,x_s;z)\prod_{i=2}^s\prod_{k=1}^{\lambda_i}\big(x_it^{k-1}-x_1\big).
\end{gather*}
\end{Lemma}

\begin{proof}See Lemma \ref{lem:B-E=EE} in Appendix \ref{SectionB}.
\end{proof}

We also use the symbol $\mathsf{E}_{\lambda_1,\ldots,\lambda_s}^{\lambda_0,(n)}(x;z)$ and $x^{\lambda_0,(n)}_{\lambda_1,\ldots,\lambda_s}$ instead of $\mathsf{E}_{\lambda_1,\ldots,\lambda_s}^{\lambda_0}(x;z)$ and
$x^{\lambda_0}_{\lambda_1,\ldots,\lambda_s}$, respectively, to specify the number $n$ of $z$ variables.

\begin{Lemma}\label{lem:vanishing}Suppose that the variables $z_1,\ldots,z_{\mu_0}$ in $x^{\mu_0,(n)}_{\mu_1,\ldots,\mu_s}$ satisfy
\begin{gather}\label{eq>>>}
|z_1|\gg |z_2|\gg \cdots\gg |z_{\mu_0}|\gg 0,
\end{gather}
which means $z_1/z_2\to \infty, z_2/z_3\to \infty, \ldots, z_{\mu_0-1}/z_{\mu_0}\to \infty$ and $z_{\mu_0}\to \infty$. Then, asymptotically, we have
\begin{gather}\label{eq:E-vanishing2}
\frac{\mathsf{E}^{\lambda_0,(n)}_{\lambda_1,\ldots,\lambda_s}(x;z)\Delta^{\!(n)}(z)}{z_1^{s+n-1}z_2^{s+n-2}\cdots z_{\mu_0}^{s+n-\mu_0}}\bigg|_{z=x^{\mu_0,(n)}_{\mu_1,\ldots,\mu_s}}\sim
\delta_{\lambda\mu}\Delta^{\!(n-\lambda_0)}\big(x^{0,(n-\lambda_0)}_{\lambda_1,\ldots,\lambda_s}\big),
\end{gather}
where $\delta_{\lambda\mu}$ is the Kronecker delta for $\lambda,\mu\in Z_{s+1,n}$ and $\Delta^{\!(k)}(z_1,\ldots,z_k)$ denotes the difference product $\Delta(z)$ of $k$ variables.
\end{Lemma}

\begin{proof} From Theorem \ref{thm:InterpolationP}, if (1) $\lambda_0=\mu_0$ and $(\lambda_1,\ldots,\lambda_s)\ne (\mu_1,\ldots,\mu_s)$, or if (2) $\lambda_0>\mu_0$, then the left-hand side of (\ref{eq:E-vanishing2}) is exactly equal to~0. On the other hand, if $\lambda_0<\mu_0$, then the degree of $\mathsf{E}^{\lambda_0,(n)}_{\lambda_1,\ldots,\lambda_s}\big(x;x^{\mu_0,(n)}_{\mu_1,\ldots,\mu_s}\big) \Delta^{\!(n)}\big(x^{\mu_0,(n)}_{\mu_1,\ldots,\mu_s}\big)$ is lower than $z_1^{s+n-1}z_2^{s+n-2}\cdots z_{\mu_0}^{s+n-\mu_0}$, so that the left-hand side of~(\ref{eq:E-vanishing2}) is estimated as $0$ under the condition~(\ref{eq>>>}). If $\lambda_0=\mu_0$ and $(\lambda_1,\ldots,\lambda_s)=(\mu_1,\ldots,\mu_s)$, from the condition (0) of Theorem~\ref{thm:InterpolationP}, (\ref{eq:Delta}) and~(\ref{eq:zeta-mu}), we have
\begin{gather*}
\mathsf{E}^{\lambda_0,(n)}_{\lambda_1,\ldots,\lambda_s}\big(x;x^{\lambda_0,(n)}_{\lambda_1,\ldots,\lambda_s}\big) \Delta^{\!(n)}\big(x^{\lambda_0,(n)}_{\lambda_1,\ldots,\lambda_s}\big)\\
\qquad{}= \Delta^{\!(n-\lambda_0)}\big(x^{0,(n-\lambda_0)}_{\lambda_1,\ldots,\lambda_s}\big)\Delta^{\!(\lambda_0)}(z_1,\ldots,z_{\lambda_0})
\prod_{i=1}^s\prod_{j=1}^{\lambda_0}\prod_{k_i=0}^{\lambda_i}\big(z_j-x_it^{k_i}\big),
\end{gather*}
which implies (\ref{eq:E-vanishing2}).
\end{proof}

\section[Difference equations with respect to $\alpha$]{Difference equations with respect to $\boldsymbol{\alpha}$}\label{Section5}

The aim of this section is to give a proof of the following proposition.
\begin{Proposition}\label{prop:q-difference}
For $\lambda=(\lambda_1,\lambda_2,\ldots,\lambda_{s})\in Z_{s,n}$ and $z=(z_1,\ldots,z_n)\in (\mathbb{C}^*)^n$, we set
\begin{gather*}
e_\lambda(z)=\mathsf{E}^{\lambda_1}_{\lambda_2,\ldots,\lambda_s}(a_{\widehat{1}};z)=\mathsf{E}^{\lambda_1}_{\lambda_2,\ldots,\lambda_{s}}(a_2,\ldots,a_s;z)
\end{gather*}
using the notation of \eqref{eq:InterpolationP}. Let $T_{\alpha}$ be the shift operator $\alpha\to\alpha+1$, i.e., $T_{\alpha}f(\alpha)=f(\alpha+1)$. Then the row vector $\big(\langle e_\lambda,z\rangle\big)_{\!\lambda\in Z_{s,n}}$ of Jackson integrals satisfies the difference system in the form
\begin{gather}\label{eq:q-diff01}
T_{\alpha}\big(\langle e_\lambda,z\rangle\big)_{\!\lambda\in Z_{s,n}} =\big(\langle e_\lambda,z\rangle\big)_{\!\lambda\in Z_{s,n}}LU,
\end{gather}
where $L=(l_{\lambda\mu})_{\lambda,\mu\in Z_{s,n}}$ and $U=(u_{\lambda\mu})_{\lambda,\mu\in Z_{s,n}}$ are lower and upper triangular matrices of size ${s+n-1\choose n}$ with respect to the lexicographic ordering $\preceq$, respectively. The diagonal entries of~$L$ and~$U$ are given by
\begin{gather}
l_{\lambda\lambda} =(-1)^{(s-1)\lambda_1}t^{(\lambda_1-n)\lambda_1}(a_2a_3\cdots a_s)^{-\lambda_1},\nonumber\\ %\label{eq:diag-L}\\
u_{\lambda\lambda}=(-1)^{(s+1)\lambda_1}\prod_{k=1}^{\lambda_1}\frac{\big(1-q^\alpha t^{k-1}\big)t^{n-k}\prod\limits_{i=1}^sa_i}{1-q^\alpha t^{2n-2-\lambda_1+k}\prod\limits_{i=1}^s a_ib_i}\times \prod_{i=2}^s\prod_{k=1}^{\lambda_i}\big(a_it^{k-1}\big),\label{eq:diag-U}
\end{gather}
and the determinant of the coefficient matrix is
\begin{gather}\label{eq:detLU}
\det LU=\prod_{\lambda\in Z_{s,n}}l_{\lambda\lambda}u_{\lambda\lambda}=\prod_{k=1}^n\left[\frac{\big(1-q^{\alpha}t^{n-k}\big)t^{k-1}\prod\limits_{i=1}^s a_i}{1-q^{\alpha}t^{n+k-2}\prod\limits_{i=1}^{s} a_ib_i}\right]^{s+k-2\choose k-1}.
\end{gather}
\end{Proposition}

\begin{Corollary}\label{cor:q-difference}The difference equation for $J =\det\big(\langle s_\lambda, a_\mu \rangle\big)_{\!\mu\in Z \atop \!\lambda\in B}$ with respect to $\alpha$ is given by
\begin{gather}\label{eq:q-diff-det}
T_{\alpha} J =J\prod_{k=1}^n\left[\frac{\big(1-q^{\alpha}t^{n-k}\big)t^{k-1}\prod\limits_{i=1}^s a_i}{1-q^{\alpha}t^{n+k-2}\prod\limits_{i=1}^{s} a_ib_i}\right]^{s+k-2\choose k-1}.
\end{gather}
\end{Corollary}

\begin{proof}Both of the sets $\{\mathsf{E}^{\lambda_1}_{\lambda_2,\ldots,\lambda_s}(a_{\widehat{1}};z)\,|\, \lambda\in Z_{s,n}\}$ and $\{s_\mu(z)\,|\, \mu\in B_{s,n}\}$ are bases of the $\mathbb{C}$-linear space $\mathsf{H}_{s-1,n}^z$ specified by~(\ref{thm:InterpolationP}). Thus $\mathsf{E}^{\lambda_1}_{\lambda_2,\ldots,\lambda_s}(a_{\widehat{1}};z)$ can be expanded in terms of $s_\mu(z)$, $\mu\in B_{s,n}$, so that
\begin{gather*}
\big(\mathsf{E}^{\lambda_1}_{\lambda_2,\ldots,\lambda_s}(a_{\widehat{1}};z) \big)_{\!\lambda\in Z_{s,n}}=\big(s_\mu(z)\big)_{\!\mu\in B_{s,n}}P
\end{gather*}
with an invertible matrix $P=P(a_{\widehat{1}})$, and hence
\begin{gather*}
\big(\langle e_\lambda,z\rangle\big)_{\!\lambda\in Z_{s,n}} =\big(\langle s_\mu,z\rangle\big)_{\!\mu\in B_{s,n}}P.
\end{gather*}
Then, the difference system (\ref{eq:q-diff01}) is transformed into
\begin{gather*}%\label{eq:q-diff01--}
T_{\alpha}\big(\langle s_\mu,z\rangle\big)_{\!\mu\in B_{s,n}}=\big(\langle s_\mu,z\rangle\big)_{\!\mu\in B_{s,n}}PLUP^{-1}.
\end{gather*}
This implies that
\begin{gather*}
T_{\alpha} J =J \det LU= J\prod_{k=1}^n\left[\frac{\big(1-q^{\alpha}t^{n-k}\big)t^{k-1}\prod\limits_{i=1}^s a_i}{1-q^{\alpha}t^{n+k-2} \prod\limits_{i=1}^{s} a_ib_i}\right]^{s+k-2\choose k-1}
\end{gather*}
by (\ref{eq:detLU}).
\end{proof}

In the remaining part of this section we complete the proof of Proposition~\ref{prop:q-difference}. We first explain a fundamental idea for deriving difference equations. Let $\Phi(z)$ be the function defined by~(\ref{eq:Phi}) as before. For each $i=1,\ldots,n$ we introduce the operator $\nabla_{\!i}$ by setting
\begin{gather}\label{eq:def-nabla}
(\nabla_{\!i}\varphi)(z)=\varphi(z)-\frac{T_{q,z_i}\Phi(z)}{\Phi(z)}T_{q,z_i}\varphi(z)
\end{gather}
for a function $\varphi(z)$, where $T_{q,z_i}$ stands for the $q$-shift operator $z_i\to qz_i$, i.e., $T_{q,z_i}f(\ldots,z_i,\ldots)=f(\ldots,qz_i,\ldots)$. We remark that the ratio $T_{q,z_i}\Phi(z)/\Phi(z)$ is expressed as
\begin{gather}\label{eq:b-function}
\frac{T_{q,z_i}\Phi(z)}{\Phi(z)}=\frac{F_i(z)}{T_{q,z_i}G_i(z)},
\end{gather}
where
\begin{gather}
F_i(z)=q^\alpha\prod_{m=1}^{s}(1-b_mz_i)\prod_{\substack{1\le j\le n\\ j\ne i}}(z_j-tz_i),\nonumber\\
G_i(z)=\prod_{m=1}^{s}\big(1-a_m^{-1}z_i\big)\prod_{\substack{1\le j\le n\\ j\ne i}}\big(z_j-t^{-1}z_i\big).\label{eq:Fj}
\end{gather}

\begin{Lemma}[key lemma]\label{lem:00nabla=0} For a point $\xi\in ({\mathbb C}^*)^n$ and a meromorphic function $\varphi(z)$ on $({\mathbb C}^*)^n$, if the Jackson integral
\begin{gather*}
\int_0^{\xi\infty}\varphi(z)\Phi(z)\varpi_q
\end{gather*}
converges, then
\begin{gather}\label{eq:00nabla=0}
\int_0^{\xi\infty}\Phi(z)\nabla_{\!i}\varphi(z)\varpi_q=0 \qquad \text{for} \quad i=1,\ldots,n.
\end{gather}
Moreover, if the Jackson integrals
\begin{gather*}
\int_0^{\sigma\xi\infty}\varphi(z)\Phi(z)\varpi_q
\end{gather*}
converge for all $\sigma\in \frak{S}_n$, then
\begin{gather}\label{eq:00A}
\int_0^{\xi\infty}\Phi(z){\cal A}\nabla_{\!i}\varphi(z)\varpi_q=0 \qquad \text{for} \quad i=1,\ldots,n,
\end{gather}
where ${\cal A}$ indicates the skew-symmetrization defined in \eqref{eq:00Af}.
\end{Lemma}

The statement \eqref{eq:00nabla=0} is equivalent to the $q$-shift invariance of the Jackson integral
\begin{gather*}
\int_0^{\xi\infty}\Phi(z)\varphi(z)\varpi_q=\int_0^{\xi\infty}T_{q,z_i}\big(\Phi(z)\varphi(z)\big)\varpi_q,
\end{gather*}
and \eqref{eq:00A} follows from the quasi-symmetry \eqref{eq:Usigma-01} of $\Phi(z)$.

For $\lambda=(\lambda_0,\lambda_1,\ldots,\lambda_s)\in Z_{s+1,n}$ we define
\begin{gather}\label{eq:tilde-e}
\tilde{e}_{\lambda}(a_{\widehat{1}};z)=\mathsf{E}^{\lambda_0}_{\lambda_1,\ldots,\lambda_s}(\varepsilon,a_2,\dots,a_s;z)\big|_{\varepsilon\to 0}.
\end{gather}

\begin{Lemma}\label{lem:Anabla}For $\lambda=(\lambda_0,\lambda_1,\ldots,\lambda_s)\in Z_{s+1,n}$ let $\tilde{e}_{\lambda}(z)=\tilde{e}_{\lambda}(a_{\widehat{1}};z)$ be functions of $z\in (\mathbb{C}^*)^n$ defined by \eqref{eq:tilde-e}. For $\lambda=(\lambda_0,\lambda_1,\ldots,\lambda_s)\in Z_{s+1, n-1}$ it follows that
\begin{gather}
 {\cal A}\big[\nabla_{\!1}\big(G_1(z)\tilde{e}^{(n-1)}_{\lambda}(z_{\widehat{1}})\Delta^{(n-1)}(z_{\widehat{1}})\big)\big]\nonumber\\
 \qquad{}= \left(c_0\tilde{e}^{(n)}_{\lambda+\epsilon_0}(z)+c_1\tilde{e}^{(n)}_{\lambda+\epsilon_1}(z)+\sum_{k=2}^{s} c_k\tilde{e}^{(n)}_{\lambda+\epsilon_k}(z)\right)\Delta^{(n)}(z),\label{eq:AnablaF1eD}
\end{gather}
where the coefficients $c_0$, $c_1$ and $c_k$, $k=2,\ldots,s$, are given by
\begin{gather}
c_0=(n-1)!(-1)^{n+s-1}\frac{\big(1-t^{\lambda_0+1}\big)\Big(1-q^\alpha t^{2n-2-\lambda_0}\prod\limits_{i=1}^sa_ib_i\Big)}{(1-t)t^{n-1}\prod\limits_{i=1}^sa_i},\label{eq:c0}\\
c_1=(n-1)!(-1)^{n-1}\frac{\big(1-t^{\lambda_1+1}\big)\big(1-q^\alpha t^{\lambda_{1}}\big)}{(1-t)t^{\lambda_{1}}},\label{eq:c1}\\
c_k=(n-1)!(-1)^{n}q^\alpha t^{n-1}\frac{\big(1-t^{\lambda_k+1}\big)\prod\limits_{m=1}^{s}\big(1-b_m a_kt^{\lambda_{k}}\big)}{(1-t)t^{\lambda_{0}+\lambda_{k}}}\prod_{\substack{2\le j\le s\\ j\ne k}}\frac{a_j-a_kt^{\lambda_k+1}} {a_jt^{\lambda_j}-a_kt^{\lambda_k+1}}.\label{eq:ck}
\end{gather}
\end{Lemma}

\begin{proof}See Appendix~\ref{SectionA}.
\end{proof}

\begin{Lemma}[fundamental relations]For $\lambda=(\lambda_0,\lambda_1,\ldots,\lambda_s)\in Z_{s+1,n}$ let $\tilde{e}_{\lambda}(z) =\tilde{e}_{\lambda}(a_{\widehat{1}};z)$ be functions of $z\in (\mathbb{C}^*)^n$ defined by~\eqref{eq:tilde-e}. Then, for an arbitrary $\lambda\in Z_{s+1, n-1}$ it follows that
\begin{gather} \label{eq:recursion-tilde e}
\langle \tilde{e}_{\lambda+\epsilon_0},z\rangle=C_{\lambda+\epsilon_1}^{\lambda+\epsilon_0}\langle \tilde{e}_{\lambda+\epsilon_1},z\rangle+\sum_{k=2}^s C_{\lambda+\epsilon_k}^{\lambda+\epsilon_0}
\langle \tilde{e}_{\lambda+\epsilon_k},z\rangle,
\end{gather}
where the coefficients $C_{\lambda+\epsilon_1}^{\lambda+\epsilon_0}$ and $C_{\lambda+\epsilon_k}^{\lambda+\epsilon_0}$, $k=2,\ldots,s$, are given by
\begin{gather}
C_{\lambda+\epsilon_1}^{\lambda+\epsilon_0}=C^{(\lambda_0+1,\lambda_1,\lambda_2,\ldots,\lambda_s)}
_{(\lambda_0,\lambda_1+1,\lambda_2,\ldots,\lambda_s)}
=(-1)^{s+1}\frac{\big(1-t^{\lambda_1+1}\big)\big(1-q^\alpha t^{\lambda_{1}}\big)t^{n-1-\lambda_1}\prod\limits_{i=1}^s a_i}
{\big(1-t^{\lambda_0+1}\big)\Big(1-q^\alpha t^{2n-2-\lambda_0}\prod\limits_{i=1}^s a_ib_i\Big)},\nonumber\\
C_{\lambda+\epsilon_k}^{\lambda+\epsilon_0}=(-1)^s\frac{q^\alpha t^{2n-2}\big(1-t^{\lambda_k+1}\big)\prod\limits_{m=1}^{s}a_m(1-b_m a_k t^{\lambda_{k}})}
{t^{\lambda_0+\lambda_k}\big(1-t^{\lambda_0+1}\big)\Big(1-q^\alpha t^{2n-2-\lambda_0}\prod\limits_{i=1}^s a_ib_i\Big)} \prod_{\substack{2\le j\le s\\ j\ne k}}\frac{a_j-a_kt^{\lambda_k+1}}{a_jt^{\lambda_j}-a_kt^{\lambda_k+1}}.\label{eq:c01}
\end{gather}
\end{Lemma}

\begin{proof}Applying Lemma \ref{lem:Anabla} to (\ref{eq:00A}) of Lemma \ref{lem:00nabla=0}, we have
\begin{gather*}
c_0\langle \tilde{e}_{\lambda+\epsilon_0},z\rangle +c_1\langle \tilde{e}_{\lambda+\epsilon_1},z\rangle+\sum_{k=2}^s c_k\langle \tilde{e}_{\lambda+\epsilon_k},z\rangle=0,
\end{gather*}
where $c_0$, $c_1$ and $c_k$, $k=2,\ldots,s$, are given by (\ref{eq:c0}), (\ref{eq:c1}) and (\ref{eq:ck}), respectively. Therefore the coefficients in~(\ref{eq:recursion-tilde e}) are given by $C^{\lambda+\epsilon_0}_{\lambda+\epsilon_1}=-c_1/c_0$, and $C^{\lambda+\epsilon_0}_{\lambda+\epsilon_k}=-c_k/c_0$, $k=2,\ldots,s$, which are expressed as~(\ref{eq:c01}), under the condition $\lambda\in Z_{s+1, n-1}$, i.e.,
$\lambda_0+\lambda_1+\cdots+\lambda_s=n-1$.
\end{proof}

\begin{Lemma}\label{lem:E=EU}For $\lambda=(\lambda_1,\lambda_2,\ldots,\lambda_s)\in Z_{s,n}$ the Jackson integral $\langle \tilde{e}_{(\lambda_1,0,\lambda_2,\ldots,\lambda_s)},z\rangle$ expands in terms of
$\langle \tilde{e}_{(0,\mu_1,\mu_2,\ldots,\mu_s)},z\rangle$, $\mu\in Z_{s,n}$, as
\begin{gather}\label{eq:expansion-U}
\langle \tilde{e}_{(\lambda_1,0,\lambda_2,\ldots,\lambda_s)},z\rangle =\tilde{u}_{\lambda\lambda} \langle \tilde{e}_{(0,\lambda_1,\lambda_2,\ldots,\lambda_s)},z\rangle
+\sum_{\substack{\mu\in Z_{s,n}\\ \mu\prec\lambda}}\tilde{u}_{\mu\lambda}\langle \tilde{e}_{(0,\mu_1,\mu_2,\ldots,\mu_s)},z\rangle,
\end{gather}
so that
\begin{gather}\label{eq:q-diff02}
\big(\langle \tilde{e}_{(\lambda_1,0,\lambda_2,\ldots,\lambda_s)},z\rangle \big)_{\!\lambda\in Z_{s,n}}=\big(\langle \tilde{e}_{(0,\mu_1,\mu_2,\ldots,\mu_s)},z\rangle\big)_{\!\mu\in Z_{s,n}}\tilde{U},
\end{gather}
where the matrix $\tilde{U}=\big(\tilde{u}_{\mu\nu}\big)_{\mu,\nu\in Z_{s,n}}$ is upper triangular with respect to~$\preceq$. The diagonal entries of $\tilde{U}$ are given by
\begin{gather}\label{eq:diag-tildeU}
\tilde{u}_{\lambda\lambda}=(-1)^{(s+1)\lambda_1}\prod_{k=1}^{\lambda_1}\frac{\big(1-q^\alpha t^{k-1}\big)t^{n-k}\prod\limits_{i=1}^sa_i}{1-q^\alpha t^{2n-2-\lambda_1+k}\prod\limits_{i=1}^s a_ib_i}.
\end{gather}
\end{Lemma}

\begin{proof}By repeated use of (\ref{eq:recursion-tilde e}), for $\tilde\lambda=(\lambda_1,0,\lambda_2,\ldots,\lambda_s)\in Z_{s+1,n}$ we have
\begin{gather*}
\langle \tilde{e}_{\tilde\lambda},z\rangle =\sum_{i=1}^s C^{\tilde\lambda}_{\tilde\lambda+\epsilon_i-\epsilon_0}\langle \tilde{e}_{\tilde\lambda+\epsilon_i-\epsilon_0},z\rangle=
\sum_{i_1=1}^s\sum_{i_2=1}^sC^{\tilde\lambda}_{\tilde\lambda+\epsilon_i-\epsilon_0}C^{\tilde\lambda+\epsilon_{i_1}-\epsilon_0}_{\tilde\lambda+\epsilon_{i_1}+\epsilon_{i_2}-2\epsilon_0}
\langle \tilde{e}_{\tilde\lambda+\epsilon_{i_1}+\epsilon_{i_2}-2\epsilon_0},z\rangle\\
\hphantom{\langle \tilde{e}_{\tilde\lambda},z\rangle}{} =\cdots=\sum_{\mu\in D_\lambda}\tilde{u}_{\mu\lambda}\langle \tilde{e}_{(0,\mu)},z\rangle,
\end{gather*}
where $D_\lambda=\{\mu=(\mu_1,\ldots,\mu_s)\in Z_{s,n}\,|\,\mu_i\ge \lambda_i,\, i=2,\ldots,s\}$ and
\begin{gather*}
\tilde{u}_{\mu\lambda} =\sum_{1\le i_1,\ldots,i_{\lambda_1}\le s}\prod_{k=1}^{\lambda_1}C^{\tilde\lambda+\epsilon_{i_1}+\cdots+\epsilon_{i_{k-1}}-(k-1)\epsilon_0}_{\tilde\lambda+\epsilon_{i_1}+\cdots+\epsilon_{i_{k}}-k\epsilon_0}.
\end{gather*}
Since $\lambda=(\lambda_1,\lambda_2,\ldots,\lambda_s)$ is the maximum of $D_\lambda$, we have
\begin{gather*}
\langle \tilde{e}_{\tilde\lambda},z\rangle =\sum_{\mu\preceq \lambda} \tilde{u}_{\mu\lambda}\langle \tilde{e}_{(0,\mu)},z\rangle=\tilde{u}_{\lambda\lambda}\langle \tilde{e}_{(0,\lambda)},z\rangle
+\sum_{\mu\prec \lambda} \tilde{u}_{\mu\lambda}\langle \tilde{e}_{(0,\mu)},z\rangle,
\end{gather*}
which gives (\ref{eq:expansion-U}). The coefficient $\tilde{u}_{\lambda\lambda}$ in (\ref{eq:expansion-U}) is computed as the case $i_1=\cdots=i_{\lambda_1}=1$, namely
\begin{gather*}
\tilde{u}_{\lambda\lambda} =C^{(\lambda_1,0,\lambda_2,\ldots,\lambda_s)}_{(\lambda_1-1,1,\lambda_2,\ldots,\lambda_s)}C^{(\lambda_1-1,1,\lambda_2,\ldots,\lambda_s)}_{(\lambda_1-2,2,\lambda_2,\ldots,\lambda_s)}
\cdots C^{(1,\lambda_1-1,\lambda_2,\ldots,\lambda_s)}_{(0,\lambda_1,\lambda_2,\ldots,\lambda_s)} =\prod_{k=1}^{\lambda_1} C^{(\lambda_1-k+1,k-1,\lambda_2,\ldots,\lambda_s)}_{(\lambda_1-k,k,\lambda_2,\ldots,\lambda_s)}\\
\hphantom{\tilde{u}_{\lambda\lambda}}{} =\prod_{k=1}^{\lambda_1} \frac{(-1)^{s+1}\big(1-t^{k}\big)\big(1-q^\alpha t^{k-1}\big)t^{n-k}\prod\limits_{i=1}^s a_i}
{\big(1-t^{\lambda_1-k+1}\big)\Big(1-q^\alpha t^{2n-2-\lambda_1+k}\prod\limits_{i=1}^s a_ib_i\Big)},
\end{gather*}
which is equal to (\ref{eq:diag-tildeU}).
\end{proof}

\begin{proof}[Proof of Proposition \ref{prop:q-difference}] Lemma \ref{lem:EE=E**} shows that for $\lambda =(\lambda_1,\lambda_2,\dots,\lambda_s)\in Z_{s,n}$ we have
\begin{gather*} \mathsf{E}^n_0(\varepsilon;z) \mathsf{E}^{\lambda_1}_{\lambda_2,\ldots,\lambda_s}(a_2,\ldots,a_s;z)
=\mathsf{E}^{\lambda_1}_{0,\lambda_2,\ldots,\lambda_s}(\varepsilon,a_2,\ldots,a_s;z) \prod_{i=2}^s\prod_{k=1}^{\lambda_i} \! \big(a_it^{k-1}-\varepsilon\big) .
\end{gather*}
Taking the limit $\varepsilon\to 0$ we obtain
\begin{gather}\label{eq:zzz e(z)}
z_1z_2\cdots z_n\, e_{\lambda}(z)=\tilde{e}_{(\lambda_1,0,\lambda_2,\ldots,\lambda_s)}(z)\prod_{i=2}^s\prod_{k=1}^{\lambda_i}a_it^{k-1}.
\end{gather}
Since $T_{\alpha}\Phi(z)=z_1z_2\cdots z_n\Phi(z)$, (\ref{eq:zzz e(z)}) implies
\begin{gather*} T_{\alpha}\langle e_{\lambda},z\rangle =\langle \tilde{e}_{(\lambda_1,0,\lambda_2,\ldots,\lambda_s)},z\rangle\prod_{i=2}^s\prod_{k=1}^{\lambda_i}a_it^{k-1}.
\end{gather*}
From (\ref{eq:q-diff02}) in Lemma \ref{lem:E=EU} we therefore obtain
\begin{gather}
T_{\alpha}\big(\langle e_\lambda,z\rangle\big)_{\!\lambda\in Z_{s,n}} =\left(\langle \tilde{e}_{(\lambda_1,0,\lambda_2,\ldots,\lambda_s)},z\rangle
\prod_{i=2}^s\prod_{k=1}^{\lambda_i}a_it^{k-1} \right)_{\!\lambda\in Z_{s,n}}\nonumber\\
\hphantom{T_{\alpha}\big(\langle e_\lambda,z\rangle\big)_{\!\lambda\in Z_{s,n}}}{} =\big(\langle \tilde{e}_{(0,\mu_1,\mu_2,\ldots,\mu_s)},z\rangle \big)_{\!\mu\in Z_{s,n}}U,\label{eq:q-diff03}
\end{gather}
where the matrix $U=\big(u_{\mu\lambda}\big)_{\mu,\lambda\in Z_{s,n}}$ is defined as
\begin{gather*}
u_{\mu\lambda}= \tilde{u}_{\mu\lambda}\prod_{i=2}^s\prod_{k=1}^{\lambda_i}a_it^{k-1}
\end{gather*}
by the coefficient $\tilde{u}_{\mu\lambda}$ of (\ref{eq:q-diff02}). In particular, from (\ref{eq:diag-tildeU}), the diagonal entry $u_{\lambda\lambda}$ is determined as~(\ref{eq:diag-U}). On the other hand, (\ref{eq:E=EL}) in Corollary~\ref{cor:E=EL} implies that
\begin{gather}\label{eq:q-diff04}
\big(\langle \tilde{e}_{(0,\mu_1,\mu_2,\ldots,\mu_s)},z\rangle \big)_{\!\mu\in Z_{s,n}} = \big(\langle e_\lambda,z\rangle\big)_{\!\lambda\in Z_{s,n}}L.
\end{gather}
From \eqref{eq:q-diff02}, \eqref{eq:q-diff03} and \eqref{eq:q-diff04}, we obtain the expression (\ref{eq:q-diff01}). The formula~\eqref{eq:detLU} for $\det LU$ is obtained by using
\begin{gather*}
\prod_{\lambda\in Z_{s,n}} \prod_{i=1}^s\prod_{j=1}^{\lambda_i}a_it^{j-1}=\prod_{i=1}^s\prod_{\lambda\in Z_{s,n}}a_i^{\lambda_i}t^{\lambda_i\choose 2}
=\prod_{i=1}^s a_i^{n+s-1\choose n-1}t^{n+s-1\choose n-2}\\
\hphantom{\prod_{\lambda\in Z_{s,n}} \prod_{i=1}^s\prod_{j=1}^{\lambda_i}a_it^{j-1}}{} =(a_1a_2\cdots a_s)^{n+s-1\choose n-1}t^{s{n+s-1\choose n-2}}.\tag*{\qed}
\end{gather*}
\renewcommand{\qed}{}
\end{proof}

\section {Evaluation of the determinant formula}\label{Section6}

In this section we provide a proof of Lemma~\ref{lem:Wronski-a}. We show three lemmas before proving Lem\-ma~\ref{lem:Wronski-a}. We denote
\begin{gather*}
a_\mu=(a_{\mu,1},a_{\mu,2},\ldots,a_{\mu,n})\in (\mathbb{C}^*)^n
\end{gather*}
for $a_\mu\in (\mathbb{C}^*)^n$, $\mu\in Z_{s,n}$, specified by (\ref{eq:a_mu}). By definition we have
\begin{gather*}
\Phi(z)\Delta(z)=I_1(z)I_2(z),
\end{gather*}
where
\begin{gather*}
I_1(z) = \prod_{i=1}^n z_i^{\alpha+2(n-i)\tau},\nonumber\\
I_2(z) =\prod_{i=1}^n\prod_{m=1}^{s} \frac{\big(qa_m^{-1}z_i\big)_\infty}{(b_mz_i)_\infty}\prod_{1\le j<k\le n}\frac{\big(qt^{-1}z_k/z_j\big)_\infty}{(t z_k/z_j)_\infty}(1-z_k/z_j).
\end{gather*}

\begin{Lemma}\label{lem:J-asymptotic}For the determinant $J=\det\big( \langle s_\lambda, a_\mu\rangle\big)_{\!\lambda\in B\atop\!\mu\in Z}$ we have
\begin{gather}
T_\alpha^N \frac{J}{\prod_{\mu\in Z}I_1(a_\mu)}\sim
\prod_{k=1}^n\left[\frac{(1-q)^s(q)_\infty^s(t)_\infty^s}{(t^{n-k+1})_\infty^s}
\frac{\prod\limits_{1\le i<j\le s}\theta\big(a_ia_j^{-1}t^{-(n-k)}\big)}{\prod\limits_{i=1}^{s}\prod\limits_{j=1}^{s}\big(a_ib_jt^{n-k}\big)_\infty} \right]^{s+k-2\choose k-1}\nonumber\\
\hphantom{T_\alpha^N \frac{J}{\prod_{\mu\in Z}I_1(a_\mu)}\sim} {} \times\prod_{k=1}^n\left[\prod_{r=0}^{n-k}\prod_{1\le i<j\le s} a_jt^r\right]^{s+k-3\choose k-1} \qquad \text{as}\quad N\to \infty.\label{eq:J-asymptotic}
\end{gather}
\end{Lemma}

\begin{proof}From \eqref{eq:asympt <m, a>} in Lemma \ref{lem:<m, a>}, the asymptotic behavior of $T_\alpha^N\langle s_\lambda,a_\mu\rangle$ as $N\to \infty$ is
\begin{gather*}
T_\alpha^N\langle s_\lambda,a_\mu\rangle \sim (1-q)^n s_\lambda(a_\mu) \Delta(a_\mu)T_\alpha^N \Phi(a_\mu)\\
\hphantom{T_\alpha^N\langle s_\lambda,a_\mu\rangle}{} = (1-q)^n s_\lambda(a_\mu) I_2(a_\mu) T_\alpha^N I_1(a_\mu) \qquad \text{as}\quad N\to \infty .
\end{gather*}
Then, we obtain the expression
\begin{gather}\label{eq:J-asymptotic2}
T_\alpha^NJ=T_\alpha^N\det\big( \langle s_\lambda, a_\mu\rangle\big)_{\!\lambda\in B\atop\!\mu\in Z} \sim (1-q)^{n{s+n-1 \choose n}} \det\big(s_\lambda(a_\mu)\big)_{\!\lambda\in B\atop\!\mu\in Z}
\prod_{\mu\in Z}I_2(a_\mu) T_\alpha^N I_1(a_\mu).
\end{gather}
From \cite[Proposition~1.6]{AI2009} or \cite[Theorem~3.2, equation~(3.5)]{IIO2013} we have
\begin{gather}\label{eq:det-sigma}
\det\big(s_\lambda(a_\mu)\big)_{\!\lambda\in B\atop\!\mu\in Z} =\prod_{k=1}^n\prod_{r=0}^{n-k}\prod_{1\le i<j\le s}\big(t^{r}a_j-t^{(n-k-r)}a_i\big)^{s+k-3\choose k-1}.
\end{gather}
From \cite[p.~1103, equations~(8.9) and (8.10)]{AI2009} and
\begin{gather*}
\prod_{\mu\in Z}\prod_{i=1}^n\prod_{m=1}^{s}(b_m a_{\mu,i})_\infty =\prod_{k=1}^n\left[\prod_{i=1}^{s}\prod_{j=1}^{s}\big(a_ib_jt^{n-k}\big)_\infty\right]^{s+k-2\choose k-1}
\end{gather*}
we have
\begin{gather}
\prod_{\mu\in Z}I_2(a_\mu)=\prod_{\mu\in Z}\prod_{i=1}^n\prod_{m=1}^s\frac{\big(qa_m^{-1} a_{\mu,i}\big)_\infty}{(b_m a_{\mu,i})_\infty}\prod_{1\le j<k\le n}\frac{\big(qt^{-1} a_{\mu,k}/ a_{\mu,j}\big)_\infty}{(t a_{\mu,k}/ a_{\mu,j})_\infty}(1- a_{\mu,k}/ a_{\mu,j})\nonumber\\
\hphantom{\prod_{\mu\in Z}I_2(a_\mu)}{} =\prod_{k=1}^n\left[\frac{(q)_\infty^s(t)_\infty^s}{(t^{n-k+1})_\infty^s}
\frac{\prod\limits_{1\le i<j\le s}\theta\big(a_ia_j^{-1}t^{-(n-k)}\big)}{\prod\limits_{i=1}^{s}\prod\limits_{j=1}^{s}\big(a_ib_jt^{n-k}\big)_\infty}\right]^{s+k-2\choose k-1}\nonumber\\
\hphantom{\prod_{\mu\in Z}I_2(a_\mu)=}{} \times\prod_{k=1}^n\left[\prod_{r=0}^{n-k}\prod_{1\le i<j\le s} \frac{1}{1-t^{2r-(n-k)}a_ia_j^{-1}}\right]^{s+k-3\choose k-1}.\label{eq:I_2(a)}
\end{gather}
From (\ref{eq:J-asymptotic2}), (\ref{eq:det-sigma}), (\ref{eq:I_2(a)}) and $n{s+n-1 \choose n}=s\sum_{k=1}^n{s+k-2 \choose k-1}$, we therefore obtain (\ref{eq:J-asymptotic}) of Lem\-ma~\ref{lem:J-asymptotic}.
\end{proof}

\begin{Lemma}The product $\prod_{\mu\in Z_{s,n}}\Theta(a_\mu)$ satisfies
\begin{gather}
\prod_{\mu\in Z_{s,n}}\frac{\Theta(a_\mu)}{I_1(a_\mu)}=\prod_{k=1}^n\left[\frac{\theta(t)^s}{\theta(t^{n-k+1})^s}
\frac{\prod\limits_{1\le i<j\le s}\theta\big(a_ia_j^{-1}t^{-(n-k)}\big)} {\prod\limits_{i=1}^{s}\prod\limits_{j=1}^{s}\theta\big(a_i b_j t^{n-k}\big)}\right]^{s+k-2\choose k-1}\nonumber\\
\hphantom{\prod_{\mu\in Z_{s,n}}\frac{\Theta(a_\mu)}{I_1(a_\mu)}=}{} \times\prod_{k=1}^n\left[\prod_{r=0}^{n-k}\prod_{1\le i<j\le s}
\frac{1}{\theta\big(a_ia_j^{-1}t^{2r-(n-k)}\big)}\right] ^{s+k-3\choose k-1}.\label{eq:Theta(a)}
\end{gather}
\end{Lemma}

\begin{proof}By the definition (\ref{eq:lla-rra}) of $\Theta(z)$, this is directly confirmed from
\begin{gather*}
\prod_{\mu\in Z}\prod_{i=1}^n\prod_{m=1}^{s}\theta(b_m a_{\mu,i}) =\prod_{k=1}^n\left[\prod_{i=1}^{s}\prod_{j=1}^{s}\theta\big(a_ib_jt^{n-k}\big)\right]^{s+k-2\choose k-1}
\end{gather*}
and
\begin{gather*}
\prod_{\mu\in Z}\prod_{1\le j<k\le n}\frac{\theta( a_{\mu,k}/ a_{\mu,j})}{\theta(t a_{\mu,k}/ a_{\mu,j})}=\prod_{k=1}^n \left[\frac{\theta(t)^s}{\theta(t^{n-k+1})^s}
\prod_{1\le i<j\le s}\theta\big(a_ia_j^{-1}t^{-(n-k)}\big) \right]^{s+k-2\choose k-1}\\
\hphantom{\prod_{\mu\in Z}\prod_{1\le j<k\le n}\frac{\theta( a_{\mu,k}/ a_{\mu,j})}{\theta(t a_{\mu,k}/ a_{\mu,j})}=}{} \times\prod_{k=1}^n\left[\prod_{r=0}^{n-k}\prod_{1\le i<j\le s}
\frac{1}{\theta\big(t^{2r-(n-k)}a_ia_j^{-1}\big)}\right]^{s+k-3\choose k-1},
\end{gather*}
which is given in \cite[p.~1103, equation~(8.10)]{AI2009}.
\end{proof}

We denote by ${\cal C}$ the right-hand side of (\ref{eq:Wronski-a1}). Then we have the following:
\begin{Lemma}\label{lem:C-asymptotic}We have
\begin{gather*}
 \mathcal{C}\prod_{\mu\in Z_{s,n}}\frac{\Theta(a_\mu)}{I_1(a_\mu)} =\prod_{k=1}^n\left[\prod_{r=0}^{n-k}\prod_{1\le i<j\le s}a_jt^r\right]^{s+k-3\choose k-1}\\
\qquad{} \times \prod_{k=1}^n\left[
\frac{(1-q)^s(q)_\infty^s (t)_\infty^s}{(t^{n-k+1})_\infty^s}
\frac{\Big(q^\alpha t^{n+k-2}\prod\limits_{i=1}^s a_ib_i\Big)_\infty}{(q^\alpha t^{n-k})_\infty}
\frac{\prod\limits_{1\le i<j\le s}\theta\big(a_ia_j^{-1}t^{-(n-k)}\big)}{\prod\limits_{i=1}^{s}\prod\limits_{j=1}^{s}(a_i b_j t^{n-k})_\infty} \right]^{s+k-2\choose k-1}.
\end{gather*}
\end{Lemma}

\begin{proof} From (\ref{eq:Theta(a)}) and the explicit form of $\mathcal{C}$, we can immediately confirm the result.
\end{proof}

We now prove Lemma \ref{lem:Wronski-a}.

\begin{proof}[Proof of Lemma \ref{lem:Wronski-a}] Define the determinant $\mathcal{J}=\det\big(\langle\!\langle s_\lambda, a_\mu\rangle\!\rangle\big)_{\!\lambda\in B\atop\!\mu\in Z}$. Then we have
\begin{gather*}
\mathcal{J} =\det\big(\langle\!\langle s_\lambda, a_\mu \rangle\!\rangle\big)_{\!\lambda\in B\atop\!\mu\in Z} =\det\big(\langle s_\lambda, a_\mu \rangle\big)_{\!\lambda\in B\atop\!\mu\in Z}
\Big/\prod_{\mu\in Z_{s,n}}\Theta(a_\mu) =J\Big/\prod_{\mu\in Z_{s,n}}\Theta(a_\mu).
\end{gather*}
Since we have
\begin{gather*}
\frac{T_{\alpha}\Big(\prod\limits_{\mu\in Z_{s,n}}\Theta(a_\mu)\Big)} {\prod\limits_{\mu\in Z_{s,n}}\Theta(a_\mu)} =\prod_{\mu\in Z_{s,n}} a_{\mu,1} a_{\mu,2}\cdots a_{\mu,n}
=(a_1a_2\cdots a_s)^{s+n-1\choose n-1}t^{s{s+n-1\choose n-2}}\\
\hphantom{\frac{T_{\alpha}\Big(\prod\limits_{\mu\in Z_{s,n}}\Theta(a_\mu)\Big)} {\prod\limits_{\mu\in Z_{s,n}}\Theta(a_\mu)}}{} =\prod_{k=1}^n\left[t^{k-1}\prod_{i=1}^s a_i \right]^{s+k-2\choose k-1},
\end{gather*}
which is proved in \cite[p.~1097, Lemma~7.3]{AI2009}, and (\ref{eq:q-diff-det}) for $J$, we obtain
\begin{gather*} T_{\alpha}\mathcal{J}
=\mathcal{J}\prod_{k=1}^n\left[\frac{1-q^\alpha t^{n-k}}{1-q^\alpha t^{n+k-2}\prod\limits_{i=1}^{s} a_ib_i}\right]^{s+k-2\choose k-1}.
\end{gather*}
On the other hand it is easily confirmed that $\mathcal{C}$ satisfies the same equation as above, i.e.,
\begin{gather*}
T_{\alpha}\mathcal{C}=\mathcal{C}\prod_{k=1}^n\left[\frac{1-q^\alpha t^{n-k}}{1-q^\alpha t^{n+k-2}\prod\limits_{i=1}^{s} a_ib_i}\right]^{s+k-2\choose k-1}.\end{gather*}
This means that the ratio ${\mathcal{J}}/{\mathcal{C}}$ is invariant under the shift $T_\alpha$ with respect to $\alpha\to \alpha+1$. Thus~${\mathcal{J}}/{\mathcal{C}}$ is also invariant under the shift $T_\alpha^N$, $N=1,2,\ldots$. From Lemmas~\ref{lem:J-asymptotic} and~\ref{lem:C-asymptotic} we therefore obtain
\begin{gather*}
\frac{\mathcal{J}}{\mathcal{C}}=\lim_{N\to +\infty}T_\alpha^N\frac{\mathcal{J}}{\mathcal{C}} =\lim_{N\to +\infty}
T_\alpha^N\left(\frac{J}{\prod\limits_{\mu\in Z}I_1(a_\mu)}\right) T_\alpha^N\left(\mathcal{C}\prod_{\mu\in Z_{s,n}}\frac{\Theta(a_\mu)}{I_1(a_\mu)}\right)^{-1}=1,
\end{gather*}
which is the claim of Lemma \ref{lem:Wronski-a}.
\end{proof}

This also completes the proof of our determinant formula of Theorem \ref{thm:Wronski}.

\appendix

\section{Proof of Lemma \ref{lem:Anabla}}\label{SectionA}

In this section we give a proof of Lemma \ref{lem:Anabla}. From (\ref{eq:def-nabla}), (\ref{eq:b-function}) and (\ref{eq:Fj}) we have
\begin{gather*}
\nabla_{\!1}\big( G_1(z)\tilde{e}^{(n-1)}_{\lambda}(z_{\widehat{1}})\Delta^{(n-1)}(z_{\widehat{1}}) \big)=\big(G_1(z)-F_1(z)\big) \tilde{e}^{(n-1)}_{\lambda}(z_{\widehat{1}})\Delta^{(n-1)}(z_{\widehat{1}}).
\end{gather*}
In order to show \eqref{eq:AnablaF1eD} of Lemma \ref{lem:Anabla} we prove that
\begin{gather}
{\cal A}\big[\big(G_1(z)-F_1(z)\big) \tilde{e}^{(n-1)}_{\lambda}(z_{\widehat{1}}) \Delta^{(n-1)}(z_{\widehat{1}})\big]\nonumber\\
\qquad={} \left(c_0\tilde{e}^{(n)}_{\lambda+\epsilon_0}(z) +c_1\tilde{e}^{(n)}_{\lambda+\epsilon_1}(z)+\sum_{k=2}^{s} c_k\tilde{e}^{(n)}_{\lambda+\epsilon_k}(z)\right)\Delta^{(n)}(z). \label{eq:Anabla01}
\end{gather}
For this purpose we slightly perturb $F_j(z)$ and $G_j(z)$ with a parameter $\varepsilon$ and define
\begin{gather}
F_j^\varepsilon(z)=\big(q^{\alpha}-\varepsilon z_j^{-1}\big)\prod_{m=1}^{s}(1-b_mz_j)\prod_{\substack{1\le k\le n\atop k\ne j}}(z_k-tz_j),\label{eq:Fje}\\
G_j^\varepsilon(z)=\big(1-\varepsilon z_j^{-1}\big)\prod_{m=1}^{s}\big(1-a_m^{-1}z_j\big)\prod_{\substack{1\le k\le n\atop k\ne j}}\big(z_k-t^{-1}z_j\big),\label{eq:Gje}
\end{gather}
which satisfy $F_j^\varepsilon(z)\to F_j(z)$ and $G_j^\varepsilon(z)\to G_j(z)$ when $\varepsilon \to 0$, respectively. For $\lambda_0\in \{0,1,\ldots,$ $n-1\}$ and $\lambda=(\lambda_1,\ldots,\lambda_s)\in Z_{s, n-\lambda_0-1}$ we prove the equation
\begin{gather}
 {\cal A}\big[\big(G_1^{\varepsilon}(z)-F_1^{\varepsilon}(z)\big) \mathsf{E}^{\lambda_0,(n-1)}_{\lambda}(z_{\widehat{1}})\Delta^{(n-1)}(z_{\widehat{1}})\big]\nonumber\\
\qquad{}= \left(c_0^\varepsilon \mathsf{E}^{\lambda_0+1,(n)}_{\lambda}(z) +c_1^\varepsilon \mathsf{E}^{\lambda_0,(n)}_{\lambda+\epsilon_1}(z)+\sum_{k=2}^{s} c_k^\varepsilon\, \mathsf{E}^{\lambda_0,(n)}_{\lambda+\epsilon_k}(z)\right) \Delta^{(n)}(z),\label{eq:AnablaG1ED}
\end{gather}
where $\mathsf{E}^{\lambda_0,(n-1)}_{\lambda}(z_{\widehat{1}}) =\mathsf{E}^{\lambda_0,(n-1)}_{\lambda_1,\ldots,\lambda_s}(\varepsilon,a_2,\ldots,a_s;z_{\widehat{1}})$ and $c_0^\varepsilon$, $c_1^\varepsilon$ and $c_k^\varepsilon$, $k=2,\ldots,s$, are
\begin{gather}
c_0^\varepsilon =(n-1)! (-1)^{n+s-1}\frac{\big(1-t^{\lambda_0+1}\big)\Big(1-q^\alpha t^{2n-2-\lambda_0}\prod\limits_{i=1}^sa_ib_i\Big)} {(1-t)t^{n-1}\prod\limits_{i=1}^sa_i}, \label{eq:c0e}\\
c_1^\varepsilon =(n-1)! (-1)^{n+1}\frac{\big(1-t^{\lambda_1+1}\big)\big(1-q^\alpha t^{\lambda_{1}}\big) \prod\limits_{m=1}^{s}\big(1-\varepsilon b_m t^{\lambda_{1}}\big)}{(1-t)t^{\lambda_1}}
\prod_{j=2}^{s} \frac{t^{\lambda_j}\big(a_j-\varepsilon t^{\lambda_1+1}\big)} {a_jt^{\lambda_j}-\varepsilon t^{\lambda_1+1}},\label{eq:c1e}\\
c_k^\varepsilon =(n-1)!(-1)^{n+1}\frac{\big(\varepsilon-q^\alpha a_kt^{\lambda_{k}}\big)
\big(\varepsilon -a_kt^{\lambda_k+1}\big)t^{\lambda_1}\big(1-t^{\lambda_k+1}\big)\prod\limits_{m=1}^{s}\big(1-b_m a_kt^{\lambda_{k}}\big)}{\big(\varepsilon t^{\lambda_1}-a_kt^{\lambda_k+1}\big)(1-t)a_kt^{\lambda_{k}}}\nonumber\\
\hphantom{c_k^\varepsilon =}{} \times \prod_{\substack{2\le j\le s\atop j\ne k}}\frac{t^{\lambda_j}\big(a_j-a_kt^{\lambda_k+1}\big)}{a_jt^{\lambda_j}-a_kt^{\lambda_k+1}}.\label{eq:cke}
\end{gather}
Since (\ref{eq:Anabla01}) is obtained from (\ref{eq:AnablaG1ED}) if $\varepsilon \to 0$, it suffices to prove (\ref{eq:AnablaG1ED}) instead of (\ref{eq:Anabla01}). We give a proof of (\ref{eq:AnablaG1ED}) below. We denote
\begin{gather}\label{eq:nablaG1ED}
\phi(z)=\big(G_1^{\varepsilon}(z)-F_1^{\varepsilon}(z)\big) \mathsf{E}^{\lambda_0,(n-1)}_{\lambda}(z_{\widehat{1}})\Delta^{(n-1)}(z_{\widehat{1}}).
\end{gather}
Taking account of the degree of the polynomial $\phi(z)$, the polynomial
\begin{gather*}
 {\cal A}\big(z_1^{s+n-1}\times (z_2\cdots z_{\lambda_0+1})^s(z_{\lambda_0+2}\cdots z_n)^{s-1}\times z_2^{n-2}z_3^{n-3}\cdots z_{n-1}\big)\\
\qquad{} ={\cal A}\big((z_1\cdots z_{\lambda_0+1})^s(z_{\lambda_0+2}\cdots z_n)^{s-1} \times z_1^{n-1}z_2^{n-2}\cdots z_{n-1}\big)
\end{gather*}
is the term of highest degree in the skew-symmetrization ${\cal A}\phi(z)$, which is thus expanded as
\begin{gather}
{\cal A}\phi(z)=\sum_{\substack{0\le \mu_0\le \lambda_0+1\\[1.5pt] (\mu_1,\ldots,\mu_s)\in Z_{s,n-\mu_0}}}
c_\mu^\varepsilon \mathsf{E}^{\mu_0,(n)}_{\mu_1,\ldots,\mu_s}(z) \Delta^{(n)}(z),\label{eq:Aphi-1}
\end{gather}
where $c_\mu^\varepsilon=c_{(\mu_0,\mu_1,\ldots,\mu_s)}^\varepsilon$ are some constants. For any $\delta\in \mathbb{R}$ we set
\begin{gather*}
\xi^{\mu_0,(n)}_{\mu_1,\ldots,\mu_s} =\big(\underbrace{\delta^{\mu_0},\delta^{\mu_0-1},\ldots,\delta\vphantom{\Big|}}_{\mu_0},
\underbrace{\varepsilon,\varepsilon t,\ldots,\varepsilon t^{\mu_1-1}\vphantom{\Big|}}_{\mu_1},
\underbrace{a_2,a_2 t,\ldots,a_2 t^{\mu_1-1}\vphantom{\Big|}}_{\mu_2}, \ldots,\underbrace{a_s,a_s t,\ldots,a_s t^{\mu_s-1}\vphantom{\Big|}}_{\mu_s}\big)\\
\hphantom{\xi^{\mu_0,(n)}_{\mu_1,\ldots,\mu_s}}{}\in (\mathbb{C}^*)^n,
\end{gather*}
which is a special case of $x^{\mu_0,(n)}_{\mu_1,\ldots,\mu_s}$ specified by (\ref{eq:zeta-mu}) with the setting $z_i=\delta^{\mu_0-i+1}$, $i=1,\ldots,\mu_0$, and $(x_1,x_2,\ldots,x_s)=(\varepsilon,a_2,\ldots,a_s)$.
Then, from Lemma~\ref{lem:vanishing}, $c_\mu^\varepsilon$ in~(\ref{eq:Aphi-1}) is written as
\begin{gather}
\lim_{\delta\to \infty}\frac{{\cal A}\phi(z)} {z_1^{s+n-1}z_2^{s+n-2}\cdots z_{\mu_0}^{s+n-\mu_0}}\bigg|_{z=\xi^{\mu_0,(n)}_{\mu_1,\ldots,\mu_s}} =c_\mu^\varepsilon\lim_{\delta\to \infty}
\frac{\mathsf{E}^{\mu_0,(n)}_{\mu_1,\ldots,\mu_s}(z)\Delta^{\!(n)}(z)} {z_1^{s+n-1}z_2^{s+n-2}\cdots z_{\mu_0}^{s+n-\mu_0}}\bigg|_{z=\xi^{\mu_0,(n)}_{\mu_1,\ldots,\mu_s}}\nonumber\\
\hphantom{\lim_{\delta\to \infty}\frac{{\cal A}\phi(z)} {z_1^{s+n-1}z_2^{s+n-2}\cdots z_{\mu_0}^{s+n-\mu_0}}\bigg|_{z=\xi^{\mu_0,(n)}_{\mu_1,\ldots,\mu_s}}}{} =c_\mu^\varepsilon
\Delta^{\!(n-\mu_0)}(\xi^{0,(n-\mu_0)}_{\mu_1,\ldots,\mu_s}).\label{eq:Aphi-2}
\end{gather}
On the other hand, from the explicit form (\ref{eq:nablaG1ED}) of $\phi(z)$, the left-hand side of (\ref{eq:AnablaG1ED}) is written as
\begin{gather}\label{eq:Aphi-3}
{\cal A}\phi(z)=(n-1)!\sum_{j=1}^n (-1)^{j}\big(F_j^\varepsilon(z)-G_j^\varepsilon(z)\big) \mathsf{E}^{\lambda_0,(n-1)}_{\lambda}(z_{\widehat{j}})\Delta^{\!(n-1)}(z_{\widehat{j}}),
\end{gather}
where $F_j^\varepsilon(z)$ and $G_j^\varepsilon(z)$ are given by (\ref{eq:Fje}) and (\ref{eq:Gje}), respectively. By definition $F_j^\varepsilon(z)$ and~$G_j^\varepsilon(z)$ satisfy the vanishing properties
\begin{gather}
G_j^\varepsilon\big(\xi^{\mu_0,(n)}_{\mu_1,\ldots,\mu_s}\big)=0\qquad \mbox{if}\quad\mu_0< j\le n,\nonumber\\
F_j^\varepsilon\big(\xi^{\mu_0,(n)}_{\mu_1,\ldots,\mu_s}\big)=0 \qquad \mbox{if}\quad \mu_0< j \quad\mbox{and}\quad j\not\in \{\mu_0+\mu_1+\cdots+\mu_k \,|\, k=1,\ldots,s\},\label{eq:Gj=Fj=0}
\end{gather}
and the evaluations
\begin{gather}
\lim_{\delta\to \infty} \frac{F_{\mu_0+\mu_1}^\varepsilon(z)}{z_1\cdots z_{\mu_0}}\bigg|_{z=\xi^{\mu_0,(n)}_{\mu_1,\ldots,\mu_s}}
=\frac{\big(q^\alpha t^{\mu_1-1} -1\big)\prod\limits_{m=1}^{s}\big(1-b_m \varepsilon t^{\mu_1-1}\big)} {t^{\mu_1-1}\big(\varepsilon t^{\mu_1-1}-\varepsilon t^{\mu_1}\big)}
\prod_{i=1}^{\mu_1}\big(\varepsilon t^{i-1}-\varepsilon t^{\mu_1}\big) \nonumber\\
\hphantom{\lim_{\delta\to \infty} \frac{F_{\mu_0+\mu_1}^\varepsilon(z)}{z_1\cdots z_{\mu_0}}\bigg|_{z=\xi^{\mu_0,(n)}_{\mu_1,\ldots,\mu_s}}=}{} \times
\prod_{i=2}^s\prod_{j=1}^{\mu_i}\big(a_it^{j-1}-\varepsilon t^{\mu_1}\big),\label{eq:limFmu-1}\\
\lim_{\delta\to \infty}\!\frac{F_{\mu_0+\mu_1+\cdots+\mu_k}^\varepsilon(z)}{z_1\cdots z_{\mu_0}}\bigg|_{z=\xi^{\mu_0,(n)}_{\mu_1,\ldots,\mu_s}}=\frac{\big(q^\alpha a_kt^{\mu_k-1} -\varepsilon\big)
\!\prod\limits_{m=1}^{s}\!\big(1-b_m a_kt^{\mu_k-1}\big)}{a_kt^{\mu_k-1}\big(a_kt^{\mu_k-1}-a_kt^{\mu_k}\big)}\prod_{i=1}^{\mu_1}\big(\varepsilon t^{i-1}-a_kt^{\mu_{k}}\big)\nonumber\\
\hphantom{\lim_{\delta\to \infty}\!\frac{F_{\mu_0+\mu_1+\cdots+\mu_k}^\varepsilon(z)}{z_1\cdots z_{\mu_0}}\bigg|_{z=\xi^{\mu_0,(n)}_{\mu_1,\ldots,\mu_s}}=}{}
\times\prod_{i=2}^s\prod_{j=1}^{\mu_i}\big(a_it^{j-1}-a_kt^{\mu_{k}}\big)\qquad \mbox{for}\quad k=2,\ldots,s.\label{eq:limFmu-k}
\end{gather}
Notice that, from Lemma \ref{lem:vanishing}, $\mathsf{E}^{\lambda_0,(n-1)}_{\lambda}(z_{\widehat{j}})\Delta^{(n-1)}(z_{\widehat{j}})$, $1\le j\le \mu_0$, satisfies the following va\-nishing property; if $1\le j\le \mu_0$, then
\begin{gather}
\lim_{\delta\to \infty}\frac{\mathsf{E}^{\lambda_0,(n-1)}_{\lambda}(z_{\widehat{j}})\Delta^{(n-1)}(z_{\widehat{j}})}
{z_1^{s+n-2}z_2^{s+n-3}\cdots z_{j-1}^{s+n-j}z_{j+1}^{s+n-j-1}\cdots z_{\mu_0}^{s+n-\mu_0}}
\bigg|_{z=\xi^{\mu_0,(n)}_{\mu}}\nonumber\\
 \qquad{} = \begin{cases}
\Delta^{(n-\lambda_0-1)}\big(\xi^{0,(n-\lambda_0-1)}_{\lambda}\big) &\mbox{if $\mu_0=\lambda_0+1$ and $\mu=\lambda\in Z_{s,n-\lambda_0-1}$},\vspace{1mm}\\
0&\mbox{otherwise}.
\end{cases} \label{eq:E=0-1}
\end{gather}
In the same manner, $\mathsf{E}^{\lambda_0,(n-1)}_{\lambda}(z_{\widehat{j}})\Delta^{(n-1)}(z_{\widehat{j}})$, $j\in \{\mu_0+\mu_1+\cdots+\mu_k\,|\, k=1,\ldots,s\}$, satisfies the vanishing property
\begin{gather}
\lim_{\delta\to \infty}\frac{\mathsf{E}^{\lambda_0,(n-1)}_{\lambda}(z_{\widehat{\mu_0+\mu_1+\cdots+\mu_k}})\Delta^{\!(n-1)}(z_{\widehat{\mu_0+\mu_1+\cdots+\mu_k}})}{z_1^{s+n-2}z_2^{s+n-3}\cdots z_{\mu_0}^{s+n-\mu_0-1}} \bigg|_{z=\xi^{\mu_0,(n)}_{\mu}}\nonumber\\
\qquad{} =
\begin{cases}
\Delta^{(n-\lambda_0-1)}\big(\xi^{0,(n-\lambda_0-1)}_{\lambda}\big) & \mbox{if $\mu_0=\lambda_0$ and $\mu=\lambda+\varepsilon_k\in Z_{s,n-\lambda_0}$}, \\
0&\mbox{otherwise}.\label{eq:E=0-2}
\end{cases}
\end{gather}
If $z=\xi^{\mu_0,(n)}_{\mu_1,\ldots,\mu_s}$ satisfies $\xi^{\mu_0,(n)}_{\mu_1,\ldots,\mu_s}\ne \xi^{\lambda_0+1,(n)}_{\lambda}$ and $\xi^{\mu_0,(n)}_{\mu_1,\ldots,\mu_s}\ne \xi^{\lambda_0,(n)}_{\lambda+\varepsilon_k}$, then, from (\ref{eq:Gj=Fj=0}), (\ref{eq:E=0-1}) and~(\ref{eq:E=0-2}) the equation~(\ref{eq:Aphi-3}) implies
\begin{gather}\label{eq:Aphi-4}
\lim_{x\to \infty}\frac{{\cal A}\phi(z)}{z_1^{s+n-1}z_2^{s+n-2}\cdots z_{\mu_0}^{s+n-\mu_0}}\bigg|_{z=\xi^{\mu_0,(n)}_{\mu_1,\ldots,\mu_s}}=0.
\end{gather}
Comparing (\ref{eq:Aphi-2}) with (\ref{eq:Aphi-4}), we obtain $c_\mu^\varepsilon=0$ if $\mu\ne (\lambda_0,\ldots,\lambda_k+1,\ldots,\lambda_s)$, $k=0,\ldots,s$, which means that ${\cal A}\phi(z)$ is expressed as in~(\ref{eq:AnablaG1ED}).

We now determine the coefficients $c_k^\varepsilon$, $k=0,\ldots,s$, in the expression (\ref{eq:AnablaG1ED}).

First we evaluate $c_1^\varepsilon$. From (\ref{eq:limFmu-1}) we have
\begin{gather*}
\lim_{\delta\to \infty} \frac{F_{\lambda_0+\lambda_1+1}^\varepsilon(z)}{z_1\cdots z_{\lambda_0}}\bigg|_{z=\xi^{\lambda_0,(n)}_{\lambda+\epsilon_1}}
=\frac{\big(q^\alpha t^{\lambda_1} -1\big)\prod\limits_{m=1}^{s}\big(1-b_m \varepsilon t^{\lambda_1}\big)}{t^{\lambda_1}\big(\varepsilon t^{\lambda_1}-\varepsilon t^{\lambda_1+1}\big)}\prod_{i=1}^{\lambda_1+1}\big(\varepsilon t^{i-1}-\varepsilon t^{\lambda_1+1}\big) \\
\hphantom{\lim_{\delta\to \infty} \frac{F_{\lambda_0+\lambda_1+1}^\varepsilon(z)}{z_1\cdots z_{\lambda_0}}\bigg|_{z=\xi^{\lambda_0,(n)}_{\lambda+\epsilon_1}}=}{} \times
\prod_{i=2}^s\prod_{j=1}^{\lambda_i}\big(a_it^{j-1}-\varepsilon t^{\lambda_1+1}\big).
\end{gather*}
Then
\begin{gather}
 \lim_{\delta\to \infty}\frac{{\cal A}\phi(z)}{z_1^{s+n-1}z_2^{s+n-2}\cdots z_{\lambda_0}^{s+n-\lambda_0}}\bigg|_{z=\xi^{\lambda_0,(n)}_{\lambda+\epsilon_1}}\nonumber\\
=(n-1)!(-1)^{\lambda_0+\lambda_1+1}\lim_{\delta\to \infty}\! \left(\frac{F_{\lambda_0+\lambda_1+1}(z)}{z_1\cdots z_{\lambda_0}}\frac{\mathsf{E}^{\lambda_0,(n-1)}_{\lambda}(z_{\widehat{\lambda_0+\lambda_1+1}})
\Delta^{(n-1)}(z_{\widehat{\lambda_0+\lambda_1+1}})}{z_1^{s+n-2}z_2^{s+n-3}\cdots z_{\lambda_0}^{s+n-\lambda_0-1}}\right)\bigg|_{z=\xi^{\lambda_0,(n)}_{\lambda+\epsilon_1}}\nonumber\\
=(n-1)!(-1)^{\lambda_0+\lambda_1+1}\Bigg[\frac{\big(q^\alpha t^{\lambda_1} -1\big)\prod\limits_{m=1}^{s}\big(1-b_m \varepsilon t^{\lambda_1}\big)}{t^{\lambda_1}\big(\varepsilon t^{\lambda_1}-\varepsilon t^{\lambda_1+1}\big)}\prod_{i=1}^{\lambda_1+1}\big(\varepsilon t^{i-1}-\varepsilon t^{\lambda_1+1}\big)\nonumber\\
\quad{} \times \prod_{i=2}^s\prod_{j=1}^{\lambda_i}\big(a_it^{j-1}-\varepsilon t^{\lambda_1+1}\big)\Bigg]\Delta^{(n-\lambda_0-1)}\big(\xi^{0,(n-\lambda_0-1)}_{\lambda}\big).\label{eq:Aphi-5***}
\end{gather}
Comparing (\ref{eq:Aphi-2}) with (\ref{eq:Aphi-5***}), the coefficients $c_1^\varepsilon=c_{(\lambda_0,\lambda_1+1,\lambda_2,\ldots,\lambda_s)}^\varepsilon$ is given by
\begin{gather*}
c_1^\varepsilon =(n-1)!(-1)^{\lambda_0+\lambda_1+1}
\Bigg[ \frac{\big(q^\alpha t^{\lambda_1} -1\big)\prod\limits_{m=1}^{s}\big(1-b_m \varepsilon t^{\lambda_1}\big)}
{t^{\lambda_1}\big(\varepsilon t^{\lambda_1}-\varepsilon t^{\lambda_1+1}\big)}
\prod_{i=1}^{\lambda_1+1}\big(\varepsilon t^{i-1}-\varepsilon t^{\lambda_1+1}\big)\\
\hphantom{c_1^\varepsilon =}{} \times\prod_{i=2}^s\prod_{j=1}^{\lambda_i}\big(a_it^{j-1}-\varepsilon t^{\lambda_1+1}\big)\Bigg] \frac{\Delta^{(n-\lambda_0-1)}\big(\xi^{0,(n-\lambda_0-1)}_{\lambda}\big)}{\Delta^{(n-\lambda_0)}\big(\xi^{0,(n-\lambda_0)}_{\lambda+\epsilon_1}\big)},
\end{gather*}
where
\begin{gather*}
\frac{\Delta^{(n-\lambda_0-1)}\big(\xi^{0,(n-\lambda_0-1)}_{\lambda}\big)} {\Delta^{(n-\lambda_0)}\big(\xi^{0,(n-\lambda_0)}_{\lambda+\epsilon_1}\big)}
=(-1)^{\lambda_{2}+\lambda_{3}+\cdots+\lambda_s} \prod_{j=1}^{\lambda_1}\frac{1}{\varepsilon t^{j-1}-\varepsilon t^{\lambda_{1}}}\prod_{i=2}^s\prod_{j=1}^{\lambda_i}\frac{1}{a_it^{j-1}-\varepsilon t^{\lambda_{1}}}.
\end{gather*}
Therefore $c_1^\varepsilon$ in (\ref{eq:AnablaG1ED}) is equal to (\ref{eq:c1e}).

Next we evaluate $c_k^\varepsilon$, $k=2,\ldots,s$, in the expression (\ref{eq:AnablaG1ED}). From (\ref{eq:limFmu-k}) we have
\begin{gather}
\lim_{\delta\to \infty}\frac{F_{\lambda_0+\lambda_1+\cdots+\lambda_k+1}^\varepsilon(z)} {z_1\cdots z_{\lambda_0}}\bigg|_{z=\xi^{\lambda_0,(n)}_{\lambda+\epsilon_k}}
=\frac{\big(q^\alpha a_kt^{\lambda_{k}} -\varepsilon\big)\prod\limits_{m=1}^{s}\big(1-b_m a_kt^{\lambda_{k}}\big)}{a_kt^{\lambda_{k}}}
\prod_{i=1}^{\lambda_1}\big(\varepsilon t^{i-1}-a_kt^{\lambda_{k}+1}\big) \nonumber\\
\hphantom{\lim_{\delta\to \infty}\frac{F_{\lambda_0+\lambda_1+\cdots+\lambda_k+1}^\varepsilon(z)} {z_1\cdots z_{\lambda_0}}\bigg|_{z=\xi^{\lambda_0,(n)}_{\lambda+\epsilon_k}}=}{} \times\prod_{i=2}^s\prod_{j=1}^{\lambda_i}(a_it^{j-1}-a_kt^{\lambda_{k}+1}).\label{eq:limFlambda=}
\end{gather}
Using (\ref{eq:Aphi-3}), (\ref{eq:Gj=Fj=0}), (\ref{eq:E=0-2}) and (\ref{eq:limFlambda=}) we obtain
\begin{gather}
\lim_{\delta\to \infty}\frac{{\cal A}\phi(z)}{z_1^{s+n-1}z_2^{s+n-2}\cdots z_{\lambda_0}^{s+n-\lambda_0}}\bigg|_{z=\xi^{\lambda_0,(n)}_{\lambda+\epsilon_k}}=(n-1)!(-1)^{\lambda_0+\cdots+\lambda_k+1}\nonumber\\
\qquad{} \times\lim_{\delta\to \infty}\left(\frac{F_{\lambda_0+\cdots+\lambda_k+1}^\varepsilon(z)}{z_1\cdots z_{\lambda_0}}
\frac{\mathsf{E}^{\lambda_0,(n-1)}_{\lambda}(z_{\widehat{\lambda_0+\cdots+\lambda_k+1}})
\Delta^{(n-1)}(z_{\widehat{\lambda_0+\cdots+\lambda_k+1}})}{z_1^{s+n-2}z_2^{s+n-3}\cdots z_{\lambda_0}^{s+n-\lambda_0-1}}\right) \bigg|_{z=\xi^{\lambda_0,(n)}_{\lambda+\epsilon_k}}\nonumber\\
\quad{} =(n-1)!(-1)^{\lambda_0+\cdots+\lambda_k+1}
\Bigg[\frac{\big(q^\alpha a_kt^{\lambda_{k}} -\varepsilon\big)\prod\limits_{m=1}^{s}\big(1-b_m a_kt^{\lambda_{k}}\big)}{a_kt^{\lambda_{k}}} \prod_{i=1}^{\lambda_1}\big(\varepsilon t^{i-1}-a_kt^{\lambda_{k}+1}\big)
\nonumber\\
\qquad{} \times\prod_{i=2}^s\prod_{j=1}^{\lambda_i}\big(a_it^{j-1}-a_kt^{\lambda_{k}+1}\big) \Bigg] \Delta^{(n-\lambda_0-1)}\big(\xi^{0,(n-\lambda_0-1)}_{\lambda}\big).\label{eq:Aphi-5}
\end{gather}
Comparing (\ref{eq:Aphi-2}) with (\ref{eq:Aphi-5}), the coefficients $c_k^\varepsilon=c_{(\lambda_0,\ldots,\lambda_k+1,\ldots,\lambda_s)}^\varepsilon$, $k=2,\ldots,s$, are given by
\begin{gather*}
c_k^\varepsilon =(n-1)!(-1)^{\lambda_0+\cdots+\lambda_k+1} \frac{\big(q^\alpha a_kt^{\lambda_{k}} -\varepsilon\big)\prod\limits_{m=1}^{s}\big(1-b_m a_kt^{\lambda_{k}}\big)}{a_kt^{\lambda_{k}}}
\prod_{i=1}^{\lambda_1}\big(\varepsilon t^{i-1}-a_kt^{\lambda_{k}+1}\big)\nonumber\\
\hphantom{c_k^\varepsilon =}{} \times\prod_{i=2}^s\prod_{j=1}^{\lambda_i}\big(a_it^{j-1}-a_kt^{\lambda_{k}+1}\big) \times\frac{\Delta^{(n-\lambda_0-1)}\big(\xi^{0,(n-\lambda_0-1)}_{\lambda}\big)}{
\Delta^{(n-\lambda_0)}\big(\xi^{0,(n-\lambda_0)}_{\lambda+\epsilon_k}\big)},
\end{gather*}
where{\samepage
\begin{gather*}
\frac{\Delta^{(n-\lambda_0-1)}\big(\zeta^{0,(n-\lambda_0-1)}_{\lambda}\big)} {\Delta^{(n-\lambda_0)}\big(\zeta^{0,(n-\lambda_0)}_{\lambda+\epsilon_k}\big)}=(-1)^{\lambda_{k+1}+\lambda_{k+2}+\cdots+\lambda_s}\prod_{j=1}^{\lambda_1}\frac{1}{\varepsilon t^{j-1}-a_kt^{\lambda_{k}}}\prod_{i=2}^s\prod_{j=1}^{\lambda_i} \frac{1}{a_it^{j-1}-a_kt^{\lambda_{k}}}.
\end{gather*}
Therefore $c_k^\varepsilon$ in (\ref{eq:AnablaG1ED}) is in accordance with (\ref{eq:cke}).}

Lastly we evaluate $c_0^\varepsilon$ in the expression (\ref{eq:AnablaG1ED}). From (\ref{eq:Fje}) and (\ref{eq:Gje}), we have
\begin{gather*}
\lim_{\delta\to \infty}\frac{F_j^\varepsilon(z)-G_j^\varepsilon(z)}{z_1z_2\cdots z_{j-1}z_j^{s+n-j}}
\bigg|_{z=\xi^{\lambda_0+1,(n)}_{\lambda}}\! =(-1)^{n+s-j}\big(q^\alpha b_1b_2\cdots b_s t^{n-j}-a_1^{-1}a_2^{-1}\cdots a_{s}^{-1}t^{-(n-j)}\big)
\end{gather*}
if $1\le j\le \lambda_0+1$. Using (\ref{eq:Aphi-3}), (\ref{eq:Gj=Fj=0}), (\ref{eq:E=0-1}) and the above equation, we have
\begin{gather}
 \lim_{\delta\to \infty}\frac{{\cal A}\phi(z)}{z_1^{s+n-1}z_2^{s+n-2}\cdots z_{\lambda_0+1}^{s+n-\lambda_0-1}}\bigg|_{z=\xi^{\lambda_0+1,(n)}_{\lambda}}\nonumber\\
\quad{} =(n-1)!\sum_{j=1}^{\lambda_0+1}(-1)^{j}\lim_{\delta\to \infty}\bigg(\frac{F_j^\varepsilon(z)-G_j^\varepsilon(z)}{z_1z_2\cdots z_{j-1}z_j^{s+n-j}}\nonumber\\
 \qquad{} \times
\frac{\mathsf{E}^{\lambda_0,(n-1)}_{\lambda}(z_{\widehat{j}})\Delta^{(n-1)}(z_{\widehat{j}})}
{z_1^{s+n-2}z_2^{s+n-3}\cdots z_{j-1}^{s+n-j}z_{j+1}^{s+n-j-1}\cdots z_{\lambda_0+1}^{s+n-\lambda_0-1}}\bigg) \bigg|_{z=\xi^{\lambda_0+1,(n)}_{\lambda}} \nonumber\\
\quad{} =(n-1)!(-1)^{n+s}\sum_{j=1}^{\lambda_0+1} \big(q^\alpha b_1b_2\cdots b_s t^{n-j}-a_1^{-1}a_2^{-1}\cdots a_{s}^{-1}t^{-(n-j)}\big)\nonumber\\
 \qquad{} \times \Delta^{(n-\lambda_0-1)}\big(\xi^{0,(n-\lambda_0-1)}_{\lambda}\big)\nonumber\\
\quad{}=
(n-1)!(-1)^{n+s-1}\frac{\big(1-q^\alpha a_1a_2\cdots a_{s}b_1b_2\cdots b_{s} t^{2n-2-\lambda_0}\big)\big(1-t^{\lambda_0+1}\big)}{a_1a_2\cdots a_{s}t^{n-1}(1-t)}\nonumber\\
\qquad{} \times\Delta^{(n-\lambda_0-1)}\big(\xi^{0,(n-\lambda_0-1)}_{\lambda}\big).\label{eq:Aphi-6}
\end{gather}
Comparing (\ref{eq:Aphi-2}) with (\ref{eq:Aphi-6}), we therefore obtain the explicit expression of $c_0^\varepsilon=c_{(\lambda_0+1,\lambda_1,\ldots,\lambda_s)}^\varepsilon$ as given by (\ref{eq:c0}).

\section[Proofs for the Lagrange interpolation polynomials of type $A$]{Proofs for the Lagrange interpolation polynomials of type $\boldsymbol{A}$}\label{SectionB}

In this section we provide proofs for the propositions in Section \ref{Section4}.

\subsection{Construction of the interpolation polynomials}
Let $\mathsf{H}_{s,n}^z$ be the $\mathbb{C}$-linear space specified by (\ref{eq:Hsn-z}). For $x=(x_1,\ldots,x_s)\in (\mathbb{C}^*)^s$ and $\mu=(\mu_1,\ldots,\mu_s)\in \mathbb{N}^s$ with
\begin{gather*}
|\mu|=\mu_1+\cdots+\mu_s\le n,
\end{gather*}
we set
\begin{gather}\label{eq:x_mu-m}
x_\mu^{n-|\mu|} =\big(z_1,z_2,\ldots,z_{n-|\mu|}, \underbrace{x_1,x_1t,\ldots,x_1t^{\mu_1-1}\vphantom{\Big|}}_{\mu_1},\ldots,\underbrace{x_s,x_st,\ldots,x_st^{\mu_s-1}\vphantom{\Big|}}_{\mu_s}\big)
\in (\mathbb{C}^*)^{n},
\end{gather}
leaving the first $n-|\mu|$ variables unspecialized.

\begin{Theorem}\label{thm:tri-E}For generic $x=(x_1,\ldots,x_s)\in (\mathbb{C}^*)^s$ there exists a unique $\mathbb{C}$-basis $\{\mathsf{E}_\mu(x;z)\,|\,\mu\in \mathbb{N}^s, \, |\mu|\le n\}$ of the $\mathbb{C}$-linear space $\mathsf{H}_{s,n}^z$ satisfying
\begin{itemize}\itemsep=0pt
\item[{\rm (1)}] $|\lambda|= |\mu| \ \Longrightarrow \ \mathsf{E}_{\lambda}\big(x;x_{\mu} ^{n-|\mu|}\big)= \delta_{\lambda\mu}\prod_{i=1}^{n-|\lambda|}\prod_{j=1}^s \big(z_i-x_jt^{\lambda_j}\big)$,
\item[{\rm (2)}] $|\lambda|<|\mu| \ \Longrightarrow \ \mathsf{E}_{\lambda}\big(x;x_{\mu}^{n-|\mu|}\big)=0$.
\end{itemize}
\end{Theorem}

\begin{remark*} The functions $\mathsf{E}_\mu(x,z)$ for $\mu=(\mu_1,\ldots,\mu_s)\in \mathbb{N}^s$, where $|\mu|\le n$, in the above theorem are the same as $\mathsf{E}^{\mu_0}_{\mu_1,\ldots,\mu_s}(x,z)$, $(\mu_0,\mu_1,\ldots,\mu_s)\in Z_{s+1,n}$, in Theorem~\ref{thm:InterpolationP}, because the map $(\mu_0,\mu_1,\ldots,\mu_s)\mapsto (\mu_1,\ldots,\mu_s)$ defines a bijection $Z_{s+1,n}\to \{\mu\in \mathbb{N}^s\,|\, |\mu|\le n\}$. Throughout this section, for simplicity we use the notation $\mathsf{E}_\mu(x,z)$ $(\mu\in \mathbb{N}^s,\, |\mu|\le n)$ instead of $\mathsf{E}^{\mu_0}_{\mu_1,\ldots,\mu_s}(x,z)$, $\mu\in Z_{s+1,n}$.
\end{remark*}

We provide a proof of this theorem in this subsection.

\begin{Example}The case $n=1$. For $x=(x_1,\ldots,x_s)$ the symbols $x_\mu^{n-|\mu|}$ of the case $n=1$ are written as
\begin{gather*}
x_{\bf 0}^1=z\in \mathbb{C}^*\qquad\mbox{and}\qquad x_{\epsilon_i}^0=x_i\in \mathbb{C}^* \qquad \text{for}\quad i=1,\ldots,s,
\end{gather*}
where ${\bf 0}=(0,\ldots,0)\in \mathbb{N}^s$ and $\epsilon_i=(0,\ldots,0,\overset{\text{\tiny$i$}}{\stackrel{\text{\tiny$\smile$}}{1}},0,\ldots,0)\in \mathbb{N}^s$. Then it immediately follows that the functions
\begin{gather}\label{eq:B-E(x;z)-n=1}
\mathsf{E}_{\bf 0}(x;z)=\prod_{i=1}^s(z-x_i) \qquad\mbox{and}\qquad \mathsf{E}_{\epsilon_i}(x;z) =\prod_{\substack{1\le j\le s\\ j\ne i}}\frac{z-x_j}{x_i-x_j}
\end{gather}
satisfy the conditions of Theorem \ref{thm:tri-E}.
\end{Example}

In the same way as (\ref{eq:Hsn-z}) we define
\begin{gather*}
\mathsf{H}_{n,s}^w=\big\{f(w)\in \mathbb{C}[w_1,\ldots,w_s]^{\frak{S}_s}\,|\,\deg_{w_i}f(w)\le n \ \text{for} \ i=1,\ldots,s\big\},
\end{gather*}
which satisfies $\dim_\mathbb{C}\mathsf{H}_{n,s}^w={n+s\choose s}$. For $x=(x_1,\ldots,x_s)\in (\mathbb{C}^*)^s$, $w=(w_1,\ldots,w_s)\in (\mathbb{C}^*)^s$ and $\mu=(\mu_1,\ldots,\mu_s)\in \mathbb{N}^s$ satisfying $|\mu|\le n$, we set
\begin{gather}\label{eq:B-F(x;w)}
\mathsf{F}_\mu(x;w)=\prod_{i=1}^s\prod_{j=1}^s\prod_{k=1}^{\mu_i}\big(x_it^{k-1}-w_j\big),
\end{gather}
whose degree with respect to $w$ is $|\mu|$, where $|\mu|\le n$. By definition we have $\mathsf{F}_\mu(x;w)\in \mathsf{H}_{n,s}^w$. For example, we have
\begin{gather} \label{eq:B-F(x;w)-n=1}
\mathsf{F}_{\bf 0}(x;w)=1\qquad\mbox{and}\qquad \mathsf{F}_{\epsilon_i}(x;w)=\prod_{j=1}^s (x_i-w_j).
\end{gather}
By definition we have the relation
\begin{gather}\label{eq:B-FF=F}
\mathsf{F}_\mu(x;w)\mathsf{F}_\nu\big(xt^{\mu};w\big)=\mathsf{F}_{\mu+\nu}(x;w),
\end{gather}
where $xt^\mu$ is given by \eqref{eq:xtmu}. This relation is the fundamental tool for studying $\mathsf{E}_\lambda(x;z)$ in the succeeding arguments.

\begin{Lemma}\label{lem:B-triangular-F}For each $\mu,\nu\in \mathbb{N}^s$, $\mathsf{F}_\mu(x;xt^\nu)=0$ unless $\mu_i\le \nu_i$ for $i=1,\ldots,s$. In particular, $\mathsf{F}_\mu(x;xt^\nu)=0$ for $\mu\succ \nu$. Moreover, if $x\in (\mathbb{C}^*)^s$ is generic, then $\mathsf{F}_\nu(x;xt^\nu)\ne 0$ for all $\nu\in\mathbb{N}^s$.
\end{Lemma}

This lemma implies that the matrix $\mathsf{F}=\big(\mathsf{F}_\mu(x;xt^\nu)\big)_{\!\mu,\nu\in\mathbb{N}^s\atop\!\! |\mu|,|\nu|\le n}$ is upper triangular, and also invertible if $x\in\mathbb{C}^{s}$ is generic.

\begin{proof} By definition, for $\mu,\nu\in \mathbb{N}^s$ we have $\mathsf{F}_\mu(x;xt^\nu)=\prod_{i=1}^s \prod_{j=1}^s \prod_{k=1}^{\mu_i} \big(x_it^{k-1}-x_jt^{\nu_j}\big)$. Thus if there exists $i\in \{1,\ldots,s\}$ such that $\nu_i <\mu_i$, then $\mathsf{F}_\mu(x;xt^\nu)=0$. If $\nu\prec\mu$, then $\nu_i<\mu_i$ for some $i\in \{1,2,\ldots,s\}$ by definition, and hence we obtain $\mathsf{F}_\mu(x;xt^\nu)=0$ if $\nu\prec\mu$. Moreover, from~\eqref{eq:B-F(x;w)}, we obtain
\begin{gather*}
\mathsf{F}_\nu(x;xt^\nu)=\prod_{i=1}^s x_i^{\nu_i}t^{\nu_i\choose 2}(t;t)_{\nu_i}\prod_{\substack{1\le j\le s\\ j\ne i}}\prod_{k=1}^{\nu_i}\big(x_it^{k-1}-x_jt^{\nu_j}\big)\ne0,
\end{gather*}
if we impose an appropriate genericity condition on $x\in (\mathbb{C}^*)^s$.
\end{proof}

\begin{Lemma}[duality] \label{lem:B-duality} For $z=(z_1,\ldots,z_n)\in (\mathbb{C}^*)^{n}$, $w=(w_1,\ldots,w_s)\in (\mathbb{C}^*)^{s}$ let $\mathsf{\Psi}(z;w)$ be the function defined by
\begin{gather}\label{eq:B-def-Psi}
\mathsf{\Psi}(z;w)= \prod_{i=1}^n \prod_{j=1}^s (z_i-w_j).
\end{gather}
Then for $x=(x_1,\ldots,x_s)\in (\mathbb{C}^*)^{s}$ the function $\mathsf{\Psi}(z;w)$ expands in terms of $\mathsf{F}_\mu(x;w)$, $|\mu|\le n$, as
\begin{gather}\label{eq:B-duality}
\mathsf{\Psi}(z;w)= \sum_{\substack{\mu\in \mathbb{N}^s\\[1pt] |\mu|\le n}}\mathsf{E}_{\mu}(x;z)\mathsf{F}_\mu(x;w),
\end{gather}
where $\mathsf{E}_{\mu}(x;z)$, $|\mu|\le n$, are the coefficients independent of $w$, and satisfy $\mathsf{E}_{\mu}(x;z)\in \mathsf{H}_{s,n}^z$, $|\mu|\le n$, as functions of~$z$.
\end{Lemma}

\begin{Example} The case $n=1$. From (\ref{eq:B-E(x;z)-n=1}) and (\ref{eq:B-F(x;w)-n=1}), the identity (\ref{eq:B-duality}) follows from
\begin{gather}
\prod_{j=1}^s (z-w_j) =\prod_{i=1}^s(z-x_i) +\sum_{i=1}^s\left(\prod_{j=1}^s (x_i-w_j)\right)\prod_{\substack{1\le j\le s\\ j\ne i}}\frac{z-x_j}{x_i-x_j}\nonumber\\
\hphantom{\prod_{j=1}^s (z-w_j)}{} =\mathsf{E}_{\bf 0}(x;z)\mathsf{F}_{\bf 0}(x;w) +\sum_{i=1}^s\mathsf{E}_{\epsilon_i}(x;z)\mathsf{F}_{\epsilon_i}(x;w).\label{eq:B-duality-n=1}
\end{gather}
\end{Example}

\begin{proof}From Lemma \ref{lem:B-triangular-F} the set $\{\mathsf{F}_\mu(x;w)\,|\, \mu\in \mathbb{N}^s, |\mu|\le n\}\subset \mathsf{H}_{n,s}^w$ is linearly independent and $|\{\mathsf{F}_\mu(x;w)\,|\, \mu\in \mathbb{N}^s, |\mu|\le n\}|={n+s\choose n}$. This indicates that $\{\mathsf{F}_\mu(x;w)\,|\, \mu\in \mathbb{N}^s, |\mu|\le n\}$ is a~basis of the $\mathbb{C}$ linear space $\mathsf{H}_{n,s}^w$. Since $\mathsf{\Psi}(z;w)\in \mathsf{H}_{n,s}^w$, the function $\mathsf{\Psi}(z;w)$ of $w$ is expanded in terms of $\mathsf{F}_\mu(x;w)$, $\mu\in \mathbb{N}^s$, $|\mu|\le n$, as the form~(\ref{eq:B-duality}).

Next we prove $\mathsf{E}_{\lambda}(x;z)\in \mathsf{H}_{s,n}^z$. From (\ref{eq:B-duality}) we have
\begin{gather}\label{eq:B-duality-special}
\mathsf{\Psi}(z;xt^\nu)=\sum_{\substack{\mu\in \mathbb{N}^s\\[1pt] |\mu|\le n}}\mathsf{E}_{\mu}(x;z)\mathsf{F}_\mu\big(x;xt^\nu\big).
\end{gather}
We denote by $\mathsf{G}=\big(\mathsf{G}_{\mu\nu}(x)\big)_{\!\mu,\nu\in\mathbb{N}^s\atop\! |\mu|,|\nu|\le n}$ the inverse matrix of $\mathsf{F}=\big(\mathsf{F}_\mu(x;xt^\nu)\big)_{\!\mu,\nu\in\mathbb{N}^s\atop\! |\mu|,|\nu|\le n}$. Then, by Lemma~\ref{lem:B-triangular-F}, (\ref{eq:B-duality-special}) can be rewritten as
\begin{gather*}
\mathsf{E}_{\lambda}(x;z)=\sum_{\substack{\nu\in \mathbb{N}^s\\[1pt] |\nu|\le n}}\mathsf{\Psi}(z;xt^\nu)\mathsf{G}_{\nu\lambda}(x),\qquad \text{where} \quad |\lambda|\le n.
\end{gather*}
Since $\mathsf{\Psi}(z;xt^\nu)\in \mathsf{H}_{s,n}^z$, we obtain $\mathsf{E}_{\lambda}(x;z)\in \mathsf{H}_{s,n}^z$.
\end{proof}

\begin{Lemma}\label{lem:B-E_0(x;z)}For $z=(z_1,\ldots,z_n)\in (\mathbb{C}^*)^{n}$, the explicit form of $\mathsf{E}_{\bf 0}(x;z)$ is given by
\begin{gather}\label{eq:B-E_0(x;z)}
\mathsf{E}_{\bf 0}(x;z)=\prod_{i=1}^n\prod_{j=1}^s(z_i-x_j).
\end{gather}
\end{Lemma}

\begin{proof}From the expansion (\ref{eq:B-duality}) of the case $w=x$, we have
\begin{gather*}
\mathsf{\Psi}(z;x)=\sum_{\substack{\mu\in \mathbb{N}^s\\[1pt] |\mu|\le n}}\mathsf{E}_{\mu}(x;z)\mathsf{F}_\mu(x;x).
\end{gather*}
Since we have $\mathsf{F}_\mu(x;x)=\delta_{\mu{\bf 0}}$ by definition, we obtain $\mathsf{E}_{\bf 0}(x;z)=\mathsf{\Psi}(z;x)$, which coincides with~(\ref{eq:B-E_0(x;z)}).
\end{proof}

\begin{Lemma}\label{lem:tri-E2} For $z=(z_1,\ldots,z_n)\in (\mathbb{C}^*)^{n}$, the functions $\mathsf{E}_\mu(x;z)$, $|\mu|\le n$, defined by the relation \eqref{eq:B-duality} satisfy the following:
\begin{itemize}\itemsep=0pt
\item[{\rm (1)}] $|\lambda|= |\mu| \ \Longrightarrow\ \mathsf{E}_{\lambda}\big(x;x_{\mu}^{n-|\mu|}\big)= \delta_{\lambda\mu}\prod_{i=1}^{n-|\lambda|}\prod_{j=1}^s \big(z_i-x_jt^{\lambda_j}\big)$,
\item[{\rm (2)}] $|\lambda|<|\mu| \ \Longrightarrow\ \mathsf{E}_{\lambda}\big(x;x_{\mu}^{n-|\mu|}\big)=0$.
\end{itemize}
As a consequence of the above, the set $\{\mathsf{E}_\mu(x;z)\,|\,|\mu|\le n\}\subset \mathsf{H}_{s,n}^z$ is linearly independent, i.e., $\{\mathsf{E}_\mu(x;z)\,|\,|\mu|\le n\}$ is a basis of $\mathsf{H}_{s,n}^z$.
\end{Lemma}

\begin{proof}
We calculate $\mathsf{\Psi}(z;w)$ of the case $z=x_{\mu}^{n-|\mu|}$ in two ways. Using the identity (\ref{eq:B-duality}) partially for $\mathsf{\Psi}\big(x_{\mu}^{n-|\mu|};w\big)$, we have the expansion
\begin{gather}\label{eq:B-duality-1}
\mathsf{\Psi}(x_{\mu}^{n-|\mu|};w)=\mathsf{F}_\mu(x;w)\prod_{i=1}^{n-|\mu|}\prod_{j=1}^s(z_i-w_j)
=\mathsf{F}_\mu(x;w)\sum_{\substack{\nu\in \mathbb{N}^{s}\\[1pt] |\nu|\le n-|\mu|}}\mathsf{E}_\nu(xt^\mu;z')\mathsf{F}_\nu(xt^\mu;w),
\end{gather}
where $z'=(z_1,\ldots,z_{n-|\mu|})\in (\mathbb{C}^*)^{n-|\mu|}$. By property (\ref{eq:B-FF=F}) of $\mathsf{F}_\mu(x;w)$, the expression~(\ref{eq:B-duality-1}) may be rewritten as
\begin{gather}
\mathsf{\Psi}(x_{\mu}^{n-|\mu|};w)=\sum_{\substack{\nu\in \mathbb{N}^{s}\\[1pt] |\nu+\mu|\le n}}\mathsf{E}_\nu(xt^\mu;z')\mathsf{F}_{\nu+\mu}(x;w)=\sum_{\substack{\lambda-\mu\in\mathbb{N}^{s}\\[1pt] |\lambda|\le n}}
\mathsf{E}_{\lambda-\mu}(xt^\mu;z')\mathsf{F}_{\lambda}(x;w)\nonumber\\
\hphantom{\mathsf{\Psi}(x_{\mu}^{n-|\mu|};w)}{} =\sum_{\substack{\lambda\in\mathbb{N}^{s}\\[1pt] |\lambda|\le n}}\mathsf{E}_{\lambda-\mu}(xt^\mu;z')\mathsf{F}_{\lambda}(x;w),\label{eq:B-duality-2}
\end{gather}
where the coefficient $\mathsf{E}_{\lambda-\mu}(xt^\mu;z')$ in the last line is regarded as 0 if $\lambda-\mu\not\in\mathbb{N}^{s}$. On the other hand, from (\ref{eq:B-duality}) with $z=x_{\mu}^{n-|\mu|}$ we have
\begin{gather}\label{eq:B-duality-3}
\mathsf{\Psi}(x_{\mu}^{n-|\mu|};w)=\sum_{\substack{\lambda\in\mathbb{N}^{s}\\[1pt] |\lambda|\le n} }\mathsf{E}_{\lambda}\big(x;x_{\mu}^{n-|\mu|}\big)\mathsf{F}_\lambda(x;w).
\end{gather}
Since $\mathsf{F}_\lambda(x;w)$ is a basis of $\mathsf{H}_{n,s}^w$, equating coefficients of $\mathsf{F}_\lambda(x;w)$ on (\ref{eq:B-duality-2}) and (\ref{eq:B-duality-3}), we have
\begin{gather}\label{eq:B-E(x;xtmu)}
\mathsf{E}_{\lambda}\big(x;x_{\mu}^{n-|\mu|}\big)=\mathsf{E}_{\lambda-\mu}\big(xt^\mu;z_1,\ldots,z_{n-|\lambda|}\big),
\end{gather}
where the right-hand side of (\ref{eq:B-E(x;xtmu)}) vanishes if $\lambda-\mu\not\in\mathbb{N}^{s}$.

Now we prove the vanishing properties (1), (2) in Lemma \ref{lem:tri-E2}. (1)~We first consider the case $|\lambda|=|\mu|$ and $\lambda\ne \mu$. Then we have $\lambda-\mu\not\in \mathbb{N}^s$. From (\ref{eq:B-E(x;xtmu)}) we therefore obtain $\mathsf{E}_{\lambda}\big(x;x_{\mu}^{n-|\mu|}\big)=0$. Next we suppose the case $|\lambda|=|\mu|$ and $\lambda=\mu$. Because of Lemma~\ref{lem:B-E_0(x;z)}, from~(\ref{eq:B-E(x;xtmu)}) we obtain $\mathsf{E}_{\lambda}\big(x;x_{\lambda}^{n-|\lambda|}\big)=\mathsf{E}_{\bf 0}\big(xt^\lambda;z'\big)=\prod_{i=1}^{n-|\lambda|}\prod_{j=1}^s\big(z_i-x_jt^{\lambda_j}\big)$. (2)~Suppose the case $|\lambda|<|\mu|$. Then we have $\lambda-\mu\not\in \mathbb{N}^s$. From (\ref{eq:B-E(x;xtmu)}) we obtain $\mathsf{E}_{\lambda}\big(x;x_{\mu}^{n-|\mu|}\big)=0$.
\end{proof}

\subsection[Explicit expression for $\mathsf{E}_{\lambda}(x;z)$]{Explicit expression for $\boldsymbol{\mathsf{E}_{\lambda}(x;z)}$}

\begin{Lemma}\label{lem:B-E-explicit-1}For $z\in (\mathbb{C}^*)^{n}$, $x\in (\mathbb{C}^*)^{s}$, the functions $\mathsf{E}_\lambda(x;z)$ for $\lambda\in \mathbb{N}^s$, where $|\lambda|\le n$, may be expressed as
\begin{gather}
\mathsf{E}_\lambda(x;z)= \sum_{\substack{{\substack{(i_1,\ldots,i_n) \in \{0,1,\ldots,s\}^n}}\\[1pt] \epsilon_{i_1}+\cdots+\epsilon_{i_n}=\lambda}}
\mathsf{E}_{\epsilon_{i_1}}(x;z_1)\mathsf{E}_{\epsilon_{i_2}}\big(xt^{\epsilon_{i_1}};z_2\big) \mathsf{E}_{\epsilon_{i_3}}\big(xt^{\epsilon_{i_1}+\epsilon_{i_2}};z_3\big)\cdots\nonumber\\
\hphantom{\mathsf{E}_\lambda(x;z)= \sum_{\substack{{\substack{(i_1,\ldots,i_n) \in \{0,1,\ldots,s\}^n,}}\\[1pt] \epsilon_{i_1}+\cdots+\epsilon_{i_n}=\lambda}}}{}\times \mathsf{E}_{\epsilon_{i_r}}\big(xt^{\epsilon_{i_1}+\cdots+\epsilon_{i_{n-1}}};z_n\big),\label{eq:B-E-explicit-1}
\end{gather}
where the symbol $\epsilon_0$ denotes $\epsilon_0={\bf 0}\in \mathbb{N}^s$ and $\mathsf{E}_{\epsilon_{i_k}}(x;z_k)$, $k=1,\ldots,n$, are given by \eqref{eq:B-E(x;z)-n=1}. In other words, if $|\lambda|=r$, $r=0,1,\ldots,n$, then $\mathsf{E}_\lambda(x;z)$ are rewritten as
\begin{gather*}
\mathsf{E}_\lambda(x;z) =\sum_{1\le i_1<\cdots <i_r\le n} \sum_{\substack{{(j_1,\ldots,j_r)\in \{1,\ldots,s\}^r}\\ \epsilon_{j_1}+\cdots+\epsilon_{j_r}=\lambda}}
\bigg[\mathsf{E}_{\bf 0}(x;z_1)\cdots \mathsf{E}_{\bf 0}(x;z_{i_1-1})\mathsf{E}_{\epsilon_{j_1}}(x;z_{i_1})\\
\hphantom{\mathsf{E}_\lambda(x;z) =}{} \times \mathsf{E}_{\bf 0}\big(xt^{\epsilon_{j_1}};z_{i_1+1}\big)\cdots \mathsf{E}_{\bf 0}\big(xt^{\epsilon_{j_1}};z_{i_2-1}\big)\mathsf{E}_{\epsilon_{j_2}}\big(xt^{\epsilon_{j_1}};z_{i_2}\big)\cdots\\
\hphantom{\mathsf{E}_\lambda(x;z) =}{} \times \mathsf{E}_{\bf 0}\big(xt^{\epsilon_{j_1}+\cdots+\epsilon_{j_{r-1}}};z_{i_{r-1}+1}\big)\cdots
\mathsf{E}_{\bf 0}\big(xt^{\epsilon_{j_1}+\cdots+\epsilon_{j_{r-1}}};z_{i_r-1}\big)\mathsf{E}_{\epsilon_{j_r}}\big(xt^{\epsilon_{j_1}+\cdots+\epsilon_{j_{r-1}}};z_{i_r}\big)\\
\hphantom{\mathsf{E}_\lambda(x;z) =}{} \times \mathsf{E}_{\bf 0}\big(xt^{\epsilon_{j_1}+\cdots+\epsilon_{j_{r}}};z_{i_{r}+1}\big)\cdots \mathsf{E}_{\bf 0}\big(xt^{\epsilon_{j_1}+\cdots+\epsilon_{j_{r}}};z_n\big)\bigg].
\end{gather*}
In particular, the leading term of $\mathsf{E}_\lambda(x;z)$ with $|\lambda|=r$ as a symmetric polynomial in $z$ is equal to $m_{((s-1)^rs^{n-r})}(z)$ up to a constant.
\end{Lemma}

Before proving Lemma \ref{lem:B-E-explicit-1}, we present several special cases below.

\begin{Example} If $|\lambda|=0$, i.e., $\lambda={\bf 0}$, then $\mathsf{E}_{\bf 0}(x;z)=\prod_{i=1}^n\mathsf{E}_{\bf 0}(x;z_i)=\prod_{i=1}^n\prod_{j=1}^s(z_i-x_j)$, which we already saw in~(\ref{eq:B-E_0(x;z)}).
\end{Example}

\begin{Example}\label{ex:B-E-explicit-1-2} If $|\lambda|=n$, then
\begin{gather*}
\mathsf{E}_\lambda(x;z)=\!\!\sum_{\substack{{(i_1,\ldots,i_n)\in \{1,\ldots,s\}^n}\\ \epsilon_{i_1}+\cdots+\epsilon_{i_n}=\lambda}}\!\!\!\!
\mathsf{E}_{\epsilon_{i_1}}(x;z_1)\mathsf{E}_{\epsilon_{i_2}}(xt^{\epsilon_{i_1}};z_2)\mathsf{E}_{\epsilon_{i_3}}(xt^{\epsilon_{i_1}+\epsilon_{i_2}};z_3)\cdots
\mathsf{E}_{\epsilon_{i_r}}(xt^{\epsilon_{i_1}+\cdots+\epsilon_{i_{n-1}}};z_n)\\
\hphantom{\mathsf{E}_\lambda(x;z)}{} =\sum_{\substack{K_1\sqcup\cdots\sqcup K_s \\ =\{1,2,\ldots,n\}}}\prod_{i=1}^s\prod_{k\in K_i}\prod_{\substack{1\le j\le s\\ j\ne i}}\frac{z_{k}-x_jt^{\lambda_j^{(k-1)}}}{x_it^{\lambda_i^{(k-1)}}-x_jt^{\lambda_j^{(k-1)}}},
\end{gather*}
where $\lambda_i^{(k)}=|K_i\cap\{1,2,\ldots,k\}|$, and the summation is taken over all partitions $K_1\sqcup\cdots\sqcup K_s=\{1,2,\ldots,n\}$ such that $|K_i|=\lambda_i$, $i=1,2,\ldots,s$.
\end{Example}

\begin{Example}[shifted symmetric polynomials] Let $s=1$. Then we have
\begin{gather*}
\mathsf{E}_{n-r}(x_1;z)=\sum_{1\le i_1<\cdots <i_r\le n} \prod_{k=1}^r \big(z_{i_k}-x_1t^{i_k-k}\big) \qquad \text{for} \quad r=0,1,\ldots,n,
\end{gather*}
which coincide with {\em the Knop--Sahi shifted symmetric polynomials} attached to partitions of a~single column, see \cite[p.~476, Proposition~3.1]{KS96}. See also \cite[Appendix]{IF15} for an application of these polynomials to the $q$-Selberg integral.
\end{Example}

\begin{proof}[Proof of Lemma \ref{lem:B-E-explicit-1}] Since we have $\prod_{j=1}^s (z_1-w_j)=\sum_{i=0}^s \mathsf{E}_{\epsilon_i}(x;z_1)\mathsf{F}_{\epsilon_i}(x;w)$ from~(\ref{eq:B-duality-n=1}), the function $\mathsf{\Psi}(z;w)$ may be expanded as
\begin{gather}
\mathsf{\Psi}(z;w) =\mathsf{\Psi}(z_1;w)\mathsf{\Psi}(z_2;w)\cdots \mathsf{\Psi}(z_n;w) \nonumber\\
\hphantom{\mathsf{\Psi}(z;w)}{} = \sum_{i_1=0}^s\mathsf{E}_{\epsilon_{i_1}}(x;z_1)\mathsf{F}_{\epsilon_{i_1}}(x;w)\sum_{i_2=0}^s \mathsf{E}_{\epsilon_{i_2}}\big(xt^{\epsilon_{i_1}};z_2\big)\mathsf{F}_{\epsilon_{i_2}}\big(xt^{\epsilon_{i_1}};w\big)\cdots
\nonumber\\
\hphantom{\mathsf{\Psi}(z;w)=}{} \times \sum_{i_n=0}^s\mathsf{E}_{\epsilon_{i_n}}\big(xt^{\epsilon_{i_1}+\cdots+\epsilon_{i_{n-1}}};z_n\big)
\mathsf{F}_{\epsilon_{i_n}}\big(xt^{\epsilon_{i_1}+\cdots+\epsilon_{i_{n-1}}};w\big).\label{eq:B-duality-0}
\end{gather}
From (\ref{eq:B-FF=F}), if $\lambda=\epsilon_{i_1}+\cdots+\epsilon_{i_n}$, we have
\begin{gather*}
\mathsf{F}_{\epsilon_{i_1}}(x;w)\mathsf{F}_{\epsilon_{i_2}}\big(xt^{\epsilon_{i_1}};w\big)\cdots \mathsf{F}_{\epsilon_{i_n}}\big(xt^{\epsilon_{i_1}+\cdots+\epsilon_{i_{n-1}}};w\big)=\mathsf{F}_{\lambda}(x;w).
\end{gather*}
Then, comparing the coefficient of $\mathsf{F}_{\lambda}(x;w)$ in (\ref{eq:B-duality-0}) with those in (\ref{eq:B-duality}), we obtain (\ref{eq:B-E-explicit-1}).
\end{proof}

\begin{remark*}Similar to (\ref{eq:B-duality-0}), we have the recursion formula
\begin{gather}\label{eq:B-recursion}
\mathsf{E}_\lambda(x;z_1,\ldots,z_n)=\sum_{k=1}^s\mathsf{E}_{\lambda-\epsilon_k}(x;z_1,\ldots,z_{n-1})\mathsf{E}_{\epsilon_k}\big(xt^{\lambda-\epsilon_k};z_1\big)
\end{gather}
from the identity (\ref{eq:B-FF=F}) of $\mathsf{F}_{\lambda}(x;w)$.
\end{remark*}

\begin{Lemma}\label{lem:c of l.t.}For $z\in (\mathbb{C}^*)^{n}$, $x\in (\mathbb{C}^*)^{s}$, the functions $\mathsf{E}_\lambda(x;z)$ for $\lambda\in \mathbb{N}^s$, where $|\lambda|\le n$, can be expanded as
\begin{gather}\label{eq:B-leading term}
\mathsf{E}_\lambda(x;z)=C_\lambda(x) m_{(s^{n-|\lambda|}(s-1)^{|\lambda|})}(z)+\cdots,
\end{gather}
where the coefficient $C_{\lambda}(x)$ of the leading term is given by
\begin{gather}\label{eq:c of l.t.}
C_{\lambda}(x)=\frac{(t;t)_{|\lambda|}}{\prod_{i=1}^s (t;t)_{\lambda_i}}\prod_{i=1}^s\prod_{\substack{1\le j\le s\\ j\ne i}}\prod_{k=1}^{\lambda_i}\frac{1}{x_it^{k-1}-x_jt^{\lambda_j}}.
\end{gather}
\end{Lemma}

\begin{proof}Since we have (\ref{eq:B-leading term}) immediately from Lemma~\ref{lem:B-E-explicit-1}, we prove (\ref{eq:c of l.t.}) by induction on $r=|\lambda|$. We denote by $C_\lambda^{(r)}(x)$ the coefficient $C_{\lambda}(x)$ in~(\ref{eq:B-leading term}). When $r=1$, we immediately find
\begin{gather*}
C_{\epsilon_i}^{(1)}(x)=\prod_{\substack{1\le j\le s\\ j\ne i}}\frac{1}{x_i-x_j} \qquad \text{for} \quad i=1,\ldots,s
\end{gather*}
from the explicit expression (\ref{eq:B-E(x;z)-n=1}) of $\mathsf{E}_{\epsilon_i}(x;z)$. From the recursion formula~(\ref{eq:B-recursion}) we have the relation
\begin{gather}\label{eq:B-relation C}
C_\lambda^{(r)}(x) =\sum_{k=1}^sC_{\lambda-\epsilon_k}^{(r-1)}(x) C_{\epsilon_k}^{(1)}\big(xt^{\lambda-\epsilon_k}\big).
\end{gather}
Since (\ref{eq:c of l.t.}) holds for $r=1$, it suffices to show that the right-hand side of (\ref{eq:c of l.t.}) satisfies the same recurrence formula as~(\ref{eq:B-relation C}). One can directly verify that the corresponding
formula for the right-hand side of~(\ref{eq:c of l.t.}) reduces to the identity
\begin{gather}\label{eq:B-identity-01}
1=\sum_{k=1}^s\frac{1-t^{\lambda_k}}{1-t^r}\prod_{\substack{1\le i\le s\\ i\ne k}}\prod_{l=1}^{\lambda_i}\frac{x_it^{l-1}-x_kt^{\lambda_k}}{x_it^{l-1}-x_kt^{\lambda_k-1}},
\end{gather}
which is equivalent to the identity \cite[p.~46, Lemma 1.51]{Mil85} with $(x_i,y_i)\to \big(t^{\lambda_i}, x_it^{\lambda_i}\big)$.
\end{proof}

\begin{remark*}The identity \eqref{eq:B-identity-01} is very well known and has played a very important role in the theory of multiple basic hypergeometric series. In the paper~\cite{Ros04}, Rosengren extended this identity to its elliptic form and applied it to his multiple elliptic hypergeometric series. It is really remarkable that, according to his paper, this identity in elliptic form already appeared in Tannery and Molk's book \cite[p.~34]{TM} (which was published in 1898) and also in Whittaker and Watson \cite[p.~451]{WW}. See his very detailed discussion about this identity in \cite{Ros04}.
\end{remark*}

\begin{Lemma}\label{lem:B-E=EE}For $z=(z_1,\ldots,z_n)$ and $x=(x_1,\ldots,x_s)$ we have
\begin{gather}\label{eq:B-E=EE}
\mathsf{E}_{(0,\mu_2,\ldots,\mu_s)}(x;z)= \frac{\mathsf{E}_0(x_1;z)\mathsf{E}_{(\mu_2,\ldots,\mu_s)}(x_{\widehat{1}};z)}{\prod\limits_{i=2}^s\prod\limits_{k=1}^{\mu_i}\big(x_it^{k-1}-x_1\big)},
\end{gather}
where $x_{\widehat{1}}=(x_2,\ldots,x_s)\in (\mathbb{C}^*)^{s-1}$.
\end{Lemma}

\begin{proof}From \eqref{eq:B-duality} of Lemma \ref{lem:B-duality} with substitution $w_s=x_1$, we have
\begin{gather*}
\mathsf{\Psi}(z,w)\Big|_{w_s=x_1} =\sum_{\substack{\mu\in \mathbb{N}^s\\[1pt] |\mu|\le n}}\mathsf{E}_\mu(x;z)\mathsf{F}_\mu(x;w)\Big|_{w_s=x_1}.
\end{gather*}
Here, from the definition \eqref{eq:B-F(x;w)} of $\mathsf{F}_{\mu}(x;w)$, we have
\begin{gather*}
\mathsf{F}_\mu(x;w)\Big|_{w_s=x_1}=
\begin{cases}
0 & \text{if} \ \mu_1>0, \vspace{1mm}\\
\mathsf{F}_{(\mu_2,\ldots,\mu_s)}(x_{\widehat{1}};w_{\widehat{s}}) \prod\limits_{i=2}^s\prod\limits_{k=1}^{\mu_i}(x_it^{k-1}-x_1) &\text{if} \ \mu_1=0,
\end{cases}
\end{gather*}
where $w_{\widehat{s}}=(w_1,\ldots,w_{s-1})\in (\mathbb{C}^*)^{s-1}$. Thus we obtain
\begin{gather}
\mathsf{\Psi}(z,w)\Big|_{w_s=x_1}=\sum_{\substack{(\mu_2,\ldots,\mu_s)\in \mathbb{N}^{s-1}\\[1pt] \mu_2+\cdots+\mu_s\le n}}
\left[\mathsf{E}_{(0,\mu_2,\ldots,\mu_s)}(x;z)\prod_{i=2}^s\prod_{k=1}^{\mu_i}\big(x_it^{k-1}-x_1\big)\right]\nonumber\\
\hphantom{\mathsf{\Psi}(z,w)\Big|_{w_s=x_1}=}{} \times \mathsf{F}_{(\mu_2,\ldots,\mu_s)}(x_{\widehat{1}};w_{\widehat{s}}).\label{B-Psi(z,w)|-1}
\end{gather}
On the other hand, from \eqref{eq:B-def-Psi}, \eqref{eq:B-duality} and \eqref{eq:B-E_0(x;z)}, we also obtain
\begin{gather}
\mathsf{\Psi}(z,w)\Big|_{w_s=x_1} =\mathsf{\Psi}(z,w_{\widehat{s}})\prod_{i=1}^n(z_i-x_1)=\mathsf{E}_0(x_1;z)
\sum_{\mu\in \mathbb{N}^{s-1}\atop |\mu|\le n}\mathsf{E}_\mu(x_{\widehat{1}};z)\mathsf{F}_\mu(x_{\widehat{1}};w_{\widehat{s}})\nonumber\\
\hphantom{\mathsf{\Psi}(z,w)\Big|_{w_s=x_1}}{} =\sum_{\substack{(\mu_2,\ldots,\mu_s)\in \mathbb{N}^{s-1}\\[1pt] \mu_2+\cdots+\mu_s\le n}}
\big[\mathsf{E}_0(x_1;z)\mathsf{E}_{(\mu_2,\ldots,\mu_s)}(x_{\widehat{1}};z)\big]\mathsf{F}_{(\mu_2,\ldots,\mu_s)}(x_{\widehat{1}};w_{\widehat{s}}).\label{B-Psi(z,w)|-2}
\end{gather}
Comparing \eqref{B-Psi(z,w)|-1} with \eqref{B-Psi(z,w)|-2}, we therefore obtain \eqref{eq:B-E=EE}.
\end{proof}

\subsection*{Acknowledgements}
The authors would like to express their gratitude to the referees for providing them many useful suggestions. This work is supported by JSPS Kakenhi Grants (C)25400118, (B)15H03626 and (C)18K03339.

\pdfbookmark[1]{References}{ref}
\LastPageEnding

\end{document}